\documentclass[reqno,10pt]{amsart}
\usepackage{amssymb,amsmath,amsthm,amsfonts,color}

\usepackage{mathrsfs,dsfont, comment,mathscinet}

\usepackage{enumerate,esint}
\usepackage[left=2.8cm,right=2.8cm,top=2.8cm,bottom=2.8cm]{geometry}
\usepackage{mathtools}
\usepackage{graphicx}
\usepackage{bbm}
 \usepackage[abs]{overpic}
 \usepackage{tikz} 
\usepackage{xcolor} 
\usepackage{caption,subcaption,pgfplots} 

\usepackage{epstopdf}
\usepackage[colorlinks=true, pdfstartview=FitV, linkcolor=blue, 
            citecolor=blue, urlcolor=blue]{hyperref}

\allowdisplaybreaks
\usetikzlibrary{shapes,snakes,spy,calc,shapes.geometric} 

\numberwithin{equation}{section}
\theoremstyle{plain}
\newtheorem{theorem}{Theorem}[section]
\newtheorem{proposition}[theorem]{Proposition}
\newtheorem{lemma}[theorem]{Lemma}

\newtheorem{corollary}[theorem]{Corollary}
\theoremstyle{definition}

\newtheorem{remark}[theorem]{Remark}

\newcommand\R{\mathbb R}

\newcommand\M{\mathbb M}
\newcommand\N{\mathbb N}

\newcommand\ep{\varepsilon}
\newcommand\eps{\varepsilon}
\newcommand\curl{{\rm curl} \, }
\newcommand\dist{{\rm dist}}

\newcommand\EEE{\color{black}}
\newcommand\RRR{\color{black}}
\newcommand\RRRR{\color{black}}

\title [Two-well rigidity and sharp-interface limits for Solid-Solid phase transitions]
{Two-well rigidity and multidimensional sharp-interface limits for Solid-Solid phase transitions}
\author[E. Davoli] {Elisa Davoli} 
\address[Elisa Davoli]{Faculty of Mathematics, University of Vienna, 
Oskar-Morgenstern-Platz 1, A-1090 Vienna, Austria}
\email{elisa.davoli@univie.ac.at}
\author[M. Friedrich] {Manuel Friedrich} 
\address[Manuel Friedrich]{Institute for Computational and Applied Mathematics
  University of  M\"{u}nster, Einsteinstr.~62, D-48149 M\"{u}nster, Germany}
\email{manuel.friedrich@uni-muenster.de}
\subjclass[2010]{}

\begin{document} 
\vskip .2truecm
\begin{abstract}

We establish a quantitative rigidity estimate for two-well frame-indifferent nonlinear energies, in the case in which  the two wells have exactly one rank-one connection. Building upon this novel rigidity result, we then
 analyze solid-solid phase transitions in arbitrary space dimensions, under a suitable anisotropic penalization of second variations.   By means of $\Gamma$-convergence, we show that, as the size of transition layers tends to zero,  singularly perturbed two-well problems   approach   an effective sharp-interface model. The limiting energy is finite only for deformations which have the structure of a laminate. In this case, it  is proportional to the total length of the interfaces between the two phases. 
\end{abstract}
\maketitle

\section{Introduction}
Solid-solid phase transitions are often observed in nature, both in basic phenomena (e.g., change between different ice forms under high pressure, or transformation from graphite to diamond in carbon under very elevated temperature and pressure) as well as in advanced materials such as shape-memory alloys (see, e.g., \cite{bhattacharya.kohn, cheng}). 
In this paper we contribute to the theory of solid-solid phase transitions by presenting a novel quantitative two-well rigidity estimate  and its application to singularly perturbed two-well  problems.  In particular,  we extend the results about sharp-interface limits obtained   by {\sc S.~Conti} and {\sc B.~Schweizer} \cite{conti.schweizer, conti.schweizer2} in dimension two to the higher-dimensional framework and, as a byproduct, we provide  a simplified convergence proof in the two-dimensional setting.

Assume that $\Omega\subset \mathbb{R}^d$, $d\in \mathbb{N}$,  is a bounded Lipschitz domain, denoting the reference configuration of a nonlinearly elastic material undergoing a solid-solid phase transition between two phases $A,B\in \mathbb{M}^{d\times d}$. Its behavior is then classically encoded by means of a nonlinear elastic energy functional of the form

\begin{align}\label{eq: basic energy}
y\in H^1(\Omega;\mathbb{R}^d)\to \int_{\Omega}W(\nabla y)\,dx,
\end{align}
where $W:\mathbb{M}^{d\times d}\to [0,+\infty)$ is a nonlinear, frame-indifferent, elastic energy whose zero set has the following two-well structure
\begin{equation}
\label{eq:zset}
\{F\in \mathbb{M}^{d\times d}:\, W(F)=0\}=SO(d)A\cup SO(d)B,
\end{equation}
where $SO(d)$ denotes the set of proper rotations in $\mathbb{M}^{d\times d}$.  It is well known that, in the presence of rank-one connections between $A$ and $B$, low energy sequences for generic boundary value problems exhibit infinitely fine oscillations.

 In order to remedy  the  issue of unphysical, highly oscillatory behavior, second  order perturbations may be added to \eqref{eq: basic energy}. Then, macroscopic transitions between the two wells $SO(d)A$ and $SO(d)B$ can be  described  via the minimization of a \emph{diffuse interface model} of the form
\begin{equation}
\label{eq:sol-sol}
 y\in H^2(\Omega;\R^d) \to I_{\ep}(y):=\frac{1}{\ep^2}\int_{\Omega}W(\nabla y)\,dx+\ep^2\int_{\Omega}|\nabla^2 y|^2\,dx.
\end{equation}
In the formula above,   $\ep>0$  is a smallness parameter introducing a length scale into the problem. Roughly speaking, $\eps^2$ describes the width of the transition layers between different phases (see, e.g., \cite{ball.james,bhattacharya, capella.otto2, kohn.muller2,muller}), so that, as $\ep$ tends to zero, the behavior of $I_{\ep}$ approaches that of a \emph{sharp-interface model}.  (We prefer to use  the parameter $\eps$ with exponent $2$ in the above formula since this will have notational advantages in the following.) We remark that a number of different possible higher order regularizations is used in the literature, both of diffuse and sharp-interface type. They all have the common feature that they   can be interpreted as surface  energies  penalizing the transition between different energy wells. Although the above described continuum models are   only ``phenomenological'', they have remarkable success in  predicting microstructures and material patterns.

Singularly perturbed nonconvex functionals of the form
\begin{align}\label{eq: liq-liq}
G_{\ep}(u):=\frac{1}{\ep^2}\int_{\Omega}W(u)\,dx+\ep^2\int_{\Omega}|\nabla u|^2\,dx
\end{align}
have also been extensively studied in connection  with   the theory of minimal surfaces and the modeling of liquid-liquid phase transitions. Starting from the seminal works by {\sc L.~Modica}, {\sc S.~Mortola}, and {\sc M.~E.~Gurtin} \cite{gurtin, modica, modica.mortola}, a thorough analysis of energy functionals as in \eqref{eq: liq-liq} has been performed both in scalar \cite{bouchitte,  owen.sternberg} and in vectorial  settings \cite{barroso.fonseca, fonseca.tartar, sternberg,sternberg2}.  We also refer to  \cite{kohn.sternberg} for an analysis of local minimizers and to   \cite{ambrosio, baldo} for extensions to the situation in which $W$ has more than two wells. In particular, a limiting description of the functionals $G_{\ep}$ has been identified by $\Gamma$-convergence (see \cite{Braides:02, DalMaso:93} for an overview), and shown to be  proportional to  the length of the interfaces between the different phases.

The corresponding $\Gamma$-convergence analysis in the solid-solid setting, addressing the passage from a diffuse  to a sharp-interface model, has been open until the early 2000s until a breakthrough was achieved by {\sc S. Conti}, {\sc I. Fonseca}, and {\sc G. Leoni} in \cite{conti.fonseca.leoni}, in the case in which frame-indifference is neglected. In dimension two, the analysis was extended to the frame-indifferent linearized setting by {\sc S.~Conti} and {\sc B.~Schweizer} in \cite{conti.schweizer2} who also accomplished the characterization of the fully nonlinear framework \eqref{eq:sol-sol} for $d=2$ in the two subsequent papers \cite{conti.schweizer, conti.schweizer3}. Recently, some related microscopic models for two-dimensional martensitic transformations as well as their discrete-to-continuum limits have been analyzed in \cite{kytavsev-ruland-luckhaus, kytavsev-ruland-luckhaus2}.

As highlighted, e.g., in \cite{ruland-ARMA}, when studying solid-solid diffuse models having the structure in \eqref{eq:sol-sol}, two nonlinear phenomena need to be tackled simultaneously, namely  a \emph{material} nonlinearity due to the two-well structure of the energy, and a \emph{geometric} nonlinearity, generated by the $SO(d)$-frame-indifference of the model. This, together with the PDE constraint ``$\curl = 0$'' on the admissible fields entering the nonconvex densities $W$, renders  the analysis much more delicate in comparison with liquid-liquid counterparts  as in  \eqref{eq: liq-liq}.

A preliminary crucial observation concerning the material nonlinearity is the fact that the mathematical description of the model strongly depends on the presence or the absence of rank-one connections between the two phases $A$ and $B$ in \eqref{eq:zset}.  Indeed, sequences with uniformly bounded energy \eqref{eq:sol-sol} which converge to non-affine limiting configurations (i.e.,  exhibiting   phase transitions) only exist if $A$ and $B$ are rank-one connected.  (This is often called the \emph{Hadamard compatibility condition for layered deformations}, see \cite{ball.james}.) Admissible limiting deformations are necessarily piecewise affine and interfaces are planes (see Figure \ref{fig:limiting-def}).

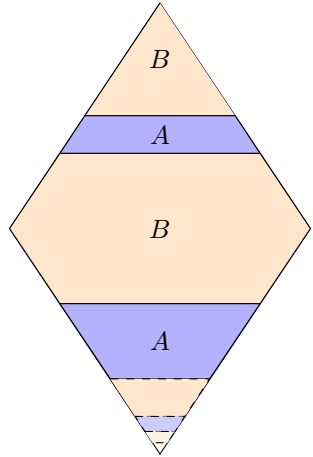
\begin{figure}[h]
\centering
\begin{tikzpicture}
\coordinate (A) at (0,3);
 \coordinate (B) at (2,0);
 \coordinate (C) at (-2,0);
 \coordinate (D) at (0,-3);
 \coordinate (E) at (1,1.5);
 \coordinate (F) at (1.33,1);
  \coordinate (G) at (1.33,-1);
  \coordinate (H) at (0.66,-2);
  \coordinate (I) at (0.33,-2.5);
  \coordinate (L) at (0.2,-2.7);
  \coordinate (M) at (0.1,-2.85);
   \coordinate (N) at (-1,1.5);
 \coordinate (O) at (-1.33,1);
  \coordinate (P) at (-1.33,-1);
  \coordinate (Q) at (-0.66,-2);
  \coordinate (R) at (-0.33,-2.5);
  \coordinate (S) at (-0.2,-2.7);
  \coordinate (T) at (-0.1,-2.85);
  \draw (A)--(B)--(D)--(C)--cycle;
  \draw (E)--(N);
  \draw (F)--(O);
  \draw (G)--(P);
  \draw[fill=orange!20] (A)--(N)--(E);
  \draw[fill=blue!30] (N)--(E)--(F)--(O)--cycle;
 \draw[fill=orange!20] (O)--(F)--(B)--(G)--(P)--(C)--cycle;
 \draw[fill=blue!30] (P)--(G)--(H)--(Q)--cycle;
  \draw[dashed,fill=orange!20] (Q)--(H)--(I)--(R);
   \draw[dashed,fill=blue!20] (R)--(I)--(L)--(S)--cycle;
 \draw[dashed,fill=orange!10] (S)--(L)--(M)--(T);
 \node at (0,0) {$B$};
 \node at (0,-1.5) {$A$};
  \node at (0,1.25) {$A$};
 \node at (0,2.25) {$B$};

\end{tikzpicture}
\caption{A limiting deformation whose gradient takes values in $\lbrace A, B\rbrace$, in the case in which $A$ and $B$ have exactly one rank-one connection.}
\label{fig:limiting-def}
\end{figure}

 Thus, the limiting sharp-interface problem is very rigid, and  hence  the analysis  seems to be simplified compared to liquid-liquid phase transitions where minimal surfaces have to be considered. However, it turns out that the  above-mentioned  PDE constraint  on   the admissible fields presents various difficulties for  the  derivation of the $\Gamma$-limsup inequality. 

In particular, in the construction of recovery sequences  approximating a limit with multiple flat interfaces, suitable quantitative two-well rigidity estimates are needed to deal with the geometric nonlinearity of the model. The main challenge appears to be the fact that for generic small-energy functions, even if one phase is predominant in a certain region, there might be small inclusions of the other phase, so-called \emph{minority islands}.  {\sc S.~Conti} and {\sc B.~Schweizer} dealt with this problem by showing  that still the deformation is \emph{$H^{1/2}$-rigid} on many lines (see \cite[Section 3.3]{conti.schweizer}). It seems, however, that their  specific geometric arguments cannot be extended easily to higher dimensions.  

 In the present paper, we overcome the issue of the dimension  by means of a novel quantitative two-well rigidity result for a model where the two wells have exactly one rank-one connection: after rotation, we may suppose without restriction that $B-A = \kappa {\rm e}_d \otimes {\rm e}_d $ for $\kappa >0$. We analyze a slightly modified version of the model in \eqref{eq:sol-sol} in which the energy is augmented by an anisotropic perturbation:
 \begin{align}\label{eq: nonlinear energy-intro}
 y\in H^2(\Omega;\R^2) \to E_{\ep,\eta}(y):=I_{\ep}(y)+ \eta^2 \int_{\Omega}\Big(  |\nabla^2 y|^2 - |\partial^2_{dd}y|^2 \Big)  \,dx
\end{align}
for $\eta>0$. We point out that the additional  anisotropic perturbation \RRR does not affect the frame-indifference of the model, and \EEE penalizes only transitions in the direction orthogonal to the rank-one connection ${\rm e}_d \otimes {\rm e}_d$. This  guarantees that the  behavior of the energies $E_{\ep,\eta}$ is qualitatively the same as that of the functionals $I_{\ep}$, see Remark \ref{rk:comparison-CS}. At the same time, from a technical point of view, it is \RRR to be expected \EEE that low-energy sequences of $E_{\ep,\eta}$ might be more rigid compared to the ones of $I_{\ep}$. We remark that  similar anisotropic perturbations have already been used in the literature for related problems (see,  e.g., \cite{kohn.muller, Zwicknagl}), and that this anisotropy is the reason why we restrict our analysis to the case of exactly one rank-one connection. 

 The additional higher order penalization situates our analysis within the framework of \emph{nonsimple} materials, whose characteristic feature is an elastic stored energy density dependent on second order derivatives of the deformations. Starting from the seminal works by {\sc R.A.~Toupin} \cite{toupin1, toupin2}, these materials have been the  subject of an intense research activity in nonlinear elasticity due to  their  enhanced compactness properties  \cite{ball-currie, mielke.roubicek, podioGuidugli}. On the one hand, the penalization factor $\eta$ will be chosen ``large enough" to \RRR ensure that \EEE the second order regularization \RRR induces sufficiently good rigidity properties. \EEE On the other hand, the factor $\eta$ will be ``small enough" to guarantee that \eqref{eq: nonlinear energy-intro} shows the same qualitative behavior as the unperturbed problem \eqref{eq:sol-sol}, at least asymptotically when passing to a linearized strain regime.  A formal asymptotic expansion, in fact, shows that, by considering deformations $y$ of the form $y={\rm id}+\ep u$, for a smooth displacement $u$, the energy $E_{\ep,\eta}(y)$ rewrites as
\begin{align}
&E_{\ep,\eta}(y)=E_{\ep,\eta}({\rm id}+\ep u)=\frac{1}{\ep^2}\int_{\Omega}W({\rm Id}+\ep \nabla u)\,dx+\ep^4 \int_{\Omega}|\nabla^2 u|^2\,dx+\eta^2\ep^2\int_{\Omega}\Big(  |\nabla^2 u|^2 - |\partial^2_{dd}u|^2 \Big)  \,dx\notag\\
&\quad\sim \frac12\int_{\Omega} D^2 W({\rm Id})\nabla u:\nabla u\,dx + {\rm O}(\ep^4) +{\rm O}(\eta^2\ep^2),\label{eq:lin-formal}
\end{align}
 where ${\rm id}$ denotes the identity function and ${\rm Id}$ its derivative. Thus, to ensure that our anisotropic penalization does not  perturb the qualitative small-strain behavior of \eqref{eq:sol-sol}, it is essential that $\eta \ll \ep^{-1}$. Let us mention that related problems in the framework of nonsimple materials have  recently been studied in the settings of viscoelasticity \cite{friedrich-kruzik} and multiwell energies \cite{alicandro.dalmaso.lazzaroni.palombaro}. There, under strong penalization of the full second gradient, rigorous counterparts of the formal linearization  \eqref{eq:lin-formal}  are performed by $\Gamma$-convergence.

 The first part of the paper is devoted to  a quantitative rigidity estimate for two-wells energies of the form \eqref{eq: nonlinear energy-intro}, see Theorem \ref{thm:rigiditythm}. Here, we formulate a simplified version illustrating the core of our result.

\begin{theorem}(Simplified statement of Theorem \ref{thm:rigiditythm})
\label{thm:rig-intro}
Let $\eta_{\ep,d}=\ep^{-1+\alpha(d)}$, where $\alpha(d)>0$ is a suitable exponent, only depending on the space dimension $d$ (see  Remark \ref{rem: afterth}(iv)  for its explicit expression). Let $\Omega = (-h,h)^d$ be a cube,  and $W = \dist^2(\cdot,SO(d)\lbrace A,B\rbrace)$.  Let $E>0$. Then  there exists a constant $C=C(h ,A,B,E)>0$ such that for every $y\in H^2(\Omega; \R^d)$ with $E_{\eps,\eta_{\ep,d}}(y) \le E$ we can find a rotation $R\in SO(d)$ and a phase indicator $\mathcal{M} \in BV(\Omega;\lbrace A,B \rbrace)$ satisfying
\begin{align}\label{eq: rigidity-intro}
\Vert\nabla y-R\mathcal{M} \Vert_{L^2( \Omega)}\leq C\ep  \ \ \ \text{and} \ \ \  |D\mathcal{M}|(\Omega) \le C. 
\end{align}
\end{theorem}

 We point out that our rigidity result holds for general dimensions $d\in \mathbb{N}$, $d\geq 2$, for every $\eta>0$, a large class of domains $\Omega$ and energy densities $W$, and a range of integrability exponents depending on the space dimension $d$. We refer to the statement of Theorem \ref{thm:rigiditythm} for the precise assumptions. 
 
The novelty with respect to previous contributions in the literature is the presence of the \emph{phase indicator} $\mathcal{M}$ that allows to quantify  the distance of the deformation gradient from the two wells by keeping track of which phase is ``active" in each region of $\Omega$. Previous quantitative rigidity estimates for two-well or multiwell energies with rank-one connections measure the distance of $\nabla y$ from  a single matrix in \emph{one} of the wells. The sharpest results in that direction either only guarantee an $L^2$-control in terms of  $\sqrt{\ep}$ (and not of $\eps$), or require the passage to a weaker norm. Interestingly, a construction involving specific minority islands   shows that the scaling $\sqrt{\eps}$ is sharp, see Remark \ref{rem: inclusions}. Thus, introducing a phase indicator is indispensable in order to obtain the optimal $\eps$-scaling in \eqref{eq: rigidity-intro}. We refer to Subsection \ref{sec: literature} for a literature  overview on multiwell rigidity estimates and for a comparison to our result.

%

The main idea in our proof is to replace the actual gradient of the deformation $\nabla y$, which satisfies $\nabla y \approx SO(d)\lbrace A, B\rbrace$, by $\nabla y B^{-1}$ on a set of finite perimeter associated to the $B$-phase region. The resulting  field $\gamma$ then satisfies $\gamma\approx SO(d)A$, but is in general incompatible (i.e., not curl-free). Estimate \eqref{eq: rigidity-intro} is then achieved by controlling carefully the curl of $\gamma$ and using one-well rigidity estimates for incompatible fields \cite{Chambolle-Giacomini-Ponsiglione:2007, lauteri.luckhaus, Mueller-Scardia-Zeppieri}. A similar  strategy of reducing  a multiwell problem to an incompatible single-well setting has been  adopted in \cite{kytavsev-lauteri-ruland-luckhaus} for proving compactness and structure results for a discrete, frame-indifferent  multiwell  problem. The added value of our argument is the combination of rigidity estimates for fields with non-zero curl and a decomposition of the domain into phase regions (see Proposition \ref{lemma: phases}).

 It turns out that the curl  of the introduced incompatible field $\gamma$ has both a bulk and a surface part. The delicate step is to control the surface curl.  As in  \cite{kytavsev-lauteri-ruland-luckhaus}, our strategy departs from the remark that a control on  the length of the interfaces between different phases  allows to provide  a bound on the surface curl. Our further step is the proof that the surface curl can be estimated in dependence of the normal vector of the interface, see Lemma \ref{lemma: curl}.  Remarkably, it turns out that the surface curl vanishes if the normal vector coincides with the direction of the rank-one connection. This observation together with the anisotropic perturbation in \eqref{eq: nonlinear energy-intro} then  guarantees  that the surface curl of such fields $\gamma$ is of order $\ep$.

The second part of the paper is devoted to an application of Theorem \ref{thm:rigiditythm} to the $\Gamma$-convergence analysis (see \cite{Braides:02,DalMaso:93} for a comprehensive introduction to the topic) of the model in \eqref{eq: nonlinear energy-intro}\RRR, in the case in which the factor $\eta$ \RRR depends on $\eps$ and satisfies $\eta\to +\infty$ as $\ep\to 0$. \EEE In particular, we show that, as $\ep\to 0$, the behavior of the energy functionals in \eqref{eq: nonlinear energy-intro} approaches that of the sharp-interface limit  
\begin{align*}
\mathcal{E}_0(y):=\begin{cases}
 K \,  \mathcal{H}^{d-1}(J_{\nabla y})&
\text{if }\nabla y\in BV(\Omega; R \lbrace A,B\rbrace) \ \text{ for some } R\in SO(d),\\
+\infty&\text{otherwise in }L^1(\Omega;\mathbb{R}^d),\end{cases}
 \end{align*}
where $K$ corresponds to the energy of optimal transitions between the two phases (see \eqref{eq: our-k1}). We now give  the exact statement of our second main result.

\begin{theorem}[Identification of a sharp-interface limit]
 Let $\eta_{\ep,d}=\ep^{-1+\beta(d)}$, where \RRR $0 < \beta(d) < 1$ \EEE is a suitable exponent, only depending on the space dimension $d$ (see \eqref{eq:eta-ep} for its explicit expression). Let $\Omega\subset \R^d$ be a bounded strictly star-shaped Lipschitz domain. Let $W$ satisfy assumptions {\rm H1}.--{\rm H5}. Then  $E_{\eps,\eta_{\ep,d}}$   $\Gamma$-converges to $\mathcal{E}_0$ in the strong $L^1$-topology.
\end{theorem}

We refer to Section \ref{sec:model} for the precise assumptions on the energy density $W$, and to Subsection \ref{sec: review-transition} for the definition of strictly star-shaped domains, as well as for an overview of the relevant existing results   on solid-solid phase transitions.

The proof is divided into two steps, relying first on the identification of a lower bound, the \emph{liminf inequality} (see Proposition \ref{thm:liminf}), and then on the proof of its optimality, the \emph{limsup inequality} (see Theorem \ref{thm:limsup2}).

The proof of the liminf inequality follows the strategy in \cite[Proof of Theorem 4.1]{conti.fonseca.leoni}, and is based on a $d$-dimensional analysis of the properties of the optimal-profile energy $K$ (see Proposition \ref{prop:cell-form}). The main novelty of our result lies in the proof of the optimality of the lower bound identified in Proposition \ref{thm:liminf} in any dimension $d \ge 3$. As a byproduct of our analysis, we also provide a simplified construction of recovery sequences in the two-dimensional setting. In the seminal paper \cite{conti.schweizer}, indeed, the identification of deformations  approximating energetically a limit with multiple flat interfaces relies on the notion of \emph{$H^{1/2}$-rigidity on lines} (see \cite[Section 3.3]{conti.schweizer}), which requires deeply geometrical and technical constructions currently non-available in dimension $d>2$. By means of our quantitative rigidity estimate, instead, we overcome this issue by directly obtaining a control on the $W^{1,p}$-norm of the restriction of deformations to $(d-1)$-dimensional slices, for suitable exponents $p$ in a range depending on the dimension $d$ (see Proposition  \ref{lemma: optimal profile}). Once this control on slices is established, we may follow   the lines of \cite{conti.schweizer,conti.schweizer2} to ``glue together'' several interfaces and  to construct recovery sequences. We include the statements and the proofs of these technical lemmas from  \cite{conti.schweizer,conti.schweizer2} in order to keep the paper self contained.  This allows us to develop a comprehensive argument valid in any dimension $d\geq 2$.

As a final remark, we point out that a second application of the rigidity estimate in Theorem \ref{thm:rigiditythm} will be provided in the companion paper \cite{davoli.friedrich}. There, again departing from a singularly perturbed two-well problem of the form \eqref{eq: nonlinear energy-intro}, we will perform a simultaneous sharp-interface and small-strain limit, complementing recent results about the linearization of multiwell energies \cite{alicandro.dalmaso.lazzaroni.palombaro, Schmidt:08}: we  will identify an effective linearized model defined on suitably rescaled displacement fields which measure the distance to simple laminates.  

The structure of the paper is the following: in Section \ref{sec:model} we describe the setting of the problem and collect the main constitutive assumptions. Section \ref{sec: rigidity estimate} contains an overview of quantitative multiwell rigidity estimates, as well as the exact statement and the proof of our two-well rigidity result. Section \ref{sec:applications} is devoted to the proof of the variational convergence of our diffuse model to a sharp-interface limit.

\subsection*{Notation}
 We will omit the target space of our functions whenever this is clear from the context. For $d\in \N$, we denote by ${\rm e}_1,\dots, {\rm e}_d$ the standard basis. In what follows, ${\rm Id}$ denotes the identity matrix and ${\rm e}_{ij}$, $i,j=1,\dots,d$, the standard basis in $\M^{d\times d}$.  Given two vectors $v,w\in \mathbb{R}^d$, their tensor product is denoted by $v\otimes w$ and is defined as the matrix $(v\otimes w)_{ij}=v_iw_j$ for $i,j=1,\dots,d$. For every set $S\subset \R^d$, we indicate by $\chi_S$ its characteristic function, defined as
$$\chi_S(x):=\begin{cases}1&\text{if }x\in S\\0&\text{otherwise}.\end{cases}$$ 
The $d$-dimensional Lebesgue and Hausdorff measure of a set will be indicated by $\mathcal{L}^d$ and $\mathcal{H}^d$, respectively. We use standard notation for Sobolev spaces and $BV$ functions.


\section{The model}
\label{sec:model}

In this section we introduce our model. Let $d\in \mathbb{N}$, $d\geq 2$,  and let $\Omega \subset \R^d$ be a bounded Lipschitz domain. To any deformation $y\in H^1(\Omega;\R^d)$, we associate the elastic energy $\int_\Omega W(\nabla y)\, dx$, 
where \mbox{$W : \M^{d \times d} \to [0, + \infty) $} represents a stored-energy density satisfying the following properties:

\begin{itemize}
\item[H1.](Regularity) $W$ is continuous;
\item[H2.](Frame-indifference) $W(RF)=W(F)$ for every $R\in SO(d)$ and $F\in \M^{d\times d}$;
\item[H3.](Two-well rigidity) $W(A)=W(B)=0$, where $A={\rm Id}$, and $B={\rm diag}(1,\ldots,1,1+\kappa)\in \M^{d\times d}$ for   $\kappa>0$; 
\item[H4.](Coercivity) there exists a constant $c_1>0$ such that  
\begin{align*}
 W(F) \ge c_1 \dist^2(F,SO(d)\lbrace A,B \rbrace) \ \ \ \text{for every $F\in \M^{d\times d}$.}
 \end{align*}
 \end{itemize} 
 
Assumptions H1.-H4.\ are standard requirements on stored-energy densities in nonlinear elasticity. Given two matrices $A, B \in \M^{d \times d}$ with $\det(A),\det(B)>0$ such that $ SO(d) \lbrace A,B\rbrace$ contains at least one rank-one connection, i.e., ${\rm rank}(B- RA)  = 1$ for some $R\in SO(d)$, one can always assume (after an affine change of variables) that $A= {\rm Id}$ and $B={\rm diag}(\lambda,1,\ldots,1,\mu)$ for $\lambda,\mu>0$ with $\lambda\mu \ge 1$. (See \cite[Discussion before Proposition 5.1 and Proposition 5.2]{dolzmann.muller} for a proof.) In particular, depending on the values of $\lambda$ and $\mu$, the set $ SO(d) \lbrace A,B\rbrace$ contains exactly two, one, or no rank-one connections  (up to rotations), see \cite[Proposition  5.1]{dolzmann.muller}. In the present contribution, we focus on the case of exactly one rank-one connection, see H3.: the only solution of $ B - R A = a \otimes \nu$ with $R \in SO(d)$, $a,\nu \in \R^d$, and  $|\nu|=1$ is given by $R= {\rm Id}$, $\nu= {\rm e}_d$, and $a = \kappa {\rm e}_d$.

In the following, we will call $A$ and $B$  the \emph{phases}. Regions of the domain where $\nabla y$ is in a neighborhood of $SO(d)A$ will be called the $A$-\emph{phase regions} of $y$  and \RRR likewise \EEE we will speak of the $B$-\emph{phase regions}.

 In order to model solid-solid phase transitions, we analyze a nonlinear energy given by the sum
of a suitable rescaling of the elastic energy, a singular perturbation, and a higher-order
penalization in the direction orthogonal to the rank-one connection. To be precise, for $\ep, \eta>0$ we consider the functional
 \begin{align}\label{eq: nonlinear energy}
E_{\ep,\eta}(y):=\frac{1}{\ep^2}\int_{\Omega}W(\nabla y)\,dx+\ep^2\int_{\Omega}|\nabla^2 y|^2\,dx+ \eta^2 \int_{\Omega}\Big(  |\nabla^2 y|^2 - |\partial^2_{dd}y|^2 \Big)  \,dx
\end{align}
for every $y\in H^2(\Omega;\R^d)$.

The parameter $\eps$ in the definition above represents the typical order of the strain, whereas $\eps^2$ is
related to the size of transition layers \cite{ball.james,bhattacharya, capella.otto2, kohn.muller2,muller}. The first term on the right-hand side of \eqref{eq: nonlinear energy} favors deformations $y$ whose gradient is close to the two wells of $W$, whereas the second and third terms  penalize transitions between two different values of the gradient. The choice $\eta = 0$ corresponds to the model for solid-solid phase transitions investigated by {\sc S.~Conti}  and  {\sc B.~Schweizer} \cite{conti.schweizer} in dimension two. For $\eta>0$, the third
term compels transitions to happen preferably in the direction of the rank-one connection. The basic idea of our contribution is that this additional anisotropic perturbation allows us to prove a stronger rigidity estimate and to extend the findings in \cite{conti.schweizer} to a multidimensional setting. 

 Although the  additional penalization term renders our model more specific, we emphasize that it does not affect the qualitative behavior of the sharp-interface limit obtained in \cite{conti.schweizer}, see Remark \ref{rk:comparison-CS}. We note that this anisotropy is the reason why we restrict our study to the case of exactly one rank-one connection. We also mention that anisotropic singular
perturbations have already been used in related problems, see, e.g., \cite{kohn.muller, Zwicknagl}.


\section{Two-well rigidity}\label{sec: rigidity estimate}
This section is devoted to a quantitative rigidity estimate for  the  two-well   energies in \eqref{eq: nonlinear energy}, with densities $W$ satisfying H1.-H4. We first formulate the main theorem.

\begin{theorem}[Two-well rigidity estimate]\label{thm:rigiditythm}
  (a) Let $\Omega$ be a bounded, simply connected Lipschitz  domain in $\R^2$  and let $\eta \ge \eps>0$.   Then  there exists a constant  $C=C(\Omega,\kappa, c_1)>0$  such that for every $y\in H^2(\Omega;\R^2)$ there exist  a rotation $R\in SO(2)$ and a phase indicator $\mathcal{M} \in BV(\Omega;\lbrace A,B \rbrace)$ satisfying
\begin{align*}
\Vert\nabla y-R\mathcal{M} \Vert_{L^2( \Omega )}\leq C \eps  \sqrt{E_{\ep,\eta}(y)} +  C\Big( \frac{\eps}{\eta}    + \frac{\eps^{1/2}}{\eta^{3/2}}\Big)   \, E_{\ep,\eta}(y)  \ \ \ \text{and} \ \ \  |D\mathcal{M}|(\Omega) \le CE_{\ep,\eta}(y).
\end{align*}     
  
  (b) Let $\Omega$ be a bounded  Lipschitz  domain in $\R^d$ with $d\in \mathbb{N}$, $d\geq 3$. Let $1\leq p\leq 2,\, p \neq \frac{d}{d-1}$,  and let $\eta \ge \eps>0$.  Then for each $\Omega' \subset \subset \Omega$ there exists a constant  $C=C(\Omega,\Omega', \kappa, p,c_1)>0$  such that for every $y\in H^2(\Omega;\R^d)$ there exist  a rotation $R\in SO(d)$ and a phase indicator $\mathcal{M} \in BV(\Omega;\lbrace A,B \rbrace)$ satisfying
\begin{align}\label{eq: rigidity-new}
\Vert\nabla y-R\mathcal{M} \Vert_{L^p(\Omega')}\leq C\ep \sqrt{E_{\ep,\eta}(y)} +   C\Big(   \Big( \frac{\eps}{\eta} + \frac{\eps^{1/2}}{\eta^{3/2}}         \Big)   \, E_{\ep,\eta}(y)\Big)^{r(p,d)}   \ \ \ \text{and} \ \ \  |D\mathcal{M}|(\Omega) \le CE_{\ep,\eta}(y),
\end{align}  
 where $r(p,d)=\min\{1,\frac{d}{p(d-1)}\}$.
\end{theorem}

\begin{remark}[Different exponents, sets, and simplified formulations]\label{rem: afterth}
{\normalfont
{\rm (i)} The difference in the formulations in two and arbitrary space dimensions, concerning the exponents and the assumptions on the sets, are due to the application of rigidity estimates for vector fields with nonzero curl, see Lemma  \ref{lemma: Muller-Chambolle} below.  Although we neglect the case $p=\frac{d}{d-1}$ in (b), we point out that our argument could be extended to also cover that scenario, by replacing Lemma \ref{lemma: Muller-Chambolle} below with the results for $p=\frac{d}{d-1}$ in \cite[Theorem 4]{lauteri.luckhaus}. 

{\rm (ii)} In (b), if $\Omega$ is a  paraxial  \RRR cuboid, \EEE the statement holds on the entire domain.

 {\rm (iii)} For general sets $\Omega$, we point out that for $p>\frac{d}{d-1}$ the rigidity estimates for vector fields with nonzero curl in Lemma  \ref{lemma: Muller-Chambolle}(b) hold on the whole set (see Remark \ref{rk:sets-for-lemma}). Nevertheless, the passage to a subdomain is still needed for Theorem \ref{thm:rigiditythm} due to a combination of covering and isoperimetric arguments in Step II of the proof.  We are aware of the possibility to formulate Theorem \ref{thm:rigiditythm}(b) on the whole set $\Omega$ for a more general class of sets having suitable geometrical properties. Nonetheless, we have decided not to dwell on this point, both to keep the exposition from becoming too technical, and as it is only of marginal interest for the applications that we will treat in this paper and in \cite{davoli.friedrich}. Note that the constant $C$ in the theorem is invariant under uniformly Lipschitz reparametrizations of the domain. 
}

 {\rm (iv)} Consider the special case  $\eta = \eps^{-1+4/  (3d)}$  for deformations $y\in H^2(\Omega;\R^d)$ with $E_{\ep,\eta}(y) \le E$ for some $E>0$. Then,  when $\Omega$ is a paraxial cube, the statement reduces to that of Theorem \ref{thm:rig-intro} choosing $\alpha(d)=4/(3d)$.

\end{remark}

The section is organized as follows. In Section \ref{sec: literature} we first  provide  a brief literature overview of quantitative  rigidity estimates for multiwell energies and situate Theorem \ref{thm:rigiditythm} within this context. In Section \ref{sec: curl} we then recall rigidity estimates for vector fields with nonzero curl, which are the main ingredient in our approach. Section \ref{sec: prelim-est} is devoted to some preliminary estimates concerning the decomposition of the domain into the phase regions of $A$ and $B$. Finally, Section \ref{sec: rig-proof} contains the proof of Theorem \ref{thm:rigiditythm} and some further remarks on the result.

\subsection{Theorem \ref{thm:rigiditythm} in the context of quantitative  rigidity estimates for multiwell energies}\label{sec: literature}

Theorem \ref{thm:rigiditythm} is related to a variety of quantitative rigidity estimates for multiwell energies.  Roughly speaking, all these results control the distance of the deformation gradient from a single matrix in  \emph{one} of the wells  in a suitable norm. We recall the most important theorems in that direction. In the sequel, $y$ denotes a deformation satisfying $\int_\Omega W(\nabla y) \, dx \le C\eps^2$.

 If the two wells are strongly incompatible in the sense of \cite{Matos}, it was proven in \cite{Chaudhuri, De Lellis} that there exist $R\in SO(d)$ and $ M\in \{A,B\}$ such that 
\begin{align}\label{eq: rig-mot1}
 \Vert \nabla y - R  M  \Vert_{L^2(\Omega)} \le C\eps,
 \end{align}
even without imposing a second order penalization. For multiple wells with possible rank-one connections,  it was shown in \cite[Theorem 2.3]{alicandro.dalmaso.lazzaroni.palombaro}  that  an  estimate of the form \eqref{eq: rig-mot1} still holds if a sufficiently strong second-order penalization \RRR for the full hessian (and hence, also for the direction of the rank-one connection) \EEE is assumed.  Both results, however, are not relevant for our  applications \RRR since they are derived for models which do not allow for the  formation of phase transitions  in the sharp-interface limit $\eps \to 0$. Indeed, for $\eps$ small, \eqref{eq: rig-mot1} induces that $\nabla y$ is close everywhere to a single matrix, and this is not consistent with the presence of a phase transition.  \EEE

%
%
%

Concerning two-well problems  with rank-one connections allowing for phase transitions, the first results have  been derived \RRR in \cite{Lorent}, and more generally in \cite{conti.schweizer}, \EEE in dimension two. These estimates have  been generalized later \RRR to arbitrary space dimensions for multiple wells under a well-separatedness assumption in \cite{Chermisi-Conti}, and under more general connectivity conditions in \cite{Jerrard-Lorent}\EEE. More precisely, in the case of two wells, the result is as follows: for $y \in H^2(\Omega;\R^d)$,  \RRR satisfying again the energy bound $\int_\Omega W(\nabla y) \, dx \le C\eps^2$ and additionally \EEE $ \Vert \nabla^2 y\Vert_{L^1(\Omega)} \le a$ for some small $a>0$, there exist $R \in SO(d)$ and $ M \in \lbrace A,B \rbrace$ such that 
\begin{align}\label{eq: rig-mot2}
\Vert \nabla y - R M \Vert_{L^2(\Omega')} \le C\sqrt{\ep}, 
\end{align}
where $\Omega'$ is subdomain of $\Omega$. In this context, the assumption that $a$ is small is essential since it guarantees that the $ M$-phase region is predominant. Still, it does not exclude the occurrence of phase transitions near the boundary.
 Indeed, \eqref{eq: rig-mot2} is  generally not true if $\Omega' = \Omega$. Moreover, a construction in  \cite[Example 6.1]{conti.schweizer2} shows that the scaling $\sqrt{\eps}$ is sharp, see also Remark \ref{rem: inclusions} below. The scaling $\sqrt{\eps}$ is insufficient for our applications to solid-solid phase transitions since the strain is typically of order $\eps$, see Remark \ref{rem: comparison}.

We recall that in \cite{conti.schweizer} also variants for the weak $L^1$-norm are discussed. In particular, it is shown that there exist  $R \in SO(d)$ and $M \in \lbrace A,B \rbrace$ such that 
\begin{align}\label{eq: rig-mot3}
\|\nabla y-R M\|_{w-L^1(\Omega')}\leq C\ep.
\end{align}
 Although the scaling in terms of $\eps$ corresponds to the typical order of the strain, the fact that the estimate only holds in the weak $L^1$-norm prohibits application of this estimate in Section \ref{sec:applications},  see Remark \ref{rem: comparison}.

We remark that all  the  results mentioned above follow the same strategy: one shows that the volume of the phase region different from $ M $ is asymptotically small in $\eps$. This is either induced by the incompatibility of the wells or by a second order penalization. For $1 \le p \le 2$ this yields the estimate
\begin{equation}
\label{eq:no-curl}
\Vert {\rm dist}(\nabla y,SO(d) M)\Vert_{L^p(\Omega)} \le C \Vert {\rm dist}(\nabla y,SO(  d )\lbrace A,B \rbrace)\Vert_{ L^{p}(\Omega) } +  C  V_\eps^{1/p} \le C\eps + V_\eps^{1/p},
\end{equation}
where $V_\eps$ denotes the volume of the phase region different from $ M $. Afterwards, one applies the seminal one-well rigidity estimate  \cite{FrieseckeJamesMueller:02} (cf. also \cite[Section 2.4]{conti.schweizer}) to obtain \eqref{eq: rig-mot1}-\eqref{eq: rig-mot3} in the various settings.  

Our approach is quite different as we establish a rigidity estimate which takes the presence of \emph{both} phases into account. This is reflected by the \emph{phase indicator} $\mathcal{M}$ and is inspired by piecewise rigidity results \cite{Friedrich-ARMA, Friedrich:15-4} in other settings. In   particular,   Theorem \ref{thm:rigiditythm} complements the existing results in the following ways: (1) For the derivation of rigidity results, no smallness assumption on the full second derivative is needed; (2) Identifying the different phase regions by means of $\mathcal{M}$ allows to improve the scaling in \eqref{eq: rig-mot2}, cf.\ Remark \ref{rem: afterth}{\rm (iv)} and Theorem \ref{thm:rig-intro};  (3) If the domain is two-dimensional or a paraxial \RRR cuboid \EEE in higher dimensions, the estimate holds on the entire set $\Omega$. (The  necessity  of taking a subset in higher dimensions is not due to the presence of different phases, but due to a combination of covering and isoperimetric arguments in the proof, see Remark \ref{rem: afterth}(iii) for a discussion.)

Note that for technical reasons we need to take an anisotropic penalization into account, see \eqref{eq: nonlinear energy}. This, however, does not affect the qualitative behavior of the sharp-interface limit derived in Section \ref{sec:applications}, see Remark \ref{rk:comparison-CS}.

\subsection{Rigidity estimates for vector fields with nonzero curl}\label{sec: curl}

The main idea in our approach will be the usage of rigidity estimates for vector fields with nonzero curl established  in \cite{Chambolle-Giacomini-Ponsiglione:2007, lauteri.luckhaus, Mueller-Scardia-Zeppieri}.  We first define the  curl and  recall the relevant results. Let $\gamma \in L^1(\Omega; \R^d)$. The distribution $\curl \gamma$ is formally equal to the matrix $(\partial_i \gamma_j - \partial_j \gamma_i)_{1 \le i,j \le d}$  and is defined as
\begin{align}\label{eq: curl definition}
\langle \curl \gamma, \varphi\rangle = \sum_{i,j=1}^d \int_\Omega \gamma_i(x) \partial_j (\varphi_{ij}(x) - \varphi_{ji}(x)) \, dx 
\end{align}
for all $\varphi \in C^\infty_c(\Omega; \M^{d \times d})$. If $\gamma$ is a matrix-valued vector field,   then $\curl \gamma$ is a distribution taking values in $\R^d\times \M^{d\times d}$, and formally defined as 
$(\curl \gamma)_{kij}=\partial_i \gamma_{kj}-\partial_j \gamma_{ki}$ for every $1\leq k,i,j\leq d$.

\begin{lemma}[Rigidity estimates  for vector fields with nonzero curl]\label{lemma: Muller-Chambolle}
  (a) Let $\Omega$ be a bounded, simply connected Lipschitz  domain in $\R^2$.  Then   there exists a constant $C=C(\Omega )>0$ satisfying the following property: for every $\gamma \in L^2(\Omega; \M^{2\times 2})$ such that  $\curl \gamma$ is a bounded measure there exists $R\in SO(2)$ for which
\begin{align*}
\Vert\gamma-R\Vert_{L^2(\Omega;\M^{2\times 2})}\leq C  \Big(\Vert \dist(\gamma,SO(2)) \Vert_{L^2( \Omega)} + |\curl \gamma|(\Omega)  \Big).
\end{align*}
  
  (b) Let $\Omega$ be a bounded  Lipschitz  domain in $\R^d$  with $d\in \mathbb{N}$,  $d\geq 3$,   and let $1 \le p \le 2$, $p \neq \frac{d}{d-1}$. Then for each $\Omega' \subset \subset \Omega$ there exists a constant  $C=C(\Omega, \Omega',p)>0$  satisfying the following property: for every  $\gamma \in L^{p}(\Omega; \M^{d\times d})$  such that  $\curl \gamma$ is a bounded measure, there exists $R\in SO(d)$ for which
\begin{align}\label{eq: rigidity2}
\Vert\gamma-R\Vert_{L^p(\Omega' ;\M^{d\times d})}\leq C  \Big(\Vert \dist(\gamma,SO(d)) \Vert_{L^p( \Omega)} +  (|\curl \gamma|(\Omega))^{r(p,d)} \Big),
\end{align}
   where $r(p,d)=\min\{1,\frac{d}{p(d-1)}\}$. 
\end{lemma}

\begin{proof}
Assertion (a) is proven in \cite[Theorem 3.3]{Mueller-Scardia-Zeppieri}. The proof of assertion (b)  for $p<\frac{d}{d-1}$ is essentially contained  in \cite[Proposition 5.1]{Chambolle-Giacomini-Ponsiglione:2007} if the domain is a cube. For general $\Omega' \subset \subset \Omega$, we use a standard covering argument (see, e.g., \cite[Proof of Theorem 1]{Chambolle-Conti-Francfort:2014} or  \cite[Proof of Theorem 1.1]{Friedrich:15-3}): we cover $\Omega'$ with a finite number of open cubes $\lbrace Q_i \rbrace_{i=1}^N$ and apply  \cite[Proposition 5.1]{Chambolle-Giacomini-Ponsiglione:2007} on each of the cubes to obtain rotations $\lbrace R_i \rbrace_{i=1}^N$ such that \eqref{eq: rigidity2} holds on $Q_i$ for a constant $C_i$ dependent on $Q_i$. The difference between  rotations in neighboring cubes is then controlled in terms of a constant which only depends on $d$, $N$, and $\min\lbrace  \mathcal{L}^d (Q_i \cap Q_j): \  Q_i \cap Q_j \neq \emptyset \rbrace $.  Assertion (b) for $p>\frac{d}{d-1}$ follows directly by \cite[Theorem 4]{lauteri.luckhaus} or \cite[Theorem 3]{kytavsev-lauteri-ruland-luckhaus} if $\Omega'$ is a ball, and by a covering argument analogous to the one described above for more general $\Omega'  \subset\subset  \Omega$. 
 \end{proof}
 
\begin{remark}[Role of the subdomain]
\label{rk:sets-for-lemma}
{\normalfont
As a direct consequence of the proof of Lemma \ref{lemma: Muller-Chambolle}(b), for $1 \le p < \frac{d}{d-1}$ and for $\Omega$ coinciding with a cube, we do not have to take a subset of the domain.  \RRR The same holds if $\Omega$ is  a paraxial cuboid: indeed, it can be covered by a finite number of paraxial cubes of equal size, and   the difference of the corresponding rotations can be controlled, cf.\ the proof of Lemma \ref{lemma: Muller-Chambolle}. \EEE
 Additionally, for $p>\frac{d}{d-1}$ the statement can also be proven for general Lipschitz sets $\Omega$ without passing to subdomains. This follows from the scaling invariance of the rigidity estimate for incompatible fields in \cite[Theorem 3]{kytavsev-lauteri-ruland-luckhaus} and by a classical covering argument  (see, e.g., \cite[Proof of Theorem 3.1]{FrieseckeJamesMueller:02}). The same argument does not apply to Lemma \ref{lemma: Muller-Chambolle}(b) for $1 \le p < \frac{d}{d-1}$ as the estimate in  \cite[Proposition 5.1]{Chambolle-Giacomini-Ponsiglione:2007} is not scaling invariant.}
\end{remark}

Our strategy to prove Theorem \ref{thm:rigiditythm} is to replace the gradient $\nabla y$, which satisfies $\nabla y \approx SO(d)\lbrace A, B\rbrace$, by an associated vector field $\gamma$ with $\gamma \approx SO(d)A$. This will be done by changing $\nabla y$ to $\nabla y B^{-1}$ on a set of finite perimeter associated to the $B$-phase regions. A similar strategy to replace a multiwell problem by an incompatible one-well problem has been used in \cite{kytavsev-lauteri-ruland-luckhaus}.  In contrast  to \cite{kytavsev-lauteri-ruland-luckhaus}, we provide a finer control on the curl of the incompatible vector field.  To this end, we investigate the curl of vector fields which are $SBV$ functions. We recall that $\gamma \in L^1(\Omega;\R^d)$ lies in $SBV(\Omega;\R^d)$ if its distributional derivative $D\gamma$ is an $\R^{d\times d}$-valued finite Radon measure on $\Omega$ such that
\begin{align}\label{eq: SBV}
D\gamma= \nabla \gamma\mathcal L^{d}+[\gamma]\otimes \nu_\gamma \mathcal H^{d-1}\lfloor J_\gamma,
\end{align}
where $\nabla \gamma = (\partial_1 \gamma,\ldots, \partial_d\gamma)$ denotes the approximate differential, $\nu_\gamma$ is a normal of the jump set $J_\gamma$ and $[\gamma] := \gamma^+ - \gamma^-$ with $\gamma^{\pm}$ being the one-sided limits of  $\gamma$ at $J_{\gamma}$ (see \cite[Chapter 4]{Ambrosio-Fusco-Pallara:2000}). The following lemma yields a control on $\curl\gamma$. For related curl-estimates for $SBV$ functions we refer to \cite[Theorem 3.1]{Chambolle-Giacomini-Ponsiglione:2007}.

\begin{lemma}[Curl for $SBV$ vector fields]\label{lemma: curl}
Let $\gamma = (\gamma_1,\dots,\gamma_d)\in SBV(\Omega;\R^d)$. Then, $\curl \gamma $   is a measure on $\Omega$ satisfying
 $$|\curl \gamma|(\Omega)   \le   d  \int_\Omega | (\nabla \gamma)^T - \nabla \gamma| \, dx +   \int_{J_\gamma} |[\gamma]  \otimes \nu_\gamma- \nu_\gamma \otimes [\gamma] | \, d\mathcal{H}^{d-1}.$$ 
\end{lemma}

\begin{proof} 
For each $\varphi \in C^\infty_c(\Omega; \M^{d \times d})$ we have by \eqref{eq: curl definition} and \eqref{eq: SBV} 
\begin{align*}
\langle \curl\gamma, \varphi\rangle & = \sum_{i,j=1}^d \int_{  \Omega  }  \gamma_i  (x)  \partial_j (\varphi_{ij}(x) - \varphi_{ji}(x))   \, dx \\
&  =  -  \sum_{i,j=1}^d \int_{J_\gamma} ([\gamma] \otimes   \nu_\gamma)_{ij} (x)  (\varphi_{ij}(x) - \varphi_{ji}(x))    \, d\mathcal{H}^{d-1}(x)  -     \sum_{i,j=1}^d \int_{\Omega} \partial_j \gamma_i(x)    (\varphi_{ij}(x) - \varphi_{ji}(x))    \, dx  \\
& =   -   \sum_{i,j=1}^d \int_{J_\gamma} ([\gamma] \otimes   \nu_\gamma-  \nu_\gamma  \otimes   [\gamma] )_{ij}(x) \, \varphi_{ij}(x)  \, d\mathcal{H}^{d-1}(x)   -    \sum_{i,j=1}^d   \int_{\Omega}  \varphi_{ij}(x)    (\partial_j \gamma_i(x) - \partial_i \gamma_j(x))    \, dx.   
\end{align*}
 This implies that
$$|\langle \curl \gamma, \varphi\rangle| \le \Vert \varphi \Vert_{L^\infty(\Omega)}  \Big( \sum\nolimits_{i,j=1}^d\int_\Omega |\partial_i \gamma_j - \partial_j \gamma_i| \, dx +   \int_{J_\gamma} |[\gamma]  \otimes \nu_\gamma- \nu_\gamma \otimes [\gamma] | \, d\mathcal{H}^{d-1}\Big)$$ 
for every $\varphi\in C^{\infty}_c(\Omega;\mathbb{M}^{d\times d})$, and concludes the proof of the lemma.  
\end{proof}

\subsection{Decomposition into phases}\label{sec: prelim-est} 
We adopt the notation $V(F) = \dist^2(F,SO(d)\lbrace A, B \rbrace)$ for brevity. We introduce the truncated geodesic distance $d_V(F,G)$ of $F,G \in \M^{d\times d}$ induced by $V$, which is defined by 
\begin{align}\label{eq: geodesic distance}
d_V(F,G) = \inf  \Bigg\{ \int_0^1 \min\lbrace\sqrt{V(g(s))},1\rbrace \, |g'(s)|\, ds: \ g \in C^1([0,1];\M^{d\times d}), \ g(0)=F,  \ g(1) = G \Bigg\}.
\end{align}
\RRR In this subsection, we will only use the fact that $|B-A| = \kappa$, see H3. The other properties of $A$ and $B$ will not be used in the proofs, not even the fact that they are rank-one connected. \EEE Clearly, we have $d_V(A,B) >0$. For later purposes, we state some elementary properties.

\begin{lemma}[Relation between \RRR Euclidean \EEE distance and geodesic distance]\label{lemma: geodesic distance}
Let $\delta>0$. There exist $C_1 \ge 1$ and $0 < C_2 < 1$ depending only on \RRR $\kappa$ and \EEE $\delta$ such that for $M \in \lbrace A,B \rbrace$
\begin{align*}
{\rm (i)} & \ \ d_V(F,SO(d)M) \le \dist(F,SO(d)M)   \ \  \ \text{for every $F\in \M^{d \times d}$},\\
{\rm (ii)} & \ \ \dist(F,SO(d)M) \le C_1 d_V(F,SO(d)M) \ \  \ \text{for every $F\in \M^{d \times d}$ such that  $d_V(F,SO(d)M)\ge \delta$.}\\
{\rm (iii)} & \ \ \dist(F,SO(d)M) \le C_2  \ \  \ \text{for every $F\in \M^{d \times d}$ such that  $d_V(F,SO(d)M) <  \delta$.}
\end{align*}
Moreover, there holds $C_2 \to 0$ when $\delta \to 0$. 
\end{lemma}

\begin{proof}
Item {\rm (i)} follows directly from the definition of the geodesic distance. For the proof of {\rm (ii)} and {\rm (iii)} we refer to \cite[Lemma 2.5 and Lemma 2.6]{alicandro.dalmaso.lazzaroni.palombaro}. The last assertion is a consequence of the proof of \cite[Lemma 2.6]{alicandro.dalmaso.lazzaroni.palombaro}.   Note that the definition of the geodesic distance in \cite{alicandro.dalmaso.lazzaroni.palombaro} is slightly different from \eqref{eq: geodesic distance}, but that \cite[Lemma 2.5 and Lemma 2.6]{alicandro.dalmaso.lazzaroni.palombaro} still hold up to very minor proof adaptations.  
\end{proof}

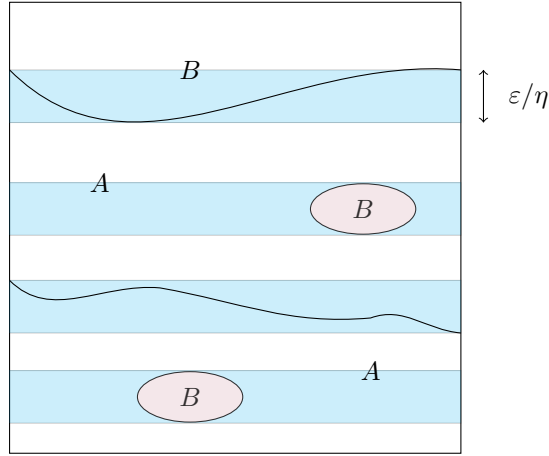
\begin{figure}[h]
\centering
\begin{tikzpicture}
\coordinate (A) at (-3,3);
 \coordinate (B) at (3,3);
 \coordinate (C) at (3,-3);
 \coordinate (D) at (-3,-3);
 \coordinate (E) at (3,2.1);
 \coordinate (F) at (3,1.4);
 \coordinate (G) at (3,0.6);
 \coordinate (H) at (3,-0.1);
 \coordinate (I) at (3,-0.7);
 \coordinate (L) at (3,-1.4);
 \coordinate (M) at (3,-1.9);
 \coordinate (N) at (3,-2.6);
  \coordinate (O) at (-3,2.1);
 \coordinate (P) at (-3,1.4);
 \coordinate (Q) at (-3,0.6);
 \coordinate (R) at (-3,-0.1);
 \coordinate (S) at (-3,-0.7);
 \coordinate (T) at (-3,-1.4);
 \coordinate (U) at (-3,-1.9);
 \coordinate (V) at (-3,-2.6);

 \draw (A)--(B)--(C)--(D)--cycle;
 \draw[fill=cyan, opacity=0.2] (O)--(E)--(F)--(P)--cycle;
 \draw [black] (-3,2.1)  to [out=-45,in=175] (3,2.1);
  \draw[fill=cyan, opacity=0.2] (Q)--(G)--(H)--(R)--cycle;
   \node[shape=ellipse, minimum height=0.1cm, minimum width=1.4cm, draw=black, fill=red!10, opacity=0.8] at (1.7, 0.25){$B$};
   \draw[fill=cyan, opacity=0.2] (S)--(I)--(L)--(T)--cycle;
    \draw [black] (-3,-0.7)  to [out=-45,in=175] (-1,-0.8) to [out=-10, in=-175] (1.8,-1.2)  to [out=20, in=-185] (3,-1.4);
    \draw[fill=cyan, opacity=0.2] (U)--(M)--(N)--(V)--cycle;
     \node[shape=ellipse, minimum height=0.1cm, minimum width=1.4cm, draw=black, fill=red!10, opacity=0.8] at (-0.6, -2.25){$B$};
  \draw[<->] (3.3,2.1)--(3.3,1.4);
 
\node at (-1.8,0.6) {$A$};
\node at  (-0.6,2.1) {$B$};
\node at  (1.8,-1.9) {$A$};
\node at (3.9,1.75) {$\ep/\eta$};
\end{tikzpicture}
\caption{ Condition (iv) for $d=2$ can be interpreted as follows: it guarantees that phase transitions occur inside cylindrical layers of height $\ep/\eta$. Additionally, $\ep/\eta$ is an upper bound on the height of minority islands in the ${\rm e}_d$-direction. In higher dimensions, a similar interpretation is possible, up to higher order terms.}
\label{fig:small-trans}
\end{figure}

The following \RRR proposition \EEE identifies the regions where the deformation gradient is near $SO(d)A$ and $SO(d)B$, respectively. Moreover, let $c_1$ be the constant of H4. For basic properties of sets of finite perimeter we refer to \cite[Section 3.3]{Ambrosio-Fusco-Pallara:2000}.  

\begin{proposition}[Decomposition into phases]\label{lemma: phases}
 Let $\eta \ge \eps$. There exist $0<\alpha < \beta \le 1/2$ and a constant $c=c(\kappa,d, c_1)>0$ such that for every $y \in H^2(\Omega;\R^d)$ there exists an associated set $T \subset \Omega$ of finite perimeter satisfying
\begin{align}\label{eq: propertiesT}
{\rm (i)} & \ \ \lbrace x\in \Omega: \  \dist(\nabla y(x),SO(d)A) \le \alpha\kappa \rbrace  \subset T \subset \lbrace x\in \Omega: \  \dist(\nabla y(x),SO(d)A) \le \beta\kappa \rbrace,\notag \\ 
{\rm (ii)}& \ \ \mathcal{H}^{d-1}(\partial^* T \cap \Omega) \le c E_{\ep,\eta}(y), \notag\\
{\rm (iii)} & \ \  \int_{\partial^* T \cap \Omega} |\langle \nu_T , {\rm e}_j \rangle| \, d\mathcal{H}^{d-1} \le c  \frac{\eps}{\eta}\, E_{\ep,\eta}(y) \ \  \text{for} \ \ j=1,\ldots, d-1, \notag \\  
{\rm (iv)} & \ \ \int_{-\infty}^\infty\mathcal{H}^{d-2}\Big( \big(\R^{d-1} \times \lbrace t \rbrace\big) \cap \partial^* T \cap \Omega\Big)  \, dt \le c \frac{\eps}{\eta}\, E_{\ep,\eta}(y),  
 \end{align}
where $\nu_T$ denotes the outer normal to $T$, $\partial^* T$ its essential boundary, and $E_{\ep,\eta}$ is the energy functional defined in \eqref{eq: nonlinear energy}.   Moreover, if $Q = x_0 + (-h,h)^d$ is a cube contained in $\Omega$ and one considers a corresponding decomposition by $(Q_l)_{l=-n}^{n-1}$ with $n =  \lfloor {\eta}/{\eps}\rfloor$,  and $Q_l := x_0 +   (lh/n)   e_d +   (-h,h)^{d-1} \times (0, h/n)$   we find  
\begin{align}\label{eq: propertiesT2}
\sum\nolimits_{l=-n}^{n-1} \min\lbrace \mathcal{L}^d(Q_l \cap T),   \mathcal{L}^d(Q_l \setminus T) \rbrace \le ch \frac{\eps}{\eta}\, E_{\ep,\eta}(y).
\end{align} 
\end{proposition}

Roughly speaking, the sets $T$ and $\Omega \setminus T$ represent the  $A$ and $B$-phase regions, respectively. 
Later in the proof of Theorem \ref{thm:rigiditythm} we will introduce a vector field which differs from $\nabla y$ exactly on the set $\Omega \setminus T$. We refer to Figure \ref{fig:small-trans} for an illustration and an explanation of property \eqref{eq: propertiesT}(iv).

\begin{proof}[Proof of Proposition \ref{lemma: phases}]
We first fix some constants which will be needed in the following. Depending on $\kappa$, we choose $\delta$  in Lemma \ref{lemma: geodesic distance} so small that 
\begin{align}\label{eq: C_2}
C_2=C_2(\delta) \le \frac{\kappa}{2}.
\end{align}
Let $C_1=C_1(\delta) \ge 1$ be the corresponding constant (depending on $\delta$, and hence on $\kappa$) provided by Lemma \ref{lemma: geodesic distance}{\rm (ii)}. We define 
$$h(x) = d_V(\nabla y(x),SO(d)A)\quad\text{for every }x\in \Omega,$$
where $d_V$ is the truncated geodesic distance introduced in \eqref{eq: geodesic distance}. The main idea of the proof consists in choosing the set $T$ as a suitable level set of the map $h$, selected by performing an $\eps/\eta$-rescaling of $h$ in its first $d-1$ variables (see \eqref{eq:hdelta}).

 \noindent\emph{Step I: Definition of $T$.} We first observe that, in view of the definition of $d_V$ and by Young's inequality, we obtain 
\begin{align*}
\int_\Omega |\partial_i h(x) |\, dx \le \int_\Omega \sqrt{V(\nabla y(x))} \sum_{j=1}^d |\partial_{ji} y(x)| \, dx \le \frac{1}{2\ep\eta} \int_\Omega  V(\nabla y(x))\, dx + \frac{\ep\eta}{2} \int_\Omega\Big( \sum_{j=1}^d |\partial_{ji} y(x)| \Big)^2\, dx
\end{align*}
for $i=1,\ldots, d-1$. Thus, H4., the definition of $V$, and \eqref{eq: nonlinear energy}  imply
\begin{align}\label{eq: unrescaled-grad1}
\Vert \partial_i h \Vert_{L^1(\Omega)} \le  \Big(\frac{ 1/c_1 +   d}{2} \Big)  \frac{\eps}{\eta} \, E_{\ep,\eta}(y) \ \ \text{for} \ \ i=1,\ldots,d-1.
\end{align}
Analogously, the definition of $d_V$ and Young's inequality yield
$$\int_\Omega |\partial_d h(x) |\, dx\le \frac{1}{2\ep^2} \int_\Omega  V(\nabla y(x))\, dx+\frac{\ep^2}{2} \int_\Omega\Big( \sum_{j=1}^d |\partial_{jd} y(x)| \Big)^2\, dx,$$
which implies
\begin{align}\label{eq: unrescaled-grad2}
 \Vert \partial_d h \Vert_{L^1(\Omega)} \le   \Big(  \frac{ 1/c_1 +  d}{2} \Big)  E_{\ep,\eta}(y).
 \end{align}
We introduce the rescaled function 
\begin{equation}
\label{eq:hdelta}
h_\eta (x', x_d) := h(\eta x'/\eps ,x_d),
\end{equation} defined on 
$$\Omega_\eta := \lbrace (x',x_d): \ (\eta x'/\eps,x_d) \in \Omega \rbrace,$$ 
where for brevity we adopt the notation $x' =(x_1,\ldots,x_{d-1})$.
By the change of variables formula and \eqref{eq: unrescaled-grad1}-\eqref{eq: unrescaled-grad2} this yields 
$$ \Vert \nabla  h_\eta \Vert_{L^1(\Omega_\eta)} \le { c (\eps/\eta)^{d-1}  E_{\ep,\eta}(y) } $$
 for $c=c(d,c_1)$. Consequently, by the coarea formula we find $t \in \big(\kappa/(4C_1), \kappa/(2C_1)\big)$, where $C_1\geq 1$ is the constant introduced below \eqref{eq: C_2}, such that the set $T_\eta:=\lbrace h_\eta \le t \rbrace$ has finite perimeter, with 
\begin{align}\label{eq: perimterTdelta}
\mathcal{H}^{d-1}(\partial^* T_\eta \cap \Omega_\eta) \le  \frac{4 C_1}{\kappa} \int_{\frac{\kappa}{4C_1}}^{\frac{\kappa}{2C_1}}\mathcal{H}^{d-1}(\partial^* \{h_{\eta}\leq s\} \cap \Omega_\eta)\,ds  \le \frac{4 C_1}{\kappa} \Vert \nabla h_\eta \Vert_{L^1(\Omega_\eta)} \le c  (\eps/\eta)^{d-1}  E_{\ep,\eta}(y),
\end{align}
 where $c$ depends  on $\kappa$, $d$, and $c_1$. We define $T :=  \lbrace h  \le t \rbrace$. We claim that $T$ satisfies properties {\rm (i)}-{\rm (iv)}. 
 
 \noindent\emph{Step II: Properties of $T$.} First, since $t > \kappa/(4C_1)$, by Lemma \ref{lemma: geodesic distance}{\rm (i)} we have   that   for all $x \in \Omega$ with   $\dist(\nabla y(x),SO(d)A) \le \kappa/(4C_1),$  there holds 
$$h(x)\leq \frac{\kappa}{4C_1}<t.$$ 
This yields $x \in T$ and implies that  the first inclusion in {\rm (i)} holds with $\alpha= 1/(4C_1)$.   Note that, since $C_1\geq 1$, we have $\alpha\leq 1/4$.  

To prove the second inclusion in {\rm (i)}, suppose that $x \in T$. Let $\delta$ be as in \eqref{eq: C_2}. If $h(x)  < \delta$, there holds 
$$\dist(\nabla y(x),SO(2)A) \le C_2 \le \frac{\kappa}{2}$$ 
by Lemma \ref{lemma: geodesic distance}{\rm (iii)} and \eqref{eq: C_2}. On the other hand, if $\delta \le h(x) \le t$, we obtain
$$\dist(\nabla y(x),SO(2)A) \le C_1 h(x) \le C_1 t \le \frac{\kappa}{2}$$ by Lemma \ref{lemma: geodesic distance}{\rm (ii)} and the definition of $t$. Setting $\beta=\frac{1}{2}$, this concludes the proof of {\rm (i)}. 

 We now address properties {\rm (ii)} and {\rm (iii)}.  For each $j=1,\ldots, d$, denote by $\pi_j$ the hyperplane $\lbrace x\in \R^d: \ x_j = 0\rbrace$. We use the coarea formula  (see \cite[Theorem 2.93]{Ambrosio-Fusco-Pallara:2000} with  $E = \partial^* T \cap \Omega$, $N=d-1$, $k=d-1$, $f(x) = (x_1,\ldots,x_{j-1},x_{j+1},x_d)$)   to find
\begin{equation}
\label{eq:co-f2}
\int_{\partial^*T \cap \Omega} |\langle \nu_T, {\rm e}_j\rangle| \, d\mathcal{H}^{d-1} = \int_{\pi_j} \mathcal{H}^{0}\big( (z + \R {\rm e}_j) \cap \partial^* T \cap \Omega\big)  \, d\mathcal{H}^{d-1}(z). 
\end{equation}
Similar identities hold for $\partial^*T_\eta\cap \Omega_\eta$ in place of $\partial^* T \cap \Omega$. The transformation formula yields for $j=1,\ldots,d-1$
\begin{align}
\int_{\pi_j} \mathcal{H}^{0}\big( (z + \R {\rm e}_j) \cap \partial^* T \cap \Omega\big)  \, d\mathcal{H}^{d-1}(z) = (\eta/\eps)^{d-2}\int_{\pi_j} \mathcal{H}^{0}\big( (  z  + \R {\rm e}_j) \cap \partial^* T_\eta \cap \Omega_\eta\big)  \, d\mathcal{H}^{d-1}(z),
\end{align}
and, in a similar fashion, we obtain for $j=d$
\begin{align}\label{eq: trafo2}
\int_{\pi_d} \mathcal{H}^{0}\big( (  z  + \R {\rm e}_d) \cap \partial^* T \cap \Omega\big)  \, d\mathcal{H}^{d-1}(  z  ) = (\eta/\eps)^{d-1}\int_{\pi_d} \mathcal{H}^{0}\big( (  z  + \R {\rm e}_d) \cap \partial^* T_\eta \cap \Omega_\eta\big)  \, d\mathcal{H}^{d-1}(  z  ).
\end{align}
Combining \eqref{eq:co-f2}--\eqref{eq: trafo2} we find
\begin{align*}
\int_{\partial^*T \cap \Omega} |\langle \nu_T, {\rm e}_j\rangle| \, d\mathcal{H}^{d-1} &=  (\eta/\eps)^{d-2}\int_{\partial^*T_\eta \cap \Omega_\eta} |\langle \nu_{T_\eta}, {\rm e}_j\rangle| \, d\mathcal{H}^{d-1} \ \ \ \ \text{for } j=1,\ldots,d-1,\notag\\
\int_{\partial^*T \cap \Omega} |\langle \nu_T, {\rm e}_d\rangle| \, d\mathcal{H}^{d-1} &=  (\eta/\eps)^{d-1}\int_{\partial^*T_\eta \cap \Omega_\eta} |\langle \nu_{T_\eta}, {\rm e}_d\rangle| \, d\mathcal{H}^{d-1}. 
\end{align*}
This along with \eqref{eq: perimterTdelta} yields 
\begin{align}\label{eq: trafo3}
\int_{\partial^*T \cap \Omega} |\langle \nu_T, {\rm e}_j\rangle| \, d\mathcal{H}^{d-1} &\le   (\eta/\eps)^{d-2} \mathcal{H}^{d-1}(\partial^* T_\eta \cap \Omega_\eta) \le c \frac{\eps}{\eta}\, E_{\ep,\eta}(y) \ \ \ \ \text{for } j=1,\ldots,d-1,\notag\\
\int_{\partial^*T \cap \Omega} |\langle \nu_T, {\rm e}_d\rangle| \, d\mathcal{H}^{d-1} & \le   (\eta/\eps)^{d-1}\mathcal{H}^{d-1}(\partial^* T_\eta \cap \Omega_\eta) \le c E_{\ep,\eta}(y). 
\end{align}
The first line in  \eqref{eq: trafo3} yields property {\rm (iii)}. To see {\rm (ii)}, we also use  \eqref{eq: trafo3} and $\eta \ge \eps$,   and we   compute
\begin{align*}
\mathcal{H}^{d-1}(\partial^* T \cap \Omega) \le \sum\nolimits _{j=1}^d \int_{\partial^* T \cap \Omega} | \langle\nu_T,  {\rm e}_j\rangle| \, d\mathcal{H}^{d-1}  \le c   (1+\eps/\eta)  E_{\ep,\eta}(y) \le c  E_{\ep,\eta}(y).
\end{align*}

 To prove {\rm (iv)}, we use the coarea formula (see \cite[Theorem 2.93]{Ambrosio-Fusco-Pallara:2000} with $E = \partial^* T \cap \Omega$, $f(x) = \langle x, {\rm e}_d \rangle$, $N=d-1$, $k=1$) to find
\begin{equation*}
\int_{\partial^*T \cap \Omega} \sqrt{1 - |\langle \nu_T, {\rm e}_d\rangle|^2} \, d\mathcal{H}^{d-1} = \int_{-\infty}^\infty \mathcal{H}^{d-2}\big( (\R^{d-1} \times \lbrace t \rbrace) \cap \partial^* T \cap \Omega\big)  \, dt. 
\end{equation*}
Consequently, {\rm (iv)} follows from property {\rm (iii)}.

\noindent\emph{Step III: Proof of \eqref{eq: propertiesT2}}.  First, define $Q^\eta= \lbrace (x',x_d): (\eta x'/\eps,x_d) \in Q \rbrace$ and $Q^\eta_l = \lbrace (x',x_d):  (\eta x'/\eps,x_d) \in Q_l \rbrace$ for $l \in \lbrace -n,\ldots,n-1\rbrace$. Note that $Q_l^\eta$ are identical cuboids and each of their sidelengths lies in $[h/n, 2h\eps/\eta]$. We apply the the relative isoperimetric inequality (see \cite[Theorem 2, Section 5.6.2]{EvansGariepy92}) on each $Q_l^\eta$ to find
$$ \min\lbrace \mathcal{L}^d(Q^\eta_l \cap T_\eta),   \mathcal{L}^d(Q^\eta_l \setminus T_\eta) \rbrace \le c\frac{h\eps}{\eta} \min\lbrace (\mathcal{L}^d(Q^\eta_l \cap T_\eta))^{\frac{d-1}{d}},   (\mathcal{L}^d(Q^\eta_l \setminus T_\eta))^{\frac{d-1}{d}} \rbrace \le c\frac{h\eps}{\eta} \mathcal{H}^{d-1} (\partial^* T_\eta\cap Q^\eta_l) , $$
where $c$ depends only on the dimension $d$.  (Note  that the theorem in the reference above is stated and proved in a ball, but that the argument only relies on Poincar\'e inequalities, and thus easily extends to bounded Lipschitz domains.) Summing over all $l$ and using \eqref{eq: perimterTdelta}
we get
$$\sum\nolimits_{l=-n}^{n-1} \min\lbrace \mathcal{L}^d(Q_l^\eta \cap T_\eta),   \mathcal{L}^d(Q_l^\eta \setminus T_\eta) \rbrace \le c\frac{h\eps}{\eta} \mathcal{H}^{d-1} (\partial^* T_\eta\cap \Omega_\eta) \le  c \frac{h\eps}{\eta}\,(\eps/\eta)^{d-1}  \, E_{\ep,\eta}(y). $$
This along with the fact that $ \mathcal{L}^d(Q^\eta_l \cap T_\eta) = (\eps/\eta)^{d-1}  \mathcal{L}^d(Q_l \cap T)  $ and $\mathcal{L}^d(Q^\eta_l \setminus T_\eta) = (\eps/\eta)^{d-1}  \mathcal{L}^d(Q_l \setminus T)$ yields \eqref{eq: propertiesT2} and concludes the proof.  
\end{proof}

 We point out that the results in Proposition \ref{lemma: phases} are sharp in terms of the scaling in $\eps$ and $\eta$. We refer to Remark \ref{rem: inclusions} for some explicit examples of $A$-phase regions with small $B$-phase  inclusions.

\subsection{Proof of  Theorem \ref{thm:rigiditythm}}\label{sec: rig-proof}  

 We now   prove   Theorem \ref{thm:rigiditythm}. \RRR In contrast to the preliminary results in the previous subsection, here we  use explicitly that  \RRR $A$ and $B$ are  rank-one connected. \EEE

\begin{proof}[Proof of Theorem \ref{thm:rigiditythm}]
We start with a preliminary observation concerning the phase regions $T$ and $\Omega \setminus T$ identified in Proposition \ref{lemma: phases}. Then we proceed with the proof of case (b) on a cube and address the case of general domains afterwards. Finally, we briefly indicate the necessary adaptions for case (a).

\noindent\emph{Step I: Phases.} Let $y \in H^2(\Omega;\R^d)$. Recall the definitions $A= {\rm Id}$ and $B = {\rm diag}(1,\ldots,1,1+\kappa)$, and the fact that this implies $|A-B| = \kappa$ and $\dist(SO(d)A,SO(d)B)=\kappa$.   We apply Proposition \ref{lemma: phases} to obtain a corresponding set of finite perimeter $T$. We   claim   that
\begin{align}\label{eq: derivative on T}
{\rm (i)}& \ \ \dist(\nabla y(x),SO(d)B) \le \Big( 1 +  \frac{1}{\alpha} \Big) \dist(\nabla y(x),SO(d)\{A,B\}) \ \ \ \ \text{for a.e.\ $x \in \Omega \setminus T$,} \notag \\
{\rm (ii)} & \ \  \dist(\nabla y(x),SO(d)A) \le  \frac{1}{1-\beta} \dist(\nabla y(x),SO(d)\{A,B\}) \ \ \ \ \text{for a.e.\ $x \in T$}
\end{align}
with $0<\alpha<\beta\le 1/2$ from Proposition \ref{lemma: phases}. 
First, by Proposition \ref{lemma: phases}{\rm (i)}, for a.e.\ $x\in \Omega \setminus T$ there holds 
\begin{equation*}
\dist(\nabla y(x),SO(d)A)\geq \alpha\kappa.
\end{equation*}
 Recalling that $|A-B| = \kappa$  we find 
 \begin{align*}
 &\dist(\nabla y(x),SO(d)B) \le \dist(\nabla y(x),SO(d)A)+\kappa\le \Big(1+\frac{1}{\alpha}\Big)\dist(\nabla y(x),SO(d)A).
 \end{align*}
 This yields \eqref{eq: derivative on T}{\rm (i)}. Analogously, for a.e.\ $x \in T$, by Proposition \ref{lemma: phases}{\rm (i)} we get $\dist(\nabla y(x),SO(d)A)\leq \beta\kappa$. As $\dist(SO(d)A,SO(d)B) = \kappa$, we obtain  
 $$ \dist(\nabla y(x),SO(d)B) { \ge } (1-\beta)\kappa$$
 for a.e.\ $x\in T$, and hence
 $$\dist(\nabla y(x),SO(d)A) {  \le } \beta\kappa \le \kappa \le    (1-\beta)^{-1} \dist(\nabla y(x),SO(d)B)$$
 for a.e. $x\in T$. This yields \eqref{eq: derivative on T}{\rm (ii)}.


%

 \noindent\emph{Step II: Proof of (b) for cubes.} We first treat the case   in which   $\Omega = x_0  + (-h,h)^d$ is a cube. The main idea is to replace $\nabla y$ by a suitable incompatible vector field $\gamma$ with $\gamma \approx SO(d)A$ and then to apply Lemma \ref{lemma: Muller-Chambolle}. It turns out that one also needs to define $\gamma$ on an appropriately scaled version of $\Omega$  in order to control the curl of $\gamma$.

 Our starting point is  \eqref{eq: propertiesT2} applied for $Q = \Omega$: we find a decomposition $(Q_l)_{l=-n}^{n-1}$ of $\Omega$ with  $n = \lfloor  \eta/\eps\rfloor$.   We choose $M_l = A$ if $\mathcal{L}^d(Q_l \setminus T) \le \mathcal{L}^d(Q_l \cap T)$ and $M_l = B$ otherwise, i.e., $M_l$ indicates the predominant phase in each cuboid $Q_l$. By \eqref{eq: propertiesT2} this implies 
 \begin{align}\label{eq: propertiesT2-application}
\sum\nolimits_{l: M_l = A}  \mathcal{L}^d(Q_l \setminus T) + \sum\nolimits_{l: M_l = B}    \mathcal{L}^d(Q_l \cap T)  \le c \frac{\eps}{\eta}\, E_{\ep,\eta}(y),
\end{align}  
where $c$ depends on $h$ and thus on  $\Omega$. Let $\Psi \in H^1(\Omega; \R^d)$ be a homeomorphism with $\nabla \Psi = M_l$ on each $Q_l$. We let $U =\Psi(\Omega)$ and note that $U$ is a paraxial cuboid. In the following, we will use the notation $\bar{x} = \Psi(x)$ for $x \in \Omega$.  We also define $U_l = \Psi(Q_l)$ for all $l \in \lbrace - n, \ldots, n -1 \rbrace$.


 We consider the vector field $\gamma \in L^2(U; \M^{d \times d})$  defined by 
\begin{align}\label{eq: gamma}
\gamma := \big(\nabla y \chi_{T}   + \nabla y B^{-1}\chi_{\Omega \setminus T} \big) \circ  \Psi^{-1}. 
\end{align} 
 In view of \eqref{eq: gamma} and the fact that $\nabla \Psi^{-1} = M_l^{-1}$ on $U_l$, we obtain by the transformation formula
\begin{align}
\label{eq:est-dist1}
\Vert \dist(\gamma,SO(d)A) \Vert_{L^2(U)}^2 \le  C\int_T \dist^2(\nabla  y, SO(d)A)\,dx+ C\int_{\Omega\setminus T}\dist^2(\nabla y, SO(d)B)\,dx,
\end{align}
where $C$ only depends on $\kappa$.   Using the definition of the energy $E_{\ep,\eta}$ (see \eqref{eq: nonlinear energy}) and H4., by combining \eqref{eq: derivative on T}   and \eqref{eq:est-dist1} we conclude that
\begin{align}\label{eq: energy-v}
\Vert \dist(\gamma,SO(d)A) \Vert_{L^2(U)}^2 \le C\ep^2 E_{\ep,\eta}(y),
\end{align}
 where $C=C(c_1,\kappa)$.

 Our goal is to apply Lemma \ref{lemma: curl} and therefore we first check that $\gamma \in SBV(U;\M^{d \times d})$. As $y \in H^2(\Omega;\R^d)$, the jump set $J_\gamma$ of $\gamma$ is contained in $\lbrace \bar{x} \in U: \ \Psi^{-1}(\bar{x}) \in \partial^* T \cap \Omega \rbrace$. Without restriction, we choose the normal $\nu_\gamma$ to the jump set such that  $\nu_\gamma(\bar{x}) = \nu_T(\Psi^{-1}(\bar{x}))$ for $\mathcal{H}^{d-1}$-a.e.\ $\bar{x} \in J_\gamma$, where $\nu_T$ denotes the outer normal to $T$. An elementary calculation yields
   \begin{align}\label{eq: uniform-bound-der}
 [\gamma](\Psi(x)) = \nabla y(x) B^{-1} - \nabla y(x) A = \frac{-\kappa}{1+\kappa} \nabla y(x)  \,    { {\rm e}_{dd}} = \frac{-\kappa}{1+\kappa}  \, \partial_d y(x) \otimes { {\rm e}_d} \in \M^{d \times d}
 \end{align} 
 for $x \in \partial^* T \cap \Omega$, where $\nabla y$ on $\partial^*T \cap \Omega$ has to be understood in the sense of traces, see \cite[Theorem 3.77]{Ambrosio-Fusco-Pallara:2000}.  By  \eqref{eq: propertiesT}(i) we find $|\nabla y(x)| \le c$ for $\mathcal{L}^d$-a.e.\ $x \in T$ for a  constant $c>0$ only depending on the dimension. Therefore,  \cite[Theorem 3.77]{Ambrosio-Fusco-Pallara:2000} yields
\begin{align}\label{eq: linftysurf}
|\nabla y(x)| \le c \ \ \ \text{ for $\mathcal{H}^{d-1}$-a.e.\ $x \in \partial^* T \cap \Omega$}.
\end{align}
This along with \eqref{eq: uniform-bound-der} shows $|[\gamma](\bar{x})| \le c$ for  $\mathcal{H}^{d-1}$-a.e.\ $\bar{x} \in J_\gamma$
 and then  \cite[Theorem 3.84]{Ambrosio-Fusco-Pallara:2000} implies that $\gamma \in SBV(U;\M^{d \times d})$.

We now determine $\curl \gamma$. We first address the bulk term. The main observation is that on  each   $U_l$   the vector field $\gamma$ defined in \eqref{eq: gamma}  can be written as the sum of  a gradient and  a small perturbation. More precisely, an elementary computation shows 
\begin{align*}
\gamma = 
\nabla (y\circ \Psi^{-1}) (\chi_{T} \circ \Psi^{-1})  + \nabla (y\circ \Psi^{-1}) B^{-1} (\chi_{\Omega \setminus T} \circ \Psi^{-1})  = \nabla (y\circ \Psi^{-1}) + z_A (\chi_{\Omega \setminus T} \circ \Psi^{-1})
\end{align*}
on $U_l$ with $M_l = A$, where 
\begin{align*}
z_A: =  \nabla (y\circ \Psi^{-1}) (B^{-1}-A) = \frac{-\kappa}{1+\kappa} (\nabla y \circ \Psi^{-1})  \,     { {\rm e}_{dd}}  =  \frac{-\kappa}{1+\kappa} (\partial_d y \circ \Psi^{-1})  \otimes e_d. 
\end{align*}
In a similar fashion, we have 
\begin{align*}
\gamma =
\nabla (y\circ \Psi^{-1})B (\chi_{T} \circ \Psi^{-1})  + \nabla (y\circ \Psi^{-1})  (\chi_{\Omega \setminus T} \circ \Psi^{-1})  = \nabla (y\circ \Psi^{-1}) + z_B (\chi_{T} \circ \Psi^{-1}),
\end{align*}
on $U_l$ with $M_l = B$, where  $z_B := \kappa (\partial_d y \circ \Psi^{-1})  \otimes e_d = -(1+\kappa)z_A$.

On each $U_l$ with $M_l = A$, we compute using the transformation formula and  H\"older's inequality
\begin{align*}
\sum\nolimits_{i,j,k=1}^d \int_{U_l}|\partial_i \gamma_{kj} - \partial_j \gamma_{ki}| \,d\bar{x} &\le C \sum\nolimits_{i,j, k=1}^d \int_{Q_l \setminus T} |\delta_{dj}\, \partial^2_{ij} y_k  - \delta_{di}\, \partial^2_{ji}  y_k |  \, dx\\
& \le C  \sum\nolimits_{i=1}^{d-1}  \int_{Q_l \setminus T} |\partial^2_{id} y| \, dx   \le C   (\mathcal{L}^d(Q_l \setminus T))^{1/2} \sum\nolimits_{i=1}^{d-1} \Vert \partial^2_{id} y \Vert_{L^2(Q_l)}, 
\end{align*}
where $\delta_{id}$ denotes the Kronecker delta.  Similarly, on each $U_l$ with $M_l = B$, we deduce
\begin{align*}
\sum\nolimits_{i,j,k=1}^d \int_{U_l}|\partial_i \gamma_{kj} - \partial_j \gamma_{ki}| \,  d\bar{x} \le C \sum\nolimits_{i=1}^{d-1} \int_{Q_l \cap  T} |\partial^2_{id} y| \, dx \le C (\mathcal{L}^d(Q_l \cap T))^{1/2}\sum\nolimits_{i=1}^{d-1} \Vert \partial^2_{id} y \Vert_{L^2(Q_l)}.
\end{align*}
Then, taking the sum over all $l$, and using \eqref{eq: nonlinear energy}, \eqref{eq: propertiesT2-application}, as well as the discrete H\"older inequality we get
\begin{align}\label{eq: bulk curl}
\sum\nolimits_{i,j,k=1}^d \int_{U}|\partial_i \gamma_{kj} - \partial_j \gamma_{ki}| \,d\bar{x} \le C \Big(\frac{\eps}{\eta}E_{\eps, \eta}(y)\Big)^{1/2} \sum\nolimits_{i=1}^{d-1} \Vert \partial^2_{id} y \Vert_{L^2(\Omega)} \le C \eps^{1/2}\eta^{-3/2}  \,   E_{\eps,\eta}(y). 
\end{align}

We now estimate the surface part of  $\curl\gamma$. In view of \eqref{eq: uniform-bound-der}--\eqref{eq: linftysurf}  and the fact that $\nu_\gamma = \nu_T \circ \Psi^{-1}$, denoting by $[\gamma]_k$ the $k$-th row of $[\gamma]$, we obtain
\begin{align*}
\big|\big([\gamma]_k   \otimes \nu_\gamma -    \nu_\gamma \otimes [\gamma]_k \big) \circ \Psi\big|  = \frac{\kappa}{1+\kappa} |\partial_d y_k ({ {\rm e}_d}\otimes \nu_T-\nu_T\otimes { {\rm e}_d})|\le c \kappa  |{ {\rm e}_d}\otimes \nu_T-\nu_T\otimes { {\rm e}_d}|
\end{align*}
$\mathcal{H}^{d-1}$-a.e.\ on $\partial^* T \cap \Omega$   for every $k=1,\dots,d$,   where $c$ is the constant of \eqref{eq: linftysurf}.  This then implies by Proposition \ref{lemma: phases}{\rm (iii)} that for every $k=1,\dots,d$
\begin{align}\label{eq: surface curl}
\int_{J_\gamma}\big| [\gamma]_k  \otimes \nu_\gamma -    \nu_\gamma \otimes  [\gamma]_k\big|\, d\mathcal{H}^{d-1} \le C \sum_{j=1}^{d-1} \int_{\partial^* T \cap \Omega}|\langle \nu_T, {\rm e}_j \rangle| \, d\mathcal{H}^{d-1} \le C\frac{\eps}{\eta} \,  E_{\ep,\eta}(y).
\end{align}
Consequently,  Lemma \ref{lemma: curl} (applied on each row of the vector field $\gamma$) and  \eqref{eq: bulk curl}--\eqref{eq: surface curl} yield
\begin{align}\label{eq: curl gamma}
|\curl \gamma|(U)   \le  C \eps^{1/2}\eta^{-3/2}  \,   E_{\eps,\eta}(y) + C \eps \eta^{-1}\,  E_{\ep,\eta}(y) 
\end{align}
for $C=C(\Omega, \kappa, d, c_1)$. Consider a smaller cube $\Omega' \subset \subset \Omega$ and let $U' = \Psi(\Omega')$. Let $1 \le p \le 2$ with $p \neq \frac{d}{d-1}$. From Lemma  \ref{lemma: Muller-Chambolle}(b) we then get a rotation  $R \in SO(d)$ such that by \eqref{eq: energy-v},  \eqref{eq: curl gamma}, and H\"older's inequality 
\begin{align}\label{eq: rig-appli}
  \Vert\gamma-R\Vert_{L^p(U')}  & \leq C \Big(\Vert \dist(\gamma,SO(d)A) \Vert_{L^2(U)} +  (|\curl \gamma|(U))^{r(p,d)}   \Big)\notag \\
& \le C\ep \sqrt{E_{\ep,\eta}(y)} +   C\Big( \eps^{1/2}\eta^{-3/2}  \,     E_{\eps,\eta}(y) +  \eps \eta^{-1}\,  E_{\ep,\eta}(y)\Big)^{r(p,d)},   
\end{align}
where the constant also depends on $\Omega$, $\Omega'$, and $p$.  Let $\mathcal{M} \in BV(\Omega;\lbrace A,B\rbrace)$
 be the function defined by $\mathcal{M} = A \chi_T + B \chi_{\Omega \setminus T}$. Clearly, 
 $$|D\mathcal{M}|(\Omega) \le |A-B|\mathcal{H}^{d-1}(\partial^* T \cap \Omega) \le CE_{\ep,\eta}(y)$$ 
 by Proposition \ref{lemma: phases}{\rm (ii)}.  Recalling \eqref{eq: gamma} we compute, again using the transformation formula 
 \begin{align}\label{eq: last esti}
 \Vert \nabla y - R\mathcal{M} \Vert_{L^p(\Omega')} &=  \Vert \nabla y - R \Vert_{L^p(\Omega' \cap T)} +  \Vert \nabla yB^{-1} - R \Vert_{L^p(\Omega' \setminus T)}  \le C\Vert\gamma-R\Vert_{L^p(U')}.
 \end{align}
 This along with \eqref{eq: rig-appli} shows \eqref{eq: rigidity-new}. We conclude this part of the proof by mentioning that, taking also  Remark \ref{rk:sets-for-lemma} into account, the passage to the subcube $\Omega'$ is actually not necessary. This in turn yields Remark \ref{rem: afterth}(ii).

  \noindent\emph{Step III: Proof of (b) for general domains.}  We perform a covering argument exactly as in the proof of Lemma \ref{lemma: Muller-Chambolle}: given $\Omega' \subset \subset \Omega$, we cover $\Omega'$ with a finite number of paraxial cubes $\lbrace Q_i \rbrace_{i=1}^N$ such that smaller cubes $Q_i' \subset \subset Q_i$ still cover $\Omega'$. We apply Step II on each $Q_i$ and obtain an estimate of the form \eqref{eq: last esti} on each $Q_i'$ with a rotation $R_i$. The difference of the rotations can be controlled as explained in the proof  of Lemma \ref{lemma: Muller-Chambolle}.

    \noindent\emph{Step IV: Proof of (a).} The essential difference is that we do not apply \eqref{eq: propertiesT2}   to obtain a decomposition of $\Omega$ with property \eqref{eq: propertiesT2-application}. However, we define a decomposition into \RRR strips $(Q_l)_l$  of height  approximately  $\eps/\eta$, which are in general not rectangular, but whose upper and lower boundaries are parallel to $e_1$. We \EEE set $M_l = A$ if $\mathcal{L}^d(Q_l \setminus T) \le \mathcal{L}^d(Q_l \cap T)$ and $M_l = B$ otherwise, and observe that \eqref{eq: propertiesT2-application} follows from   \eqref{eq: propertiesT}(iv) (see Figure \ref{fig:small-trans}). The rest of the argument remains unchanged with the only difference that we use part (a) of  Lemma \ref{lemma: Muller-Chambolle} instead of part (b). 
  \end{proof}

\begin{remark}\label{rem: setT}
For later purposes, we note that   by   the construction of the phase indicator $\mathcal{M}$ in the previous proof, the set $\lbrace \mathcal{M} = A \rbrace$ coincides with the set $T$ considered in Proposition \ref{lemma: phases}. 
\end{remark}

\begin{remark}[Examples  of minority islands and their sharpness]\label{rem: inclusions}
We provide prototype configurations with a small $B$-phase region completely contained in the $A$-phase region. These illustrate sharpness of the estimates in  Proposition   \ref{lemma: phases}. We follow the $2d$-example in \cite[Remark 6.1]{conti.schweizer2} and take the \RRR opportunity \EEE to present a $d$-dimensional analog here.

Let $\Omega = (-2,2)^d$ and let $0<r<1$. Consider the polyhedron $P$ consisting of the vertices ${\rm e}_d$, $-r {\rm e}_d$, and  $(x',0)$, $x' \in \lbrace -1,1 \rbrace^{d-1}$. By $\mathcal{F}$ we denote the $2(d-1)$  faces of dimension $(d-2)$  in $[-1,1]^{d-1} \times \lbrace 0 \rbrace$ obtained by setting one of the first $(d-1)$ components equal to $\pm 1$. Observe that the polyhedron $P$ consists of $4(d-1)$ convex polyhedra with $2^{d-2}+2$ vertices each:  $2(d-1)$ polyhedra with vertex in  $0$,  vertex in ${\rm e}_d$,  and the $2^{d-2}$ vertices  of a face in  $\mathcal{F}$ (we denote their union by $P^1$), as well as  $2(d-1)$  polyhedra with vertex in  $0$, vertex in $-r {\rm e}_d$, and the $2^{d-2}$ vertices  of a face in  $\mathcal{F}$ (we denote their union by $P^2$). See Figure \ref{fig:3D-polihedron} for an illustration in dimension 3. Observe that $\mathcal{L}^d(P) \le c$ and $\mathcal{L}^d(P^2) \ge  cr$ for a dimensional constant $c>0$. 

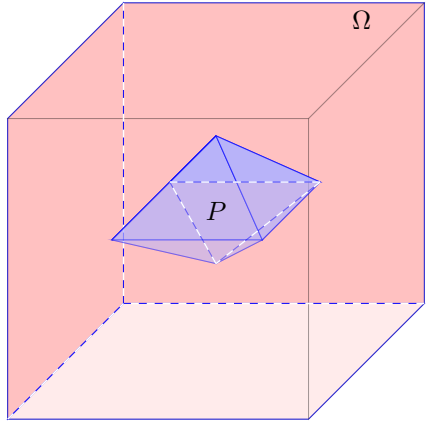
\begin{figure}[h]
\centering
\begin{tikzpicture}
\coordinate (H) at (-2,-2,-2);
\coordinate (A) at (-2,-2,2);
\coordinate (B) at (2,-2,2);
\coordinate (C) at (2,-2,-2);
\coordinate (D) at (-2, 2,-2);
\coordinate (E) at (-2, 2,2);
\coordinate (F) at (2, 2, 2);
\coordinate (G) at (2, 2,-2);

\coordinate (P) at (0,0,0); 

\coordinate (I) at (0,1,0);
\coordinate (Q) at (0,-0.7,0); 
\coordinate (L) at (-1,0,-1);
\coordinate (M) at (-1,0,1);
\coordinate (N) at (1,0,1);
\coordinate (R) at (1,0,-1);

\draw[blue, fill=red!30, opacity=0.8] (H) -- (C) -- (G) -- (D) -- cycle;
\draw[blue,fill=red!30, opacity=0.8] (H) -- (A) -- (E) -- (D) -- cycle;
\draw[blue,fill=red!10, opacity=0.8] (H) -- (A) -- (B) -- (C) -- cycle;
\draw[opacity=0.2] (D) -- (E) -- (F) -- (G) -- cycle;
\draw[opacity=0.2] (C) -- (B) -- (F) -- (G) -- cycle;
\draw[opacity=0.2] (A) -- (B) -- (F) -- (E) -- cycle;
\draw[white!, dashed, line width=0.2mm] (C)--(H);
\draw[white!, dashed, line width=0.2mm] (D)--(H);
\draw[white!, dashed, line width=0.2mm] (A)--(H);

\draw[blue,fill=blue!30, opacity=0.8] (R) -- (L) -- (I) -- cycle;
\draw[blue,fill=blue!30, opacity=0.6] (M) -- (N) -- (I) -- cycle;
\draw[blue,fill=blue!30, opacity=0.6] (R) -- (N) -- (I) -- cycle;
\draw[blue,fill=blue!30, opacity=0.6] (M) -- (L) -- (I) -- cycle;
\draw[blue,fill=blue!30, opacity=0.3] (R) -- (L) -- (Q) -- cycle;
\draw[blue,fill=blue!30, opacity=0.3] (M) -- (N) -- (Q) -- cycle;
\draw[blue,fill=blue!30, opacity=0.3] (R) -- (N) -- (Q) -- cycle;
\draw[blue,fill=blue!30, opacity=0.3] (M) -- (L) -- (Q) -- cycle;
\draw[white!, dashed, line width=0.2mm] (R)--(L);
\draw[white!, dashed, line width=0.2mm] (Q)--(L);
\draw[white!, dashed, line width=0.2mm] (R)--(Q);
\node at (P) {$P$};
\node at (1.4,2,-1.4) {$\Omega$};
\end{tikzpicture}
\caption{The set $\Omega$ and the polyhedron $P$.}
\label{fig:3D-polihedron}
\end{figure}

Set $u= 0$ outside $P$. At the origin we set $u(0) = \kappa r {\rm e}_d$ and let $u$ be affine on each of the $4(d-1)$ polyhedra contained in $P$. Define $v = {\rm id}+ u \in H^1(\R^d ; \R^d)$. In view of $B = A + \kappa {\rm e}_{dd} = {\rm Id} + \kappa {\rm e}_{dd}$, this implies $|\nabla v - A| \le cr$ on $P^1 \cup (\R^d \setminus P)$ and  $|\nabla v - B| \le c r$ on $P^2$, where $c=c(d,\kappa)>0$. In particular, for $r$ small enough we find $T = \Omega \setminus P^2$ with $T$ from Proposition \ref{lemma: phases}. In view of $\mathcal{L}^d(P^2) \ge  cr$, a short calculation yields (assuming that $W$ is smooth)
$$\int_{\R^d } W(\nabla v)\, dx \le cr^2, \ \ \ \ \ \ \min_{F \in SO(d)\lbrace A,B\rbrace}\int_{\Omega} |\nabla v - F |^2\,dx  \ge cr.$$
We now mollify $v$. To this end, denoting by $[\nabla v]$ the jump  of the gradient, we observe that 
\begin{align*}
x \in \partial P^1 \setminus \partial P^2: & \ \ |[\nabla v](x)| \le cr, \\
x \in   \partial P^2: & \ \ |[\nabla v](x){\rm e}_{dd}| \le c, \ \ \ \ |[\nabla v](x)e| \le cr \text{ for all $e \in \lbrace {\rm e}_{ij}: i,j =1,\ldots,d \rbrace \setminus \lbrace {\rm e}_{dd} \rbrace$}.
\end{align*}
We define $y = v * \rho_{\eps^2} \in H^2(\Omega;\R^d)$, where $\rho_{\eps^2}$ is a mollification kernel on the scale $\eps^2$. 
After some calculations we obtain
\begin{align}\label{eq: rig-en}
\int_{\Omega} W(\nabla y) \, dx \le c(r^2+\eps^2), \ \ \ \ \ \int_{\Omega}|\nabla^2 y|^2\,dx \le c\eps^{-2}, \ \ \ \ \  \int_{\Omega}\Big(  |\nabla^2 y|^2 - |\partial^2_{dd}y|^2 \Big)  \,dx \le cr^2\eps^{-2}
\end{align}
and 
\begin{align}\label{eq: rig-sharp}
\min_{F \in SO(d)\lbrace A,B\rbrace}\int_{\Omega} |\nabla y - F |^2\,dx  \ge cr - C\eps^2 
\end{align}
for some   $C=C(d,\kappa)>0$ sufficiently large. Therefore, recalling \eqref{eq: nonlinear energy} and using \eqref{eq: rig-en} we observe 
$$E_{\eps,\eta}(y) \le c + cr^2 \eps^{-2}(1+\eta^2)$$
 which is uniformly controlled in $\eps$ when $r(1+\eta) \le c\eps$. Thus, for all $0 \le \eta \le 1$ the critical scaling for $r$ is $r \sim \eps$. Observe that \eqref{eq: rig-sharp} (for $r = \eps$) shows that the estimate \eqref{eq: rig-mot2} obtained  in \cite{conti.schweizer} is sharp. (The model considered there corresponds to the case $\eta=0$.)

On the other hand, for $\eta >1$, in order to  have bounded energy, the critical scaling  for $r$ is $r \sim \eps/\eta$. Note that in this regime we find  $\int_{-2}^2\mathcal{H}^{d-2}\big( (\R^{d-1} \times \lbrace t \rbrace) \cap \partial^* P^2 \cap \Omega\big)  \, dt \ge c r  \sim c\eps /\eta$, which illustrates the sharpness of estimate \eqref{eq: propertiesT}{\rm (iv)}. We also mention that \eqref{eq: rig-sharp} shows that  the  scaling in an estimate of the form \eqref{eq: rig-mot2} (for the model considered in \eqref{eq: nonlinear energy}) cannot be better than
$$\Vert \nabla y - R M \Vert_{L^2(\Omega')} \le C\sqrt{\ep/\eta}.$$
Thus, for all $\eta \ll \frac{1}{\eps}$, introducing a phase indicator is indispensable to obtain the $\eps$-scaling in Theorem \ref{thm:rig-intro}. (Recalling the discussion in \eqref{eq:lin-formal}, the choice  $\eta \ll \frac{1}{\eps}$ is essential to ensure that our perturbed model has the same qualitative behavior as the unperturbed problem \eqref{eq:sol-sol}, at least asymptotically when passing to a linearized strain regime.)

\end{remark}

\begin{remark}[Necessity of the curl estimates for $p=2$]
\label{rem:curl-yes}
 Our fine estimates on the curl of incompatible vector fields are necessary in order to obtain the rigidity estimate in any dimension $d\in \mathbb{N}$, $d\geq 2$, for $p=2$, see Theorem \ref{thm:rig-intro}. Indeed, without passing to incompatible fields, by combining directly  Proposition \ref{lemma: phases} with an argument along the lines of \eqref{eq:no-curl}, one can show that an inequality of the form
\begin{equation}
\label{eq:why-curl}
\|\nabla y-R\mathcal{M}\|_{L^p(\Omega)}\leq C\Big(\ep \sqrt{ E_{\ep,\eta}(y)} +\Big(\frac{\ep}{\eta}E_{\ep,\eta}(y)\Big)^{\frac1p}\Big)
\end{equation}
 holds. For a map $y$ with bounded energy, this provides the rigidity estimate
$$\|\nabla y-R\mathcal{M}\|_{L^p(\Omega)}\leq C\ep,$$
only if $\eta\geq \ep^{1-p}$. As highlighted in the discussion above Theorem \ref{thm:rig-intro}, see \eqref{eq:lin-formal}, it is necessary to impose that $\eta \ll \frac1\ep$.
For $p<2$, estimate \eqref{eq:why-curl} would still allow to guarantee $\eta \ll \frac1\ep$, although in general providing a less sharp estimate on $\eta$ compared to the one of Theorem \ref{thm:rigiditythm}. For $p=2$, \eqref{eq:why-curl} would lead to consider $\eta\geq \frac1\ep$, which would modify the qualitative behavior of the model. \end{remark}


\section{Solid-solid phase transitions}\label{sec:applications}

 In this section we present an application of the quantitative  two-well  rigidity estimate proved in Theorem \ref{thm:rigiditythm} to the theory of solid-solid phase transitions.  We start by recalling the literature  representing  the departure point of our analysis (see Subsection \ref{sec: review-transition}) and then present a sharp-interface limit for energies of the form \eqref{eq: nonlinear energy} as $\eps$ tends to zero (see Subsection \ref{sec: main results}).  Subsection \ref{sec: proof1}, Subsection \ref{sec: cell},  and Subsection \ref{sec: proof2} contain the proofs of our results.  
 
 In the following let $d\in \mathbb{N}$, $d\geq 2$,  and let $\Omega \subset \R^d$ be  a bounded Lipschitz domain. We consider the energy functionals defined in   \eqref{eq: nonlinear energy}, with stored-energy densities $W : \M^{d \times d} \to [0, + \infty)$ satisfying  H1.--H4.\ and additionally 
 \begin{itemize}
\item[H5.] (Growth condition from above) there exists a constant $c_2>0$ such that 
$$  W(F) \le c_2 \dist^2(F,SO(d)\lbrace A,B \rbrace)  \ \ \ \text{for every $F\in \M^{d\times d}$.}$$ 
\end{itemize}

\subsection{A sharp-interface limit for a model of solid-solid phase transitions}\label{sec: review-transition}

A standard singularly perturbed two-well problem takes the form
\begin{align}\label{eq: I functional}
I_{\ep}(y):=\frac{1}{\ep^2}\int_{\Omega}W(\nabla y)\,dx+\ep^2\int_{\Omega}|\nabla^2 y|^2\,dx
\end{align}
for every $y\in H^2(\Omega;\R^d)$. This corresponds to the choice $\eta = 0$ in \eqref{eq: nonlinear energy}. The restriction of the functional to a subset $\Omega' \subset \Omega$ will be denoted by $I_\eps(y,\Omega')$.  In this subsection, we recall the results obtained by {\sc S.~Conti}  and  {\sc B.~Schweizer} \cite{conti.schweizer} about the sharp-interface limit of this model as $\eps$ tends to zero. We again concentrate on compatible wells with exactly one rank-one connection (see assumption H3.), but mention that in \cite{conti.schweizer} also the case of two rank-one connections is addressed.

 Denote by   $\mathcal{Y}(\Omega)$  the class of admissible limiting deformations, defined as 
\begin{align}\label{eq: limiting deformations}
\mathcal{Y}(\Omega):= \bigcup_{R\in SO(d)} \mathcal{Y}_R(\Omega),   \ \ \text{where}  \ \  \mathcal{Y}_R(\Omega):=\big\{& y\in H^1(\Omega;\R^d):\,\nabla y\in BV(\Omega;R\{A,B\})\big\} \ \  \text{for } R \in SO(d).
\end{align}
Analogously, for every open subset $\Omega'\subset \Omega$,  let  $\mathcal{Y}(\Omega')$   be    the corresponding set of admissible deformations on $\Omega'$. The following compactness result has been proven in \cite[Proposition 3.2]{conti.schweizer}.

\begin{lemma}[Compactness]
\label{lemma:comp-def}
Let $d\in \mathbb{N}$, $d \ge 2$, and let $\Omega\subset \R^d$ be a bounded Lipschitz domain. Let $W$ satisfy assumptions {\rm H1}.--{\rm H4}. Then, for all sequences $\{y^{\ep}\}_\eps\subset H^2(\Omega;\R^d)$ for which 
$$\sup_{\ep>0} I_{\ep}(y^{\ep})<+\infty,$$ 
there exists a map $y\in  \mathcal{Y}(\Omega)$ such that, up to the extraction of a (non-relabeled) subsequence, there holds
$$y^{\ep}-  \frac{1}{\mathcal{L}^d(\Omega)} \int_{\Omega}y^{\ep}(x)\,dx \to y\quad\text{strongly in }H^1(\Omega;\R^d).$$
\end{lemma}
Here and in the sequel, we follow the usual convention that convergence of the continuous parameter $\eps \to 0$ stands for convergence of arbitrary sequences $\lbrace \eps_i \rbrace_i$ with $\eps_i \to 0$ as $i \to \infty$, see \cite[Definition 1.45]{Braides:02}. The limiting deformations $y$ have the  structure of a simple laminate. Indeed, {\sc G.~Dolzmann}  and  {\sc S.~M\"uller} \cite{dolzmann.muller} have shown that for  $y \in \mathcal{Y}_R(\Omega)$ the  essential boundary of the set $T:=\{x\in \Omega:\,\nabla y(x)\in RA\}$ consists of subsets of hyperplanes that intersect $\partial \Omega$ and are orthogonal to ${\rm e}_d$, and  that  $y$ is affine on balls whose intersection with $\partial T$ has zero $\mathcal{H}^{d-1}$-measure.

We now introduce the limiting sharp-interface energy. We denote by $Q =(-\frac{1}{2},\frac{1}{2})^d$ the $d$-dimensional unit cube centered in the origin and with sides parallel to the coordinate axes. Consider the \emph{optimal-profile energy}
\begin{align}\label{eq: conti-schweizer-k}
 K_{0}  :=\inf\Big\{&\liminf_{\ep\to 0} I_{\ep}(y^{\ep},Q): \   \lim_{\eps \to 0}  \Vert  y^\eps -  \RRR y_0^+\EEE \Vert_{L^1(Q)}= 0\Big\},
\end{align}
where $\RRR y_0^+\EEE \in H^1_{ \rm loc }(\R^d;\R^d)$ is the continuous function with $\RRR y_0^+\EEE(0) = 0$ and 
\begin{align}\label{eq: conti-schweizer-k-y0}
\nabla \RRR y_0^+\EEE=A\chi_{\{x_d>0\}}+B\chi_{\{x_d<0\}}.
\end{align} 
The parameter $K_0$ represents the energy of an \emph{optimal profile} transitioning from phase $A$ to $B$. We point out that  $K_0$ is  invariant   under reflection of the two phases $A$ and $B$, i.e., one could replace $\RRR y_0^+\EEE$ in \eqref{eq: conti-schweizer-k} by a continuous function with gradient $B\chi_{\lbrace x_d >0\rbrace} + A\chi_{\lbrace x_d < 0\rbrace}$. Let $I_0:L^1(\Omega;\R^d)\to [0,+\infty]$ be the functional
$$
I_0(y):=\begin{cases}K_0 \mathcal{H}^{d-1}(J_{\nabla y}) &\text{if }y\in  \mathcal{Y}(\Omega),\\
+\infty&\text{otherwise}.\end{cases}
$$
The following characterization of $I_0$ by means of $\Gamma$-convergence has been proven in \cite[Theorem 3.1]{conti.schweizer} in the two-dimensional setting. For an exhaustive treatment of $\Gamma$-convergence we refer the reader to \cite{Braides:02, DalMaso:93}. 
\begin{theorem}[$\Gamma$-convergence]\label{theorem: conti.schweizer}
Let $d=2$, let $\Omega\subset \R^2$ be a bounded,  strictly star-shaped  Lipschitz domain, and let $W$ satisfy {\rm H1}.--{\rm H5}. Then
$$\Gamma-\lim_{\ep \to 0} I_{\ep}=I_0$$ 
with respect to the strong $L^1$-topology.
 \end{theorem}
  We recall that an open set $\Omega$ is strictly star-shaped if there exists a point $x_0\in \Omega$ such that 
  \begin{align*}
  \{tx+(1-t)x_0:\,t\in (0,1)\}\subset \Omega \ \ \ \ \text{for every $x\in \partial \Omega$}.
  \end{align*}
This assumption on the geometry of $\Omega$ simplifies the construction of recovery sequences . We refer to \cite{conti.fonseca.leoni} for a related problem where more general domains are considered. We point out that assumption H5.\ is not compatible with the impenetrability condition
\begin{align*}
W(F)\to +\infty\quad\text{as}\quad {\rm det}\,F\to 0^+,\ \ \ \ \ \ \ \ \ \ \ \ W(F)=+\infty\quad \text{if}\quad {\rm det}\,F \le 0, 
\end{align*}
 which is usually enforced to model a blow-up of the elastic energy under strong compressions. Assumption H5.\ is not required for the proof of the liminf inequality in Theorem \ref{theorem: conti.schweizer}, but is instrumental for the construction of recovery sequences. We note that, by means of a more elaborated construction performed in \cite{conti.schweizer3}, assumption H5.\ may be dropped.    

The above result is limited to the two-dimensional setting due to the limsup inequality: the definition
of sequences with optimal energy approximating a limit that has multiple flat interfaces relies on a deep technical construction. This so-called \emph{$H^{1/2}$-rigidity on lines} (see \cite[Section 3.3]{conti.schweizer}) is only available in dimension $d=2$. We overcome this issue for our model \eqref{eq: nonlinear energy} by means of the rigidity estimate proven in Section \ref{sec: rigidity estimate}.

\subsection{The limiting sharp-interface model in the present setting}\label{sec: main results}

In this subsection we describe our limiting sharp-interface model and present our main $\Gamma$-convergence result. Consider  the   energy   functionals  defined in \eqref{eq: nonlinear energy}, under the choice 
\begin{equation}
\label{eq:eta-ep}
\eta=\eta_{\ep,d}=\ep^{\frac{(r_d-2)}{3r_d}},
\end{equation} 
where $r_d:=\min\{1,\frac{d^2}{2(d-1)^2}\}$.

We point out that 
\begin{equation}
\label{eq:qdrd}
r_d=r(p_d,d),
\end{equation} where $r(\cdot,d)$ is the quantity defined in the statement of Theorem \ref{thm:rigiditythm}(b), and 
\begin{equation}
\label{eq:def-pd}
p_d:=\begin{cases}
2&\text{if }\,d=2,\\
2(d-1)/d&\text{if }d>2,
\end{cases}
\end{equation} 
 is the exponent for which the embedding  $W^{1,p} \hookrightarrow H^{1/2}$   holds in dimension $d-1$. (See, e.g.,  \cite[Theorem 14.32,  Remark 14.35,  Proposition 14.40]{leoni} and  \cite[Theorem 7.1, Proposition 2.3]{lions.magenes} for the embedding results in the whole space $\mathbb{R}^{d-1}$ for $d>2$ and $d=2$, respectively. Bounded Lipschitz domains in $\mathbb{R}^{d-1}$ can be reduced to the setting above by means of  a Sobolev extension.)


For simplicity, we write $\mathcal{E}_\eps(y)$ instead of   $E_{\eps,\eta_{\ep,d}}$  in the following. Similarly to the energies in the previous subsection, we denote the restriction of the functional to a subset $\Omega' \subset \Omega$ by $\mathcal{E}_{\eps}(y,\Omega')$. We  first introduce the \emph{optimal-profile energy} associated to our model by 
\begin{align}\label{eq: our-k1}
K :=\inf\Big\{&\liminf_{\ep\to 0} \mathcal{E}_{\ep}(y^{\ep},Q): \   \lim_{\eps \to 0}  \Vert  y^\eps -  \RRR y_0^+\EEE \Vert_{L^1(Q)}  = 0\Big\},
\end{align}
 where $Q =(-\frac{1}{2},\frac{1}{2})^d$, and $\RRR y_0^+\EEE$ is defined in  \eqref{eq: conti-schweizer-k-y0}.  We again point out that $K$ is   invariant   under reflection of the two phases $A$ and $B$. Note that \eqref{eq: our-k1} corresponds to \eqref{eq: conti-schweizer-k}, and that we have the relation
\begin{align}\label{eq: constant change}
 K \ge K_0. 
 \end{align}
  Indeed, this is immediate from the definition of the   optimal-profile energy  and the fact that the penalization in  \eqref{eq: nonlinear energy} (with $\eta=  \eta_{\ep,d}$) is stronger than the one in \eqref{eq: I functional}.

Recall $\mathcal{Y}(\Omega)$ in \eqref{eq: limiting deformations}. We introduce the sharp-interface limit $\mathcal{E}_0:L^1(\Omega;\R^d)\to   [0,+\infty]$ by 
\begin{align*}
\mathcal{E}_0(y):=\begin{cases}
 K \,  \mathcal{H}^{d-1}(J_{\nabla y})&
\text{if }y\in  \mathcal{Y}(\Omega),\\
+\infty&\text{otherwise}.\end{cases}
 \end{align*}
 
We now state the main results of this section. 
 
 \begin{proposition}[Liminf inequality]\label{thm:liminf}
Let $d\in \mathbb{N}$, $d \ge 2$, and  let $\Omega\subset \R^d$ be a bounded Lipschitz domain. Let $W$ satisfy assumptions {\rm H1}.--{\rm H4}., let $y\in L^1(\Omega;\mathbb{R}^d)$, and let $\{y^{\ep}\}_\eps \subset H^2(\Omega;\R^d)$ be such that $y^{\ep} \to y$ strongly in $L^1(\Omega;\mathbb{R}^d)$. Then
\begin{equation*}
\liminf_{\ep\to 0}\mathcal{E}_{\ep}(y^{\ep})\geq \mathcal{E}_{0}(y).
\end{equation*}
 \end{proposition}

  \begin{theorem}[Limsup inequality]
 \label{thm:limsup2}
Let $d\in \mathbb{N}$, $d \ge 2$, and let $\Omega \subset \mathbb{R}^d$ be a bounded, strictly star-shaped  Lipschitz  domain.   Let $W$ satisfy assumptions {\rm H1}.--{\rm H5}.\ and let $y\in \mathcal{Y}(\Omega)$. Then, there exists a sequence $\{y^{\ep}\}_\eps \subset H^2(\Omega;\R^d)$ such that $y^{\ep} \to y$ strongly in $L^1(\Omega;\mathbb{R}^d)$ and
\begin{equation*}
\limsup_{\ep\to 0}\mathcal{E}_{\ep}(y^{\ep})\leq \mathcal{E}_{0}(y).
\end{equation*}
\end{theorem}

\begin{remark}[Comparison to the model in Subsection \ref{sec: review-transition}]
\label{rk:comparison-CS}
{\normalfont
We emphasize that the additional penalization term in \eqref{eq: nonlinear energy} with respect to \eqref{eq: I functional} does not affect the qualitative behavior of the sharp-interface limit, only the constant may change, cf.\ \eqref{eq: constant change}. Note that for the physically relevant dimensions $d=2,3$ there holds $r_d=1$, and thus $\eta_{\ep,d}=\ep^{-1/3}$. For $d>3$, the fact that $\frac12 < r_d < 1$ implies that  $\ep^{-1/3}  \ll  \eta_{\ep,d}  \ll  \ep^{-1}$. This  guarantees that, also asymptotically when passing to a linearized strain regime, our perturbed model behaves qualitatively as the unperturbed problem (see the discussion above Theorem \ref{thm:rig-intro}). We remark that our results still hold  up to very minor proof adaptations if $\eta_{\ep,d}$ is replaced by any $\eta \in [\eta_{\ep,d}, 1/\ep]$.}
\end{remark}


 The proof of Proposition \ref{thm:liminf} is similar to \cite[Theorem 4.1]{conti.fonseca.leoni}  and \cite[Proposition 3.3]{conti.schweizer}. We will, however, present the main steps for completeness and will particularly highlight the adaptions which are necessary due to   the  anisotropic singular perturbations. The main point of our contribution is Theorem \ref{thm:limsup2}: the  novelty is that we  can prove the optimality of the lower bound identified in Proposition \ref{thm:liminf} in dimension $d \ge 3$. As a byproduct, we also exhibit a simplified construction of recovery sequences  in the two-dimensional setting. In contrast to \cite{conti.schweizer3},  for simplicity, we work with assumption H5.\ and we do not address the issue of dropping this condition.

The next   three  subsections are devoted to the proof of  our $\Gamma$-convergence result. In Subsection \ref{sec: proof1} we prove Proposition \ref{thm:liminf} and Theorem \ref{thm:limsup2}. The proof of the liminf inequality essentially relies on the properties of the optimal-profile energy (see Proposition \ref{prop:cell-form}), which are the subject of Subsection \ref{sec: cell}. The crucial idea in the proof of Theorem \ref{thm:limsup2} is a novel construction of local recovery sequences (see Proposition \ref{lemma: local1}), which is detailed in Subsection \ref{sec: proof2}.

\subsection{Proof of the $\Gamma$-convergence result}\label{sec: proof1}
 
This subsection is devoted to the proof of Proposition \ref{thm:liminf} and Theorem \ref{thm:limsup2}. As a preparation, we introduce some notation. We let 
$y_0^- \in H^1_{\rm loc}(\R^d;\R^d)$ be the continuous function with $y_0^-(0)=0$ and 
\begin{equation}
\label{eq:y1-}
\nabla y_0^-=B\chi_{\{x_d> 0\}}+A\chi_{\{x_d<0\}}.
\end{equation}
We now state some properties of the   optimal-profile energy given in \eqref{eq: our-k1}. Consider $\omega \subset \R^{d-1}$ open, bounded and let  $h >0$. For brevity, we introduce the notation  of \emph{cylindrical sets} 
\begin{equation}
\label{eq:def-dlh}
D_{\omega,h} := \omega\times (-h,h).
\end{equation} 
We define the \emph{optimal-profile energy   function}  
\begin{align}\label{eq: k-intro}
\mathcal{F}(\omega;h) = \inf\Big\{\liminf_{\ep\to 0} \mathcal{E}_{\ep}(y^{\ep},D_{\omega,h}): \   \lim_{\eps \to 0}  \Vert  y^\eps -  \RRR y_0^+\EEE \Vert_{L^1(D_{\omega,h})}  = 0\Big\}
\end{align}
for every $\omega\subset \mathbb{R}^{d-1}$ and $h>0$. Here and in the following, we again use the shorthand  notation $\mathcal{E}_\eps = E_{\eps,\eta_{\ep,d}}$ for the energy introduced in \eqref{eq: nonlinear energy} and $\eta_{\ep,d}$ from \eqref{eq:eta-ep}.  Letting $Q' =(-\frac{1}{2},\frac{1}{2})^{d-1}$ we observe that $K = \mathcal{F}(Q'; \frac{1}{2})$, where $K$ is the constant defined in \eqref{eq: our-k1}. 

We note that the  optimal-profile energy   is independent of the direction in which the transition between the two phases $A$ and $B$ occurs. Indeed, since the energy functionals $\mathcal{E}_{\ep}$ are invariant under the operation $Ty(x) = -y(-x)$,   there holds (see, e.g., \cite[Lemma 3.2]{conti.schweizer2} for details) 
 \begin{align}\label{eq: minus-version}
\mathcal{F}(\omega;h) = \inf\Big\{\liminf_{\ep\to 0} \mathcal{E}_{\ep}(y^{\ep},D_{\omega,h}): \   \lim_{\eps \to 0}  \Vert  y^\eps -  y_0^- \Vert_{L^1(D_{\omega,h})}  = 0\Big\},
\end{align}
where $y_0^-$ is the function defined in \eqref{eq:y1-}. Some crucial properties of the function $\mathcal{F}$ are summarized in the following proposition.

\begin{proposition}[Properties of the optimal-profile energy function]
\label{prop:cell-form}
The function $\mathcal{F}$ introduced in \eqref{eq: k-intro} satisfies for all $h>0$ and all  open, bounded   sets  $\omega \subset \R^{d-1}$ with $\mathcal{H}^{d-1}(\partial\omega)=0$: 
\begin{itemize}
\item[{\rm (i)}] $\mathcal{F}(\alpha\omega;\alpha h) \ge  \alpha^{d-1}\mathcal{F}(\omega;h) $ for all $0 < \alpha < 1$.
\item[{\rm (ii)}] $\mathcal{F}(\omega;h)= \mathcal{H}^{d-1}(\omega) \, \mathcal{F}(Q';h)$, where $Q' := (-\frac{1}{2},\frac{1}{2})^{d-1}$.
\item[{\rm (iii)}] $\mathcal{F}(\omega;h)= \mathcal{F}(\omega;\frac{1}{2}) = K\, \mathcal{H}^{d-1}(\omega)$. \end{itemize}
\end{proposition}

We defer the proof of Proposition \ref{prop:cell-form} to Subsection \ref{sec: cell} below and now proceed with the proof of Proposition \ref{thm:liminf}.

\begin{proof}[Proof of Proposition  \ref{thm:liminf}]
The   proof follows the strategy in   \cite[Proof of Theorem 4.1]{conti.fonseca.leoni}.  If the liminf is infinite, there is nothing to prove. Otherwise, we apply Lemma \ref{lemma:comp-def} to find that the limit $y$ lies in $\mathcal{Y}(\Omega)$. Without restriction, we can assume that $y \in \mathcal{Y}_{{\rm Id}}(\Omega)$, see \eqref{eq: limiting deformations}. As  $\Omega$ has Lipschitz boundary, we can decompose the jump set of $\nabla y$ as 
$$J_{\nabla y} = \bigcup\nolimits_{i=1}^\infty \omega_i \times \lbrace \alpha_i \rbrace, \ \ \ \ \ \sum\nolimits_{i=1}^\infty \mathcal{H}^{d-1}(\omega_i \times \lbrace \alpha_i \rbrace) < + \infty,$$
where the sets $\omega_i \subset \R^{d-1}$ are open, bounded, connected,  and have  Lipschitz boundary. Let $\delta >0$. We can find $I \in \N$ such that
\begin{align}\label{eq: lowerbound1}
\mathcal{H}^{d-1}(J_{\nabla y}) -\delta \le \sum\nolimits_{i=1}^I \mathcal{H}^{d-1}(\omega_i \times \lbrace \alpha_i \rbrace)  \end{align}
and corresponding $h_i>0$, $i=1,\ldots,I$, such that $\alpha_j \notin (\alpha_i -h_i, \alpha_i + h_i)$ for all $j \in \N$, $j \neq i$, i.e., the cylindrical sets $\alpha_i{\rm e}_d + D_{\omega_i,h_i}$ (see \eqref{eq:def-dlh}) contain exactly one interface. The latter is possible since the interfaces $(\omega_i \times \lbrace \alpha_i \rbrace)_{i >I}$ can only accumulate at $\partial \Omega$, see \cite[Proof of Proposition 3.1]{conti.schweizer2} for details, and the lower part of Figure \ref{fig:limiting-def} for an illustration. 

 Choose $\omega'_i \subset\subset \omega_i$ with Lipschitz boundary such that 
\begin{align}\label{eq: lowerbound2}
\mathcal{H}^{d-1}(\omega_i \times \lbrace \alpha_i \rbrace)  \le \mathcal{H}^{d-1}(\omega'_i \times \lbrace \alpha_i \rbrace)  + \delta/I 
\end{align}
for $i = 1,\ldots,I$, and such that  $\alpha_i{\rm e}_d + D_{\omega_i',h_i}$ is compactly contained in $\Omega$ (possibly  passing  to a smaller $h_i$). Now for any sequence   $\{y^{\ep}\}_\eps \subset H^2(\Omega;\R^d)$ satisfying $y^{\ep} \to y$ strongly in $L^1(\Omega;\mathbb{R}^d)$, by \eqref{eq: k-intro}--\eqref{eq: minus-version}, Proposition \ref{prop:cell-form}, and the fact that the sets $\alpha_i{\rm e}_d +D_{\omega'_i,h_i}$ are pairwise disjoint  we obtain  
\begin{align*}
\liminf_{\eps \to 0} \mathcal{E}_\eps(y^\eps) \ge \sum\nolimits_{i=1}^I \liminf_{\eps \to 0}\mathcal{E}_\eps(y^\eps,\alpha_i{\rm e}_d +D_{\omega_i',h_i}) \ge \sum\nolimits_{i=1}^I\mathcal{F}(\omega_i';h_i) = K \sum\nolimits_{i=1}^I\mathcal{H}^{d-1}(\omega'_i),
\end{align*}
where we used that $y^\eps$ converges (up to a translation) to  $y_0^+$ or $y_0^-$ on each  set $\alpha_i{\rm e}_d +D_{\omega_i',h_i}$. The result follows from \eqref{eq: lowerbound1}-\eqref{eq: lowerbound2} and the arbitrariness of $\delta$. 
\end{proof}

 We now address the limsup inequality. We first describe the local structure of recovery sequences around a single interface. To this end,   recall the definition of the functions $y_0^+$ and $y_0^-$ introduced in \eqref{eq: conti-schweizer-k-y0}  and \eqref{eq:y1-}, respectively, and the structure of cylindrical sets in \eqref{eq:def-dlh}.

\begin{proposition}[Local recovery sequences]\label{lemma: local1}
Let  $d\in \mathbb{N}$, $d\geq 2$. Let $\omega \subset \R^{d-1}$ open, bounded with Lipschitz  boundary.   Then\RRR, we find $h_0>0$ such that, for every $0<h<h_0$, \EEE there exist sequences $\lbrace w^+_\eps \rbrace_\eps, \lbrace w^-_\eps \rbrace_\eps \subset H^2(D_{\omega,h};\R^d)$ with 
\begin{equation}
\label{eq:local1-1}
w^\pm_\eps \to y^\pm_0 \quad\text{ in } H^1(D_{\omega,h} ;\R^d),
\end{equation}
 such that
 \begin{equation}
 \label{eq:local1-3a}
 \lim_{\eps \to 0} \mathcal{E}_{\ep} (w^\pm_\eps, D_{\omega,h})  =  K \, \mathcal{H}^{d-1}(\omega), 
 \end{equation}
 and for $\eps$ sufficiently small we have
\begin{equation}
\label{eq:local1-3} w^\pm_\eps = 
\begin{cases} 
I^\pm_{1,\eps} \circ y_0^\pm & \text{if $x_d \ge   3h/4$},\\
 I^\pm_{2,\eps} \circ y_0^\pm & \text{if $x_d \le -  3h/4$},\\
  \end{cases}  \end{equation}
where $\{I^+_{1,\eps}\}_{ \eps}$, $\{I^+_{2,\eps}\}_{ \eps}$, as well as $\{I^-_{1,\eps}\}_{ \eps}$ and $\{I^-_{2,\eps}\}_{ \eps}$ are sequences of isometries which converge to the identity as $\eps \to 0$. 
\end{proposition}

 We emphasize that Proposition \ref{lemma: local1} means that for \emph{any} sequence $\lbrace \eps_i\rbrace_i$ converging to zero a local recovery sequence can be constructed.   The crucial point is that the sequence $\lbrace w^\pm_\eps\rbrace_\eps$ is rigid away from the interface. This will allow us to  appropriately glue together  local recovery sequences around different interfaces. We defer the proof of Proposition \ref{lemma: local1}  to Subsection \ref{sec: proof2} below and continue with the proof of the limsup inequality. 

\begin{proof}[Proof of Theorem \ref{thm:limsup2}]
Without loss of generality, we can assume that $y\in \mathcal{Y}_{\rm Id}(\Omega)$. For convenience of the reader, we subdivide the proof of the theorem into three steps. 

 \noindent\emph{Step I: Reduction to  a  finite number of interfaces.}   Exploiting the star-shapeness of the domain (say, with respect to the origin),  one can replace $y$ by a slightly rescaled   version   $y_\rho$ defined by $y_\rho(x)=\rho y(x/\rho)$, $\rho>1$, where $\mathcal{E}_0(y_\rho) \to \mathcal{E}_0(y)$ as $\rho \to 1$. One can show that for each $\rho>1$ the jump set $J_{\nabla y_\rho}$ consists only  of a \emph{finite} number of  subsets of hyperplanes that intersect $\partial \Omega$ and are orthogonal to ${\rm e}_d$. We refer to  \cite[Proof of Proposition 5.1]{conti.schweizer2} for the details of this rescaling. The geometrical intuition is that, since infinitely many interfaces can only occur close to the boundary (see also Figure \ref{fig:limiting-def}), a rescaling allows to reduce the study to a finite number of interfaces. It suffices to construct recovery sequences for $y_\rho$ since a recovery sequence for $y$ can then be obtained by a diagonal argument. Thus, in the following it is not restrictive to assume that $J_{\nabla y}$ consists only of a finite number of interfaces.\\

\noindent\emph{Step II: Local recovery sequence.}  In view of Step I, we can suppose that $J_{\nabla y}$ has the form $J_{\nabla y} = \bigcup\nolimits_{j=1}^J (\omega_j \times \lbrace \alpha_j \rbrace)$, where $\omega_j\subset \R^{d-1} $ are open, bounded, and with Lipschitz boundary. Let $\delta>0$. As $\partial \Omega$ has Lipschitz boundary and the $J$ interfaces intersect $\partial \Omega$, we can choose $\omega_j' \supset \supset \omega_j$ open with Lipschitz boundary   and \RRR $0<h<h_0$ (where $h_0$ is the constant in the statement of Proposition \ref{lemma: local1}) \EEE such that  the sets $\partial \omega_j' \times (\alpha_j-h,\alpha_j + h)$ do not intersect $\overline{\Omega}$, the different cylindrical sets $\alpha_j{\rm e}_d +D_{\omega_j',h} = \omega_j' \times (\alpha_j-h,\alpha_j + h)$ are pairwise disjoint,  and one has
\begin{align}\label{eq: slighly larger}  
\mathcal{H}^{d-1}(\omega_j') \le \mathcal{H}^{d-1}(\omega_j) + \delta/J. 
\end{align}
We write  $D_j := \alpha_j{\rm e}_d +D_{\omega_j',h}$ for brevity. Note that on each $D_j \cap \Omega$ the function $y$ coincides with $y_0^+$ or $y_0^-$ up to a translation. Thus,   by Proposition \ref{lemma: local1} we can find $\{w_{\ep}^+\}_{\ep}$   or   $\{w_{\ep}^-\}_{\ep}$ such that \eqref{eq:local1-3a}--\eqref{eq:local1-3} are satisfied and the sequence converges to $y$ in $L^1(D_j\cap \Omega;\R^d)$.   For convenience, we denote this sequence by  $\{w_{\ep}^j\}_{\ep} \subset H^2(D_j;\R^d)$ for $j=1,\ldots,J$.

\noindent\emph{Step III: Global recovery sequence.} Using that $\Omega$ is star-shaped, we find that  $\Omega \setminus \bigcup_{j=1}^J D_j$ consists of $J+1$ components which we denote  by $\lbrace B_j \rbrace_{j=1}^{J+1}$. Applying Proposition \ref{lemma: local1}, one can select isometries $\lbrace I^j_\eps\rbrace_{j=1}^J$ and  $\lbrace\hat{I}^j_\eps\rbrace_{j=1}^{J+1}$, such that the functions $y^\eps : \Omega \to \R^d$ defined by
$$y^\eps = I_\eps^j \circ w_\eps^j  \ \ \text{on} \ \ D_j \cap \Omega, \ \ \ \ \ \ y^\eps = \hat{I}_\eps^j \circ y  \ \ \text{on} \ \ B_j$$
 are  in $H^2(\Omega;\R^d)$, and all isometries converge to the identity as $\eps \to 0$. These isometries can be chosen iteratively, and we refer to \cite[Proof of Proposition 3.5]{conti.schweizer} for details. Since $w_\eps^j$ converges to $y$ in $L^1(D_j\cap \Omega;\R^d)$ and all isometries converge to the identity, we obtain $y^\eps \to y$ in $L^1(\Omega;\R^d)$. The construction also implies that  on $\bigcup_{j=1}^{J+1} B_j$ there holds  $\nabla y^\eps \in  SO(d)\lbrace A,B \rbrace$ and $\nabla^2 y^\eps = 0$. Therefore, by Proposition \ref{lemma: local1} and \eqref{eq: slighly larger}   we   deduce  
\begin{align*}
\limsup_{\eps \to 0}\mathcal{E}_{\ep}({y}^\eps) &   \leq   \limsup_{\eps \to 0}\sum\nolimits_{j=1}^J \mathcal{E}_{\ep}(w_\eps^j,D_j)   = K \, \sum\nolimits_{j=1}^J \mathcal{H}^{d-1}(\omega'_j) \\
& \le  K \, \sum\nolimits_{j=1}^J \mathcal{H}^{d-1}(\omega_j) + K\delta = K\, \mathcal{H}^{d-1}(J_{\nabla y}) + K\delta.
\end{align*}
Letting $\delta \to 0$ and  using a standard  diagonal argument we obtain the thesis. 
\end{proof}

\subsection{Properties of optimal-profile energy}\label{sec: cell}

In this subsection we prove Proposition \ref{prop:cell-form}. Additionally, we show in Proposition \ref{prop:FG} that, \RRRR given \emph{any} sequence $\{\ep_i\}_i$, there exists a sequence of optimal profiles, i.e., sequences of deformations whose energies asymptotically converge to the value of the optimal-profile energy. \EEE As a byproduct, we also get that the energy of optimal-profile sequences  concentrates near  the interface, see Corollary \ref{cor: layer energy}.

\begin{proof}[Proof of Proposition \ref{prop:cell-form}]
We first observe that for all $h>0$
\begin{align}\label{eq: prelim-prop}
{\rm (a)} &  \ \ \mathcal{F}(x' + \omega;h) = \mathcal{F}(\omega;h) \ \ \ \text{ for all $x' \in \R^{d-1}$}\notag, \\
{\rm (b)} &  \ \ \mathcal{F}(\omega_1;h) \le \mathcal{F}(\omega_2;\tau) \ \ \ \text{ if   $\omega_1 \subset \omega_2$   and $h \le \tau$},\notag \\
{\rm (c)} &   \ \ \mathcal{F}(\omega_1 \cup \omega_2;h) \ge \mathcal{F}(\omega_1;h)  + \mathcal{F}(\omega_2;h) \ \ \ \text{ if $\omega_1 \cap \omega_2 = \emptyset$.} 
\end{align}
These elementary properties follow from the fact that $\mathcal{E}_\eps$ is nonnegative and invariant under translations, and the observation that sequences in \eqref{eq: k-intro} on $D_{\omega_2,\tau}$ are still admissible on $D_{\omega_1,h}$, whenever $\omega_1 \subset \omega_2$ and $h \le \tau$.  

As a preparation for the proof of {\rm (i)}, we perform a standard rescaling argument for a configuration $y \in H^2(\alpha D_{\omega,h};\R^d)$ with $0 < \alpha <1$. We define $\bar{y} \in H^2(D_{\omega,h};\R^d)$ by $\bar{y}(x) =  y(\alpha x)/\alpha$, and observe that $\nabla \bar{y}(x) = \nabla y (\alpha x)$ and  $\nabla^2 \bar{y}(x) = \alpha  \nabla^2y(\alpha x)$ for all $x \in D_{\omega,h}$.  The fact that the sequence $\lbrace  \eta_{\ep,d} \rbrace_\eps$ is increasing as $\eps \to 0$ (see \eqref{eq:eta-ep}) implies  $\eta^2_{\sqrt{\alpha}\eps,d} \ge \alpha\eta^2_{\ep,d}$. Thus, we obtain by \eqref{eq: nonlinear energy}  
\begin{align}\label{eq: transformation preparation}
\mathcal{E}_{\sqrt{\alpha}\ep}(y,\alpha D_{\omega,h})  &\ge  \frac{1}{\alpha\ep^2}\int_{\alpha D_{\omega,h}}W(\nabla y)\,dx+ \alpha\ep^2\int_{\alpha D_{\omega,h}}|\nabla^2 y|^2\,dx+ \alpha\eta_{\ep,d}^2  \int_{\alpha D_{\omega,h}}(|\nabla^2 y|^2-|\partial^2_{dd} y|^2)\,dx \notag \\
& =  \frac{\alpha^{d-1}}{\ep^2}\int_{D_{\omega,h}}W(\nabla \bar{y})\,dx+\alpha^{d-1} \ep^2\int_{D_{\omega,h}}|\nabla^2 \bar{y}|^2\,dx+ \alpha^{d-1} \eta_{\ep,d}^2 \int_{D_{\omega,h}}(|\nabla^2 \bar{y}|^2-|\partial^2_{dd} \bar{y}|^2)\,dx \notag\\ 
&=  \alpha^{d-1}  \mathcal{E}_{\ep}(\bar{y},D_{\omega,h}). 
\end{align}

We now prove {\rm (i)}. Let $0< \alpha < 1$.  By \eqref{eq: k-intro}, for a given $\delta>0$, we can choose sequences $\lbrace \eps_i\rbrace_{i}$ and  $\lbrace y^{\eps_i}\rbrace_i \subset H^2(\alpha D_{\omega,h};\R^d)$ with $  \Vert  y^{\eps_i} -  \RRR y_0^+\EEE \Vert_{L^1(\alpha D_{\omega,h})} \to 0$
and 
\begin{align}\label{eq: quasi opt}
\liminf_{i \to \infty} \mathcal{E}_{\sqrt{\alpha}\eps_i}(y^{\eps_i}, \alpha D_{\omega,h}) \le \mathcal{F}(\alpha\omega; \alpha h) + \delta.
\end{align}
Let $\lbrace \bar{y}^{\eps_i}\rbrace_i\subset H^2(D_{\omega,h};\R^d)$ be the rescaled functions defined before \eqref{eq: transformation preparation}. Note that $  \Vert  \bar{y}^{\eps_i} -  \RRR y_0^+\EEE \Vert_{L^1(D_{\omega,h})} \to 0$, which follows from a scaling argument and  the fact that the function  $\bar{y}_0$   defined by $\bar{y}_0(x):=  \RRR y_0^+\EEE(\alpha x)/\alpha$ coincides with $\RRR y_0^+\EEE$. The definition of $\mathcal{F}$, \eqref{eq: transformation preparation}, and \eqref{eq: quasi opt} imply
$$\delta + \mathcal{F}(\alpha \omega; \alpha h) \ge  \liminf_{i \to \infty} \mathcal{E}_{\sqrt{\alpha}\eps_i}(y^{\eps_i}, \alpha D_{\omega,h}) \ge \alpha^{d-1}\liminf_{i \to \infty} \mathcal{E}_{\ep_i}(\bar{y}^{\eps_i}, D_{\omega,h}) \ge \alpha^{d-1}\mathcal{F}(\omega;h).   $$
Since $\delta >0$ was arbitrary, {\rm (i)} follows. 

The proof of properties {\rm (ii)} and {\rm (iii)} is similar to the one in \cite[Lemma 4.3]{conti.fonseca.leoni}. We present the main steps here for convenience of the reader. We show {\rm (ii)}. We use  a covering theorem (see, e.g., \cite[Remark 1.148{\rm (ii)}]{fonseca.leoni}) to decompose $\omega = \bigcup_{i \in \N} (a_i + \delta_i Q') \cup N_0$ into pairwise disjoint sets  $a_i + \delta_i Q'$, for  $a_i \in \R^{d-1}$ and $0 < \delta_i < 1$,  where $\mathcal{H}^{d-1}(N_0) = 0$, $Q'=\big(-\tfrac12,\tfrac12\big)^{d-1}$,  and 
\begin{align}\label{eq: correct-vol}
\sum\nolimits_{i=1}^\infty \delta_i^{d-1} = \mathcal{H}^{d-1}(\omega).
\end{align}
 Then \eqref{eq: prelim-prop} and {\rm (i)} imply for all $I \in \N$
\begin{align*}
\mathcal{F}(\omega;h) \ge \sum\nolimits_{i=1}^I  \mathcal{F}(\delta_i Q'; h) \ge \sum\nolimits_{i=1}^I  \mathcal{F}(\delta_i Q'; \delta_i h) \ge \sum\nolimits_{i=1}^I  \delta_i^{d-1} \mathcal{F}(Q'; h).
\end{align*}
Letting $I \to \infty$ and using \eqref{eq: correct-vol} we conclude   that   $\mathcal{F}(\omega;h) \ge \mathcal{H}^{d-1}(\omega)\,  \mathcal{F}(Q';h)$. The reverse inequality follows by interchanging the roles of $\omega$ and $Q'$ in the above argument, see \cite[Lemma 4.3]{conti.fonseca.leoni} for details.

We finally prove {\rm (iii)}. The second identity in {\rm (iii)} follows from {\rm (ii)} and the fact that $K = \mathcal{F}(Q'; \frac{1}{2})$, see \eqref{eq: our-k1}. We show the first identity. To this end, it suffices to prove   that 
\begin{align}\label{eq: kappa-h-1}
\mathcal{F}(Q';\tau) = \mathcal{F}(Q';\gamma \tau) \ \ \ \ \text{ for all $\tau>0$ and for all $\gamma \in \N$}.
\end{align}
Indeed, by \eqref{eq: prelim-prop}(b) we get $\mathcal{F}(Q';\tau) \le \mathcal{F}(Q';\frac{1}{2}) \le \mathcal{F}(Q';\gamma \tau)$ for all $0<\tau< \frac{1}{2}$ and $\gamma \in \N$ such that $\gamma \tau \ge \frac{1}{2}$. This along with \eqref{eq: kappa-h-1} then implies $\mathcal{F}(Q';\tau)= \mathcal{F}(Q';\frac{1}{2})$ for all $0<\tau< \frac{1}{2}$. Using \eqref{eq: kappa-h-1} once more, we get $\mathcal{F}(Q';h)= \mathcal{F}(Q';\frac{1}{2})$ for all $h>0$. The statement follows with {\rm (ii)}.

Let us now show \eqref{eq: kappa-h-1}. We decompose   $\gamma Q'$ into the union  
$$\gamma Q' = \bigcup\nolimits_{i=1}^{\gamma^{d-1}} (a_i  + Q') \cup N_0$$
consisting of pairwise disjoint hypercubes, where $\mathcal{H}^{d-1}(N_0) = 0$. By  {\rm (i)} (with $\omega = \gamma Q'$,   $h = \gamma \tau$,   and  $\alpha =1/\gamma$) we find $\mathcal{F}(Q'; \tau) \ge  \gamma^{-(d-1)}\mathcal{F}(\gamma Q'; \gamma \tau)$. Thus, using \eqref{eq: prelim-prop} we compute
\begin{align*}
 \mathcal{F}(Q'; \tau) \ge  \gamma^{-(d-1)}\mathcal{F}(\gamma Q'; \gamma \tau) \ge \gamma^{-(d-1)}\sum\nolimits_{i=1}^{\gamma^{d-1}} \mathcal{F} (a_i + Q';\gamma \tau) \ge \mathcal{F} (Q';\gamma \tau) \ge  \mathcal{F}(Q'; \tau). 
\end{align*}
This concludes the proof  of the proposition.  
\end{proof}

We now show that a sequence of optimal profiles can be chosen independently of the particular choice of $\lbrace \eps_i \rbrace_i$. To this end, \RRR we first observe that the  function $\mathcal{F}$ defined in \eqref{eq: k-intro} can be written equivalently as
\begin{align}\label{eq: k-intro-NNN}
\mathcal{F}(\omega;h) = \inf_{\lbrace\eps_i\rbrace_i, \, \eps_i \to 0}  \, \inf \Big\{\liminf_{i \to \infty} \mathcal{E}_{\ep_i}(y^{\ep_i},D_{\omega,h}): \   \lim_{i \to \infty}  \Vert  y^{\eps_i} -   y_0^+ \Vert_{L^1(D_{\omega,h})}  = 0\Big\}.
\end{align}
Similar to  $\mathcal{F}$, we introduce  the function $\mathcal{G}$,   given by  
\begin{align*}
\mathcal{G}(\omega;h) = \sup_{\lbrace \eps_i\rbrace_i, \, \eps_i \to 0} \, \inf \Big\{\RRRR \limsup_{i \to \infty} \EEE \mathcal{E}_{\ep_i}(y^{\ep_i},D_{\omega,h}): \   \lim_{i \to \infty}  \Vert  y^{\eps_i} -   y_0^+ \Vert_{L^1(D_{\omega,h})}  = 0\Big\},
\end{align*}
for every $\omega\subset \mathbb{R}^{d-1}$ and $h>0$. \EEE 


\begin{proposition}[$\mathcal{F} = \mathcal{G}$]
\label{prop:FG}
\RRR There exists $\bar{h}>0$ such that \EEE $\mathcal{F}(\omega;h) = \mathcal{G}(\omega;h)$ for all $\omega \subset \R^{d-1}$ open, bounded with $\mathcal{H}^{d-1}(\partial \omega)=0$ and all \RRR $0<h<\bar{h}$. \EEE 
\end{proposition}

For the proof of Proposition \ref{prop:FG} we need the following technical lemma. Recall $\kappa = |A-B|$ and the constant $c_1$ from H4. Recall also the definition of $E_{\eps,\eta}$ in  \eqref{eq: nonlinear energy}.

\begin{lemma}[Zooming to the interface]\label{lemma: intefacefind}
Let $\lbrace \eps_i\rbrace_i$ be an infinitesimal sequence and let  $\eta_{\ep_i}\geq \ep_i^{-\frac13}$  for every $i \in \N$. Let $\mathcal{\mathcal{Q}}' \subset \R^{d-1}$ be a cube. Let $\mathcal{Q} \subset \R^d$ be the cube whose orthogonal projection on $\R^{d-1} \times \lbrace 0 \rbrace$ is $\mathcal{Q}'$. Let  $\lbrace h_i\rbrace_i \subset \R_+$. For every $i\in \mathbb{N}$, let $y^{i} \in H^2(\mathcal{Q} \cup D_{\mathcal{Q}',h_i};\R^d)$  with $E_{\eps_i,\eta_{\ep_i}}(y^{i},\mathcal{Q} \cup D_{\mathcal{Q}',h_i}) \le M <+\infty$, and   assume that  
\begin{align}\label{eq: converg assu}
h_i^{-1}\Vert  \nabla  y^{i} - \nabla  \RRR y_0^+\EEE \Vert^2_{L^2(D_{\mathcal{Q}',h_i})} \to 0.
\end{align} 
Then \RRR for \EEE every bounded sequence $\lbrace \tau_i \rbrace_i \subset \R_+$  with 
\begin{align}\label{eq: cond on h}
 \tau_i \le \RRR h_i, \EEE  \ \ \ \ \ \tau_i\eta_{\eps_i}/\eps_i \to \infty, \ \ \ \ \ \ \ \tau_i/ \eps_i^{1+\frac{1}{d}} \to \infty 
\end{align}
 we find \RRR $\mu >0$, \EEE a sequence $\lbrace \alpha_i \rbrace_i \subset \R$ with $\alpha_i{\rm e}_d+D_{\mathcal{Q}',\RRR \mu \EEE \tau_i}\subset D_{\mathcal{Q}',h_i}$, and a sequence of isometries $\lbrace I_i\rbrace_i$ such that the  maps  $\lbrace v^i\rbrace_i \subset H^2(D_{\mathcal{Q}',\RRR \mu \EEE\tau_i};\R^d)$, defined by 
\begin{align}\label{eq: vi}
v^i(x) = I_i \circ  y^{i}  (x + \alpha_i {\rm e}_d) \ \ \ \text{ for every }\, x \in  D_{\mathcal{Q}',\RRR \mu \EEE\tau_i}, 
\end{align}
satisfy
\begin{align}\label{eq: to confirm}
\tau_i^{-1}\Vert \nabla v^i - \nabla \RRR y_0^+\EEE \Vert^2_{L^2(D_{\mathcal{Q}',\RRR \mu \EEE\tau_i})} \to 0.
\end{align} 
\end{lemma}

Assumption \eqref{eq: converg assu} means that asymptotically a big portion of $D_{\mathcal{Q}',h_i} \cap \lbrace x_d >0 \rbrace$ and $D_{\mathcal{Q}',h_i} \cap \lbrace x_d <0 \rbrace$, respectively, is contained in the $A$ and $B$-phase region, respectively. The lemma states that one may find cylindrical sets inside $D_{\mathcal{Q}',h_i}$ with (much) smaller heights (satisfying  suitable assumptions,  cf.\ \eqref{eq: cond on h}) such that a similar property holds on these cylindrical sets, see \eqref{eq: to confirm} and  Figure \ref{fig:zoom-in}. Loosely speaking, the result shows that the interface between the $A$ and $B$-phase  regions  becomes asymptotically flat, where the nonflatness can be quantified in terms of the sequence $\lbrace \tau_i \rbrace_i$.  

\begin{figure}[h]
\centering
\begin{tikzpicture}
\coordinate (A) at (-3,3);
 \coordinate (B) at (3.5,3);
 \coordinate (E) at (3.5,2.1);
 \coordinate (F) at (3.5,1.4);
 \coordinate (G) at (3.5,0.6);
  \coordinate (O) at (-3,2.1);
 \coordinate (P) at (-3,1.4);
 \coordinate (Q) at (-3,0.6);


 \draw (A)--(B)--(G)--(Q)--cycle;
 \draw[fill=cyan, opacity=0.2] (-3,2.4)--(3.5,2.4)--(3.5,1.7)--(-3,1.7)--cycle;
 \draw [black] (-3,2.1)  to [out=-23,in=165] (3.5,2.1);
 
\node at (-1.8,1) {$A$};
\node at  (-0.6,2.1) {$B$};

 \coordinate (C) at (4,3);
  \coordinate (D) at (10,3);
   \coordinate (H) at (4,0);
    \coordinate (I) at (10,0);
    \draw[cyan, fill=cyan, opacity=0.2] (C)--(D)--(I)--(H)--cycle;
    \draw[cyan] (3.5,2.4)--(C);
     \draw[cyan] (3.5,1.7)--(H);
     \draw[black] (4,1.5)  to [out=-23,in=165] (10,1.5);
     \node at (7,2.25) {$A$};
      \node at (7,0.75) {$B$};


\end{tikzpicture}
\caption{The interface between the $A$ and $B$-phase regions becomes asymptotically flat.}
\label{fig:zoom-in}
\end{figure}
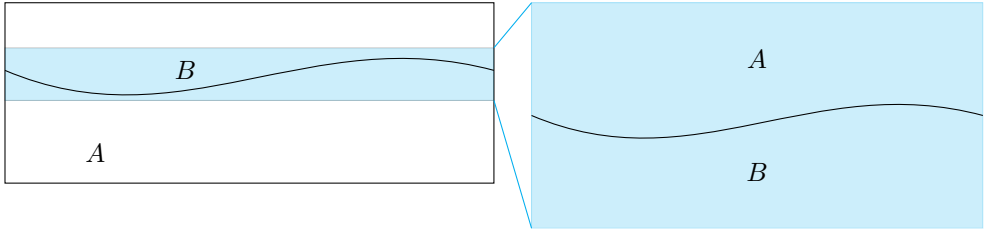

For the proof of Proposition \ref{prop:FG}, we will need this lemma only for $\tau_i \sim 1$ and $\eta_{\ep_i}=\eta_{\ep_i,d}$. However, we prefer to present this more general version  since this  will be instrumental in the companion work \cite{davoli.friedrich}.  We also remark that the assumption  $\tau_i\eta_{\eps_i}/\eps_i \to \infty$ on $\tau_i$ is sharp in order to obtain the above result.

We postpone the proof of the lemma and proceed with the proof of Proposition \ref{prop:FG}.

\begin{proof}[Proof of Proposition \ref{prop:FG}]
 For convenience of the reader, we subdivide the proof into three steps. In Step I we show how the problem can be reduced to the case in which  $\omega$ is a cube. In \RRR Steps \EEE II and III we then address this special  setting. Here, we will use Lemma \ref{lemma: intefacefind} and also some arguments inspired by \cite[Proposition 5.5]{conti.schweizer2}. 

\noindent\emph{Step I: Reduction to a cube.}  We first observe that  the essential point is to prove \RRR that there exists $\bar{h}>0$ such that \EEE
\begin{align}\label{eq: for cubes only}
\mathcal{F}(\mathcal{Q}';h) = \mathcal{G}(\mathcal{Q}';h) \ \ \ \ \text{for all cubes } \mathcal{Q}' \subset \R^{d-1} \text{ and all \RRR $0<h<\bar{h}$\EEE}. 
\end{align}
Once this is established, we may conclude as follows. Given $\omega \subset \R^{d-1}$ open, bounded, with $\mathcal{H}^{d-1}(\partial \omega)=0$, we select a cube $\mathcal{Q}' \subset \R^{d-1}$ containing $\omega$. Suppose by contradiction that the statement was wrong, i.e., \RRR that there exists $0<h<\bar{h}$ such that \EEE $\delta:= \frac{1}{\RRR 3 \EEE } (\mathcal{G}(\omega;h)- \mathcal{F}(\omega;h)) >0$.    Let  $\lbrace \eps_i \rbrace_i$ be a sequence such that for any $\lbrace v^{\eps_i} \rbrace_i \subset H^2(D_{\omega,h};\R^d)$ with $\Vert  v^{\eps_i} - \RRRR y_0^+ \EEE \Vert_{L^1(D_{\omega,h})}  \to 0$ one has 
\begin{align}\label{eq: for cubes only2}
\RRRR \limsup_{i \to \infty} \EEE \mathcal{E}_{\ep_i}(v^{\ep_i},D_{\omega,h}) \ge \mathcal{G}(\omega;h) \RRR - \delta. \EEE
\end{align}
In view of \eqref{eq: for cubes only}, for this specific sequence $\lbrace \eps_i \rbrace_i$, we can find a sequence of functions $\lbrace y^{\eps_i} \rbrace_i \subset H^2(D_{\mathcal{Q}',h};\R^d)$ such that $\Vert  y^{\eps_i} -  \RRRR y^+_0 \EEE \Vert_{L^1(D_{\mathcal{Q}',h})}  \to 0$ and 
\begin{align}\label{eq: for cubes only3}
\limsup_{i \to \infty} \mathcal{E}_{\ep_i}(y^{\ep_i},D_{\mathcal{Q}',h}) \le  \mathcal{G}(\mathcal{Q}';h) + \delta \EEE =   \mathcal{F}(\mathcal{Q}';h) +  \delta. \EEE
\end{align}
Using  \eqref{eq: k-intro-NNN}, Proposition \ref{prop:cell-form}{\rm (ii)},  \eqref{eq: for cubes only2}, and the equality $\RRR 3 \EEE \delta = \mathcal{G}(\omega;h)- \mathcal{F}(\omega;h)$ we derive
\begin{align*}
\RRRR \limsup_{i \to \infty} \EEE \mathcal{E}_{\ep_i}(y^{\ep_i},D_{\mathcal{Q}',h}) & \ge \liminf_{i \to \infty} \mathcal{E}_{\ep_i}(y^{\ep_i},D_{\mathcal{Q}',h} \setminus D_{\omega,h}) + \RRRR \limsup_{i \to \infty} \EEE \mathcal{E}_{\ep_i}(y^{\ep_i},D_{\omega,h}) \\
&  \ge  \mathcal{F}(\mathcal{Q}'\setminus \omega;h) + \mathcal{G}(\omega;h) - \RRR \delta \EEE =\mathcal{F}(\mathcal{Q}'\setminus \omega;h) + \mathcal{F}(\omega;h) + 2\delta = \mathcal{F}(\mathcal{Q}';h) + 2\delta. 
\end{align*}
This, however, contradicts \eqref{eq: for cubes only3}.

\noindent\emph{Step II: Construction of an admissible sequence on increasing cylindrical sets.} It remains to prove \eqref{eq: for cubes only}. To this end, \RRR let $\delta>0$, and  \RRRR let $\lbrace \eps_i \rbrace_i$ be a sequence   converging   to zero such that
for any $h \in \mathbb{Q} \cap (0,1)$ and any  $\lbrace v^{\eps_i} \rbrace_i \subset H^2(D_{\mathcal{Q}',h};\R^d)$ with $\Vert  v^{\eps_i} - \RRRR y_0^+  \Vert_{L^1(D_{\mathcal{Q}',h})}  \to 0$ one has 
\begin{align}\label{eq: for cubes only2XXX}
\limsup_{i \to \infty}  \mathcal{E}_{\ep_i}(v^{\ep_i},D_{\mathcal{Q}',h}) \ge \mathcal{G}(\mathcal{Q}';h) \RRR - \delta. \EEE
\end{align}
(In fact, for each $h \in \mathbb{Q} \cap (0,1)$, the existence of  such a sequence $\lbrace \eps^h_i \rbrace_i$ follows from the definition of $\mathcal{G}$. Then, by a diagonal argument one can choose a sequence $\lbrace \eps_i \rbrace_i$ converging   to zero such that for each $h$ we have  $\lbrace \eps^h_j \rbrace_{j \ge j_h} \subset \lbrace \eps_i \rbrace_i$ for $j_h \in \mathbb{N}$ sufficiently large.) \EEE Choose $\lbrace \tilde{\eps}_j\rbrace_j$ and $\lbrace \tilde{y}^{\tilde{\eps}_j} \rbrace_j \subset H^2(D_{\mathcal{Q}',\RRR 1 \EEE };\R^d)$ such that $\Vert  \tilde{y}^{\tilde{\eps}_j} -  \RRR y_0^+\EEE \Vert_{L^1(D_{\mathcal{Q}',\RRR 1 \EEE })}  \to 0$ and 
\begin{align}\label{eq: for cubes only5}
\limsup_{j \to \infty} \mathcal{E}_{\tilde{\ep}_j}(\tilde{y}^{\tilde{\ep}_j},D_{\mathcal{Q}',\RRR 1 \EEE }) \le \mathcal{F}(\mathcal{Q}';\RRR 1 \EEE ) + \delta.
\end{align}
By Lemma \ref{lemma:comp-def} we may also assume  that $\Vert  \tilde{y}^{\tilde{\eps}_j} -  \RRR y_0^+\EEE \Vert_{H^1(D_{\mathcal{Q}',\RRR 1 \EEE })}  \to 0$.  After passing to a subsequence,  we   may also suppose that $\lbrace \tilde{\eps}_j\rbrace_j$ is monotone. For each $i$, we let   $j(i)>i$  be the smallest index  such that $\tilde{\eps}_{j{ (i)}}  < \eps_i /i$. We now rescale $\tilde{y}^{\tilde{\eps}_{j{ (i)}}}$ using \eqref{eq: transformation preparation}: letting $\alpha_i = (\tilde{\eps}_{j{ (i)}}/\eps_i)^2$, we find $\bar{y}^i \in H^2(\alpha_i^{-1}D_{\mathcal{Q}',\RRR 1 \EEE };\R^d)$ such that  
\begin{equation}
\label{eq:optimality-i}
\alpha_i^{d}\Vert  \nabla \bar{y}^i- \nabla  \RRR y_0^+\EEE \Vert^{ 2 }_{L^{2}(\alpha_i^{-1} D_{\mathcal{Q}',\RRR 1 \EEE })}  \to 0
\end{equation} and
\begin{align*}
 \alpha_i^{d-1}  \mathcal{E}_{\ep_i}(\bar{y}^i,\alpha_i^{-1}D_{\mathcal{Q}',\RRR 1 \EEE }) \le \mathcal{E}_{\sqrt{\alpha_i}\ep_i}(\tilde{y}^{\tilde{\eps}_{j{ (i)}}},D_{\mathcal{Q}',\RRR 1 \EEE }) = \mathcal{E}_{\tilde{\ep}_{j{ (i)}}}(\tilde{y}^{\tilde{\ep}_{j{ (i)}}},D_{\mathcal{Q}',\RRR 1 \EEE }).
\end{align*}
We can (almost) cover $\alpha_i^{-1}D_{\mathcal{Q}',\RRR 1 \EEE }$ by $\lfloor \alpha_i^{-1} \rfloor^{d-1}$  pairwise disjoint translated copies of $D_{\mathcal{Q}',h_i}$, where we define $h_i = \alpha_i^{-1}$. This implies that  for every $i\in \mathbb{N}$   we can find   $z_i \in \R^{d-1} \times \lbrace 0 \rbrace$  such that by a De Giorgi argument   there holds  
\begin{align}\label{eq_ delta fina}
{\rm {\rm (i)}} & \ \ \mathcal{E}_{\ep_i}(\bar{y}^i, z_i  + D_{\mathcal{Q}',h_i}) \le  \frac{(1+\delta)}{ \lfloor \alpha_i^{-1} \rfloor^{d-1} } \, \mathcal{E}_{\ep_i}(\bar{y}^i,\alpha_i^{-1}D_{\mathcal{Q}',\RRR 1 \EEE }) \le  \frac{(1+\delta)}{(\lfloor \alpha_i^{-1} \rfloor  \alpha_i)^{ d-1}}  \,\mathcal{E}_{\tilde{\ep}_{j{ (i)}}}(\tilde{y}^{\tilde{\ep}_{j{ (i)}}},D_{\mathcal{Q}',\RRR 1 \EEE }),\notag \\
{\rm {\rm (ii)}} & \ \  \Vert  \nabla \bar{y}^i-  \nabla \RRR y_0^+\EEE \Vert^{ 2 }_{L^{2}(z_i+ D_{\mathcal{Q}',h_i})} \le C\delta^{-1}\alpha_i^{d-1}\Vert  \nabla \bar{y}^i- \nabla  \RRR y_0^+\EEE \Vert^{ 2 }_{L^{2}(\alpha_i^{-1} D_{\mathcal{Q}',\RRR 1 \EEE })}.
\end{align}
By the definition of $\alpha_i$ there holds  $\alpha_i^{-1} \ge i^2$,  thus  we get $\alpha_i \lfloor \alpha_i^{-1} \rfloor \to 1$. This along with \eqref{eq: for cubes only5} and  \eqref{eq_ delta fina}{\rm (i)} yields
\begin{align}\label{eq: for cubes only7}
\limsup_{i \to \infty}\mathcal{E}_{\ep_i}(\bar{y}^i,  z_i  + D_{\mathcal{Q}',h_i}) \le  (1+\delta)(\mathcal{F}(\mathcal{Q}'; \RRR 1 \EEE  ) + \delta).
\end{align}
Moreover,  by \eqref{eq:optimality-i}, \eqref{eq_ delta fina}{\rm (ii)}, and $h_i = \alpha_i^{-1}$  we obtain   $h_i^{-1}\Vert  \nabla \bar{y}^i- \nabla \RRR y_0^+\EEE \Vert^{ 2 }_{L^{2}(z_i  + D_{\mathcal{Q}',h_i})}  \to 0$. 

\noindent\emph{Step III: Construction of an admissible sequence on a fixed cylindrical set.}  
The goal is now to choose \RRR $\bar{h}>0$, and \EEE a cylindrical set of height \RRR $\bar{h}$ \EEE inside $z_i + D_{\mathcal{Q}',h_i}$ such that $\bar{y}^i$ converges to $\RRR y_0^+\EEE$  on this cylindrical set. After a translation it is not restrictive to assume that $z_i=0$ in the following.  Recall that \RRR $h_i  = \alpha_i^{-1}   \ge i^2$. \EEE We apply Lemma \ref{lemma: intefacefind} for $\lbrace \bar{y}^i \rbrace_i$,  $\lbrace h_i\rbrace_i$,  and $\tau_i = \RRR 1 \EEE $. (Note that \eqref{eq: cond on h} clearly holds  for $i$ sufficiently large in view of $h_i \to \infty$ and \eqref{eq:eta-ep}. Similarly, we find $ D_{\mathcal{Q}',h_i} \supset \mathcal{Q}$ for $i$ large enough.)   We \RRRR find $\mu \in \mathbb{Q} \cap (0,1)$,   \EEE  and  a sequence  of functions $v^i \in H^2(D_{\mathcal{Q}', \RRR \mu \EEE  }; \R^d )$   with 
\begin{align}\label{eq: for cubes only6}
{\rm {\rm (i)}} \ \ \limsup_{i \to \infty}\mathcal{E}_{\ep_i}(v^i, D_{\mathcal{Q}',\RRR \mu \EEE }) \le  \limsup_{i \to \infty}\mathcal{E}_{\ep_i}(\bar{y}^i, D_{\mathcal{Q}',h_i}), \ \ \ \ \  {\rm {\rm (ii)}} \ \    \Vert \nabla v^i - \nabla \RRR y_0^+\EEE\Vert^2_{L^2(D_{\mathcal{Q}', \RRR \mu \EEE })} \to 0. 
\end{align}
By \eqref{eq: for cubes only7}, \eqref{eq: for cubes only6}{\rm (i)}, and Lemma \ref{lemma:comp-def} we find a (non-relabeled) subsequence  and  a map  $v \in \mathcal{Y}(D_{\mathcal{Q}',\RRR \mu \EEE })$ such that, up to translations,  $v^i \to v$ in $H^1(D_{\mathcal{Q}',\RRR \mu \EEE };\R^d)$. Due to \eqref{eq: for cubes only6}{\rm (ii)}, the limit $v$ can be identified  with  $\RRR y_0^+\EEE$. As the limit is independent of the particular subsequence, we  then get that 
$\Vert v^i - \RRR y_0^+\EEE \Vert_{L^1(D_{\mathcal{Q}', \RRR \mu \EEE })} \to 0 $ for the whole sequence $\lbrace \eps_i \rbrace_i$. Thus, $\lbrace v^i \rbrace_i$ is  an admissible \RRRR sequence in \eqref{eq: for cubes only2XXX}  for {\color{red}$h=\mu$} \EEE and we find by \eqref{eq: for cubes only7}--\eqref{eq: for cubes only6} 
$${\mathcal{G}(\mathcal{Q}'; \RRR \mu \EEE ) \RRR - \delta \EEE \le \limsup_{i \to \infty} \mathcal{E}_{\ep_i}(  v^i,D_{\mathcal{Q}', \RRR \mu \EEE })   \le     (1+\delta)(\mathcal{F}(\mathcal{Q}'; \RRR 1 \EEE ) + \delta).}$$
Since $\delta >0$ was arbitrary, we conclude  that  $\mathcal{G}(\mathcal{Q}';h) \le \mathcal{F}(\mathcal{Q}';h)$ \RRR for all $h \le \mu$ by Proposition \ref{prop:cell-form}(iii), and by the fact that $\mathcal{G}(\mathcal{Q}';h)\leq \mathcal{G}(\mathcal{Q}';\mu)$ for every $h\leq \mu$. \EEE As $\mathcal{G}(\mathcal{Q}';h) \ge \mathcal{F}(\mathcal{Q}';h)$ trivially holds, the proof of \eqref{eq: for cubes only} is completed \RRR choosing $\bar{h}=\mu$. \EEE 
\end{proof}

%
%
%
%
%
%

We proceed with a consequence of Proposition \ref{prop:cell-form} and   Proposition \ref{prop:FG}, namely that the energy of optimal-profile sequences concentrates near the interface.

\begin{corollary}[Concentration of the energy near the interface]\label{cor: layer energy}
Let $\omega \subset \R^{d-1}$ open, bounded with $\mathcal{H}^{d-1}(\partial \omega)=0$ and let \RRR $0<h<\bar{h}$, where $\bar{h}$ is the constant introduced in Proposition \ref{prop:FG}. \EEE Let $\lbrace \eps_i\rbrace_i$ be an infinitesimal sequence. Then   there exists   $\lbrace y^{\eps_i}\rbrace_i \subset H^2(D_{\omega,h};\mathbb{R}^d)$ such that
$${\lim_{i\to \infty} \mathcal{E}_{\eps_i}(y^{\eps_i}, D_{\omega,h})= K\mathcal{H}^{d-1}(\omega), \ \ \   \mathcal{E}_{\eps_i}(y^{\eps_i}, D_{\omega,h} \setminus D_{\omega,h/4}) \to 0, \ \   \  \Vert  y^{\eps_i} -   \RRR y_0^+\EEE\Vert_{H^1(D_{\omega,h})} \to 0.}$$
\end{corollary} 
\begin{proof}
Using Lemma \ref{lemma:comp-def}, Proposition \ref{prop:cell-form},   Proposition \ref{prop:FG}, and a diagonal argument we let $\lbrace y^{\eps_i}\rbrace_i \subset  H^2(D_{\omega,h};\R^d)$ be a sequence with 
$$
{\lim_{i\to \infty} \mathcal{E}_{\eps_i}(y^{\eps_i}, D_{\omega,h}) = \mathcal{F}(\omega,\tfrac{1}{2}) = K\mathcal{H}^{d-1}(\omega), \ \ \ \ \ \ \ \ \ \ \Vert     y^{\eps_i} -    \RRR y_0^+\EEE\Vert_{H^1(D_{\omega,h})} \to 0.}
$$
By Proposition \ref{prop:cell-form} we also  get $\liminf_{i\to \infty} \mathcal{E}_{\eps_i}(y^{\eps_i}, D_{\omega,h/4}) \ge K\mathcal{H}^{d-1}(\omega)$. This in turn implies $\mathcal{E}_{\eps_i}(y^{\eps_i}, D_{\omega,h} \setminus D_{\omega,h/4}) \to 0$. 
\end{proof}

\begin{remark}\label{rem: layer energy}
{\normalfont
Using Lemma \ref{lemma: intefacefind}  one can also show the following generalization, whose proof is deferred to \cite{davoli.friedrich}: for \RRR $h$ as in Corollary \ref{cor: layer energy}, and for \EEE each sequence $\lbrace \tau_i \rbrace_i$ satisfying
$$\tau_i \le  h/4,  \ \ \ \ \ \tau_i\eta_{ \eps_i,d}/\eps_i \to \infty, \ \ \ \ \ \  \tau_i/ \eps_i^{1+\frac{1}{d}} \to \infty,  $$
 there exists  $\lbrace y^{\eps_i}\rbrace_i \subset H^2(D_{\omega,h};\mathbb{R}^d)$ such that
$$\lim_{i\to \infty}  \mathcal{E}_{\eps_i}   (y^{\eps_i}, D_{\omega,h})=  K\mathcal{H}^{d-1}(\omega),  \ \ \    \mathcal{E}_{\eps_i}  (y^{\eps_i}, D_{\omega,h} \setminus D_{\omega,\tau_i}) \to 0, \ \   \  \tau_i^{-1}\Vert \nabla y^{\eps_i} - \nabla \RRR y_0^+\EEE \Vert^2_{L^2(D_{\omega,4\tau_i})} \to 0.$$
This means that  the  energy is concentrated in a $\tau_i$-neighborhood around $\omega \times \lbrace 0 \rbrace$.}
\end{remark}

 To conclude the proof of Proposition \ref{prop:FG}, we need to show Lemma \ref{lemma: intefacefind}.

\begin{proof}[Proof of Lemma \ref{lemma: intefacefind}]
 We proceed in two steps. We first define the cylindrical sets and then find suitable isometries such that the functions defined in \eqref{eq: vi} satisfy \eqref{eq: to confirm}. For brevity, let $\Omega_i =  \mathcal{Q} \cup D_{\mathcal{Q}',h_i}$. Let $\lbrace \tau_i\rbrace_i$ be a \RRR sequence satisfying \eqref{eq: cond on h}. \EEE

\noindent\emph{Step I: Definition of the cylindrical sets.} In view of \eqref{eq: cond on h}, we can choose $\lbrace \lambda_i\rbrace_i \subset(0,1/4)$ such that 
\begin{align}\label{eq: cond on h2}
\lambda_i \to 0, \ \ \ \ \ \   \tau_i  \eta_{\eps_i} \lambda_i /\eps_i \to \infty.
\end{align}
We use  Proposition \ref{lemma: phases} for $ y^{i} \in H^2(\Omega_i;\R^d)$ to find a corresponding set $T_i$ with properties \eqref{eq: propertiesT}. Recall that $T_i$ corresponds to the $A$-phase regions and $\Omega_i \setminus T_i$ to the $B$-phase regions of the function $y^{i} $. Let 
\begin{align}\label{eq: good layer5}
\mathcal{T}^i_A& = \big\{ t \in (-h_i,h_i): \ \mathcal{H}^{d-1}( (\mathcal{Q}' \times \lbrace t\rbrace) \cap T_i) \ge (1- \lambda_i) \mathcal{H}^{d-1}(\mathcal{Q}')\big\}, \notag \\ 
\mathcal{T}^i_B& = \big\{ t \in (-h_i,h_i): \ \mathcal{H}^{d-1}( (\mathcal{Q}' \times \lbrace t\rbrace) \setminus T_i) \ge  (1- \lambda_i) \mathcal{H}^{d-1}(\mathcal{Q}') \big\}.
\end{align}
Note that for $i$ sufficiently large (i.e., $\lambda_i$ small) the relative isoperimetric inequality on $\mathcal{Q}' \times \lbrace t\rbrace$ in dimension $d-1$, cf.\ \cite[Theorem 2, Section 5.6.2]{EvansGariepy92}, shows that, if $\mathcal{H}^{d-2}((\mathcal{Q}' \times \lbrace t \rbrace) \cap \partial^* T_i ) \le \lambda_i\mathcal{H}^{d-1}(\mathcal{Q}')$, then $t \in \mathcal{T}^i_A \cup \mathcal{T}^i_B$. Indeed, by the relative isoperimetric inequality we get
$$\min\big\{ \mathcal{H}^{d-1} ((\mathcal{Q}'\times \{t\})\cap   T_i   ),\, \mathcal{H}^{d-1}((\mathcal{Q}'\times \{t\})\setminus   T_i   )\big\}  \le C(\lambda_i\mathcal{H}^{d-1}(\mathcal{Q}'))^{\frac{d-1}{d-2}}.$$
 (The theorem in the reference above is stated and proved in a ball, but the argument only relies on Poincar\'e inequalities, and thus easily extends to bounded Lipschitz domains. \RRR For $d=2$, the right hand side has to be interpreted as $0$.) \EEE  By \eqref{eq: cond on h2}, this  in turn implies
$$\min\big\{\mathcal{H}^{d-1}((\mathcal{Q}'\times \{t\})\cap T_i ),\,\mathcal{H}^{d-1}((\mathcal{Q}'\times \{t\})\setminus T_i )\big\}\leq \lambda_i\mathcal{H}^{d-1}(\mathcal{Q}')$$
for $i$ large enough, and gives the claim. Thus, by \eqref{eq: propertiesT}{\rm (iv)}   and $E_{\eps_i,\eta_{\ep_i}}(y^{i},\Omega_i) \le M$   we obtain
\begin{align}\label{eq: good layer3}
\mathcal{H}^{1}( (-h_i,h_i) \setminus (\mathcal{T}^i_A \cup \mathcal{T}^i_B)  ) \le cM  \eps_i \eta_{\eps_i}^{-1} (\lambda_i\mathcal{H}^{d-1}(\mathcal{Q}'))^{-1}.
\end{align}
By the coarea formula, cf.\ \eqref{eq:co-f2}, we get for $\mathcal{H}^1$-a.e.\ $t_A \in \mathcal{T}^i_A$, $t_B \in \mathcal{T}^i_B$
\begin{align*}
\mathcal{H}^{d-1}\big(\partial^*T_i \cap (\mathcal{Q}' \times (t_A,t_B))\big) &\ge  \int_{\partial^*T_i \cap (\mathcal{Q}' \times (t_A,t_B))} |\langle \nu_{T_i},   {\rm e}_d   \rangle| \, d\mathcal{H}^{d-1} \\
&= \int_{\pi_d} \mathcal{H}^{0}\big( (  z + (t_A,t_B) {\rm e}_d) \cap \partial^* T_i \cap (\mathcal{Q}'\times (t_A,t_B))\big)  \, d\mathcal{H}^{d-1}(z),
\end{align*}
where $\pi_d = \R^{d-1} \times \lbrace 0 \rbrace$. In view of \eqref{eq: good layer5}  and $\lambda_i \le \frac{1}{4}$,  it follows that 
\begin{align}\label{eq: LLL}
\int_{\pi_d} \mathcal{H}^{0}\big( ( z + (t_A,t_B) {\rm e}_d) \cap \partial^* T_i \cap   (\mathcal{Q}'\times (t_A,t_B))\big)  \, d\mathcal{H}^{d-1}( z )\geq \frac12 \mathcal{H}^{d-1}(\mathcal{Q}').
\end{align}
 Define the indicator function $\RRR\psi_i: \EEE (-h_i,h_i) \to \lbrace A,B\rbrace$ by $\psi_i(t) = A$ if $\sup \lbrace t' \le t, t' \in \mathcal{T}^i_A \cup \mathcal{T}^i_B \rbrace \in \overline{\mathcal{T}^i_A}$ and $\psi_i(t) = B$ else. By using \eqref{eq: LLL} it is elementary to see that  $\psi_i$ jumps at most
\begin{align}\label{eq: good layer1}
 N_i  := \lfloor 2\mathcal{H}^{d-1}(\partial^* T_i \cap  \Omega_i)/\mathcal{H}^{d-1}(\mathcal{Q}')\rfloor + 1 
\end{align}
times. Using \eqref{eq: propertiesT}{\rm (ii)},  we note that 
$$ N_i \le 2cM  \, (\mathcal{H}^{d-1}(\mathcal{Q}'))^{-1} + 1.$$
Hence,  we have that  $N := \sup_i N_i < +\infty$ only depends on the constant $c$ from Proposition \ref{lemma: phases},  $M$,  and $\mathcal{Q}'$.   We now show that   
\begin{align}\label{eq: good layer2}
\mathcal{H}^1(\mathcal{T}^i_A)\ge h_i/2\ \ \ \ \text{and}  \ \ \  \mathcal{H}^1(\mathcal{T}^i_B)\ge h_i/2
\end{align}
 for all  $i$ sufficiently large. In fact, we observe that $\lim_{i \to \infty} h_i^{-1}\mathcal{H}^{1}( (-h_i,h_i) \setminus (\mathcal{T}^i_A \cup \mathcal{T}^i_B)  ) = 0$ by \eqref{eq: cond on h2}, $\tau_i \le h_i$,  and  \eqref{eq: good layer3}. Using assumption \eqref{eq: converg assu}, choose $i_0 \in \N$ such that $\mathcal{H}^{1}( (-h_i,h_i) \setminus (\mathcal{T}^i_A \cup \mathcal{T}^i_B)  ) \le \frac{h_i}{4}$ and
\begin{align}\label{eq: converg assu-new}
\Vert  \nabla   y^{i}  - \nabla  \RRR y_0^+\EEE\Vert^2_{L^2(D_{\mathcal{Q}',h_i})} \le  \frac{(1-\beta)^2\kappa^2}{16}  \mathcal{H}^{d-1} (\mathcal{Q}')h_i
\end{align}
for all $i \ge i_0$, where $\beta$ is given in Proposition \ref{lemma: phases} and $\kappa = |B-A|$. Now assume by contradiction that, e.g., $\mathcal{H}^1(\mathcal{T}^i_B)< h_i/2$ for some $i \ge i_0$. We then get   $\mathcal{H}^1(\mathcal{T}^i_A) \ge  \frac{5}{4}h_i$. By \eqref{eq: good layer5} and $\lambda_i \le \frac{1}{4}$ this implies 
 $$\mathcal{L}^d(T_i \cap \lbrace x_d <0 \rbrace) \ge \frac{1}{4}h_i (1-\lambda_i)  \mathcal{H}^{d-1} (\mathcal{Q}') \ge \frac{3}{16}h_i\,   \mathcal{H}^{d-1}  (\mathcal{Q}').$$
By  \eqref{eq: propertiesT}{\rm (i)} and \eqref{eq: conti-schweizer-k-y0}  we also have $\Vert  \nabla y^{i}  - \nabla  \RRR y_0^+\EEE\Vert^2_{L^2(T_i \cap \lbrace x_d <0 \rbrace)} \ge (1-\beta)^2\kappa^2\mathcal{L}^d(T_i \cap \lbrace x_d <0 \rbrace)$. The previous two estimates contradict \eqref{eq: converg assu-new}.

\RRR  In the following, we denote by $s^i_1 < s^i_2 \ldots < \ldots < s^i_{\RRR m_i \RRR }$ the jump points of the function $\psi_i$ defined below \eqref{eq: LLL}. \RRR Note that $m_i \le N_i$, see \eqref{eq: good layer1}. \RRR  Let $\mathcal{J}_i =\lbrace 0\le j \le \RRR m_i:  \RRR \,  (s^i_{j},s^i_{j+1}) \cap \mathcal{T}^i_A = \emptyset \rbrace$, where we set $s^i_0 = -h_i$ and $s^i_{m_i+1} = h_i$. Note that for $j \in \mathcal{J}_i$ there holds $(s^i_{j-1},s^i_j) \cap \mathcal{T}^i_B = \emptyset$.  Recalling \eqref{eq: good layer1}, up to passing to a subsequence, we can assume that $\mathcal{J}_i$ and $m_i$ are independent of $i$, and we denote them by $\mathcal{J}$ and $m$, respectively, for simplicity. In view of \eqref{eq: good layer2}, possibly after a rotation by $\pi$ corresponding to the transformation $y \mapsto -y(-x)$, we can suppose that there are two indices $k_1,k_2 \in \mathcal{J}$ and a constant $\bar{c}> 0 $  independent of $i \in \N$ such that   
 \begin{align}\label{eq: middle-length}
s_{k_1}^i - s_{k_1-1}^i \ge \bar{c}\,h_i, \ \ \ \  \ \ \ s_{k_2+1}^i - s_{k_2}^i \ge \bar{c}\,h_i,
\end{align}
where we recall that $(s_{k_1-1}^i, s_{k_1}^i ) \cap \mathcal{T}^i_B = \emptyset$ and $(s_{k_2}^i, s_{k_2+1}^i ) \cap \mathcal{T}^i_A = \emptyset$. We now choose the largest index $k_1 \le j \le k_2$, $j \in \mathcal{J}$, such that   $\liminf_{i \to \infty} \tau_i^{-1}(s^i_{j} - s^i_{j-1}) >0$, and define
$$\mu' := \liminf_{i \to \infty} \tau_i^{-1}(s^i_{j} - s^i_{j-1})  \in(0,+\infty].$$  
Note that such an index exists by \eqref{eq: middle-length} and the fact that $\tau_i \le h_i$. We set $\mu = \min\lbrace \mu', \bar{c}\rbrace$, and we observe that
\begin{align}\label{eq: intervals}
\lim_{i\to\infty}\tau_i^{-1} \mathcal{H}^1\big((s^i_{j} -\mu \tau_i, s^i_{j})  \cap  \mathcal{T}^i_B\big) =0, \ \ \ \ \lim_{i\to\infty}\tau_i^{-1} \mathcal{H}^1\big((s^i_{j}, s^i_{j} + \mu \tau_i)  \cap  \mathcal{T}^i_A\big) =0.
\end{align} 
In fact, the first property follows directly from the definition. For the second part, it suffices to recall that $\liminf_{i \to \infty} \tau_i^{-1}(s^i_{l} - s^i_{l-1}) =0$ for all $l \in \mathcal{J}$ with $j < l \le k_2$.  We define $\alpha_i = s^i_j$ and  $D_i := \mathcal{Q}' \times  (\alpha_i - \mu\tau_i, \alpha_i + \mu\tau_i) = \alpha_i {\rm e}_d + D_{\mathcal{Q}',\mu\tau_i}$. By  \eqref{eq: cond on h2}--\eqref{eq: good layer3} and \eqref{eq: intervals} we get \EEE
\begin{align}\label{eq: good layer6}
\tau_i^{-1}\mathcal{L}^d( & \lbrace x \in D_i: \  x_d  \leq \EEE \alpha_i  \rbrace\setminus T_i)\notag \\
 &\le   \tau_i^{-1} \big( \RRR \mathcal{H}^1\big((s^i_{j} -\mu \tau_i, s^i_{j})  \cap  \mathcal{T}^i_B\big) + \EEE \mathcal{H}^{1}( (-h_i,h_i) \setminus (\mathcal{T}^i_A \cup \mathcal{T}^i_B)  )  + \RRR \mu \EEE \tau_i  \, \lambda_i\big)\mathcal{H}^{d-1}(\mathcal{Q}') \notag\\
& \le  \RRR \tau_i^{-1}   \mathcal{H}^1\big((s^i_{j} -\mu \tau_i, s^i_{j})  \cap  \mathcal{T}^i_B\big) \RRR  \mathcal{H}^{d-1}(\mathcal{Q}')  + \EEE C\eps_i (\eta_{\eps_i}  \tau_i   \lambda_i)^{-1} + \RRR \mu \EEE \lambda_i  \mathcal{H}^{d-1}(\mathcal{Q}')  \to 0
\end{align}
as $i \to \infty$.  In a similar fashion, we find 
\begin{align}\label{eq: good layer6-2}
 \tau_i^{-1}  \mathcal{L}^d( \lbrace x \in D_i: \  x_d  \geq \EEE \alpha_i  \rbrace\cap T_i) \to 0.
\end{align}

\noindent\emph{Step II: Construction of the maps  $v^i$.} Since $\Omega_i$ contains a cube and $\lbrace \tau_i\rbrace_i$ is a bounded sequence, we observe that $D_i$ can be covered with a bounded number of cubes contained in $\Omega_i$. Suppose first that there exists one  cube $ \tilde{\mathcal{Q}}_i  \subset \R^d$ with $D_i \subset\subset  \tilde{\mathcal{Q}}_i  \subset \Omega_i$. We apply Theorem \ref{thm:rigiditythm}   (for $p= \frac{d+1}{d} <  \frac{d}{d-1}$),    Remark \ref{rem: afterth}{\rm (ii)},  and Remark \ref{rem: setT} to find $R_i \in SO(d)$ such that 
\begin{align}\label{eq: good layer7}
\Vert\nabla  y^{i}  -R_iA \Vert_{L^{ p }(  D_i   \cap T_i)} + \Vert\nabla  y^{i}  -R_iB \Vert_{L^{ p }(  D_i   \setminus T_i)}\leq C\ep_i   +   C(\eps_i/\eta_{\eps_i})+C(\ep_i^{\frac12}/\eta_{\ep_i}^{\frac32}) \le C\eps_i
\end{align}
where in the last step we used $\eta_{\ep_i}\geq \ep_i^{-\frac13}$, and where $C$ depends on  $M$. This estimate remains true if more than one cube is needed to cover $D_i$, \RRR cf.\ Remark \ref{rem: afterth}(ii). \EEE We now   prove   \eqref{eq: to confirm}  for  isometries $I_i$ whose derivative is given by $R^T_i$. 

Let $E_i = D_i \cap \lbrace |\nabla  y^{i} | \le L \rbrace$, where $L\ge \sqrt{d}$ is sufficiently large such that  $\dist(F,SO(d)\lbrace A,B\rbrace) \ge  |F-R M|/2 $ for all $F \in \M^{d \times d}$ with $|F| \ge L$,  $R \in SO(d)$, and $ M  \in \lbrace A,B\rbrace$. 
 Using H4.\  we observe that  
\begin{align}\label{eq: bad set}
\Vert\nabla  y^{i}  -R_iA \Vert^2_{L^2(  D_i\setminus E_i  ) } + \Vert\nabla  y^{i}  -R_iB \Vert^2_{L^2(  D_i\setminus E_i  )} \le C\int_{  D_i  } W(\nabla  y^{i} )  \, dx  \le C\eps^2_i,
\end{align}
where $C$ depends on $c_1$. We now consider the behavior on $E_i$. First, we calculate by \eqref{eq: good layer7} and the definition of $E_i$
\begin{align*}
\int_{\lbrace x \in E_i: \ x_d  \le\EEE  \alpha_i\rbrace} |R_i^T \nabla  y^{i}  -A|^2 \, dx &\le \int_{E_i\cap T_i} | \nabla  y^{i}  -R_i|^2 \, dx +  \int_{\lbrace x \in E_i: \ x_d \le\EEE \alpha_i\rbrace \setminus T_i} |\nabla  y^{i}  -R_i|^2 \, dx\\
& \le  (2L)^{2-p}\int_{  D_i  \cap T_i} | \nabla  y^{i}  -R_i|^{ p } \, dx +  (2L)^2\mathcal{L}^d({\lbrace x \in D_i: \ x_d \le\EEE \alpha_i\rbrace \setminus T_i})\\
& \le C  \eps_i^p  +  (2L)^2\mathcal{L}^d({\lbrace x \in D_i: \ x_d \le\EEE \alpha_i\rbrace \setminus T_i}).
\end{align*}
The fact that  $\eps_i^p/\tau_i \to 0$   (recall $p = \frac{d+1}{d}$  and see \eqref{eq: cond on h})    and \eqref{eq: good layer6} now imply
\begin{align}
\tau_i^{-1}\int_{\lbrace x \in E_i: \ x_d \le\EEE \alpha_i\rbrace} |R_i^T \nabla  y^{i}  -A|^2 \, dx \to 0. 
\end{align}
In a similar fashion, using \eqref{eq: good layer6-2} instead of  \eqref{eq: good layer6}, we obtain 
\begin{align}\label{eq: good set2}
\tau_i^{-1}\int_{\lbrace x \in E_i: \ x_d  \ge   \alpha_i\rbrace} |R_i^T \nabla  y^{i}   -B|^2 \, dx \to 0. 
\end{align}
Combining \eqref{eq: bad set}--\eqref{eq: good set2} and using  that  $\eps_i^2 /\tau_i \to 0$, we conclude the proof of \eqref{eq: to confirm}, when we define $v^i$ as in \eqref{eq: vi} with an isometry with derivative $R^T_i$. 
\end{proof}

\subsection{Local construction of recovery sequences}\label{sec: proof2}

This subsection is devoted to the proof of Proposition \ref{lemma: local1}.   Let $h>0$ and $\omega \subset \R^{d-1}$ open, bounded with Lipschitz boundary. Our goal is to  suitably modify   functions with optimal-profile energy, see \eqref{eq: our-k1}, such that they have the structure given in \eqref{eq:local1-3}. As a preparation, we introduce the following notion for $y\in H^2(D_{\omega,h};\mathbb{R}^d)$, where $D_{\omega,h}$ denotes the cylindrical set defined in \eqref{eq:def-dlh}: for $\ep,\eta>0$  and for $0<\tau\le h/4$   we define the \emph{$(\ep,\eta)$-closeness  of $y$ to the limiting   map $y_0^+$} by
\begin{equation}
\label{eq:eta}
\delta_{\ep,\eta}(y;\omega,h,\tau): = E_{\eps,\eta}(y,  D_{\omega,h}\setminus D_{\omega,\tau})  +     (\mathcal{L}^d(D_{\omega,4\tau}))^{-1}  \Vert \nabla y -  \nabla y^+_0  \Vert^2_{L^2(D_{\omega,4\tau})},
\end{equation}
where $y_0^+$ is the map defined in \eqref{eq: conti-schweizer-k-y0}.

In the following, we will use that by Corollary \ref{cor: layer energy}, for given $\omega \subset \R^{d-1}$, \RRR $0<h<\bar{h}$, \EEE and  $\lbrace \eps_i \rbrace_i$ converging to zero, there exists a  sequence   $\lbrace y^{\eps_i}\rbrace_i \subset H^2(D_{\omega,h};\mathbb{R}^d)$ of deformations attaining the optimal-profile energy $K$ (see \eqref{eq: our-k1}) such that  
\begin{align*}
\delta_{\ep_i,\eta_{\eps_i,d}}(y^{\eps_i};\omega,h, h/4) \to 0 \ \ \ \ \text{ as } i \to \infty.
\end{align*}
More generally, the existence of such a sequence is still guaranteed when  $\tau = h/4$ is replaced by a sequence $\lbrace \tau_i \rbrace_i$ with $\tau_i\eta_{\eps_i, d}/\eps_i \to \infty$ and $\tau_i/ \eps_i^{1+1/d}   \to \infty$, see Remark \ref{rem: layer energy}.  Although we only need the case $\tau = h/4$ and $\eta= \eta_{\eps_i,d}$ for the proof of Proposition \ref{lemma: local1},  we formulate the definition of $(\ep,\eta)$-closeness and some statements below in a more general way as this will be needed in the companion paper \cite{davoli.friedrich}.  


 The proof strategy for Proposition \ref{lemma: local1} is as follows:  relying on the quantitative rigidity estimate in Theorem \ref{thm:rigiditythm}, we first show in Proposition \ref{lemma: optimal profile} and Corollary \ref{cor: Besov} that it is possible to find  two  $(d-1)$-dimensional  slices on which the energy of $y$  and the $L^{p}$-distance  of $\nabla y$  from suitable rotations of $\nabla y_0^+$ can be quantified  in terms of $\delta_{\ep,\eta}(y;\omega,h, \tau)$. In Lemma \ref{lemma: transition1}, for each of the  slices identified above  we construct a transition to a rigid movement,  where the energy can  again be quantified in terms of  $\delta_{\ep,\eta}(y;\omega,h, \tau)$.  The latter construction relies on suitable extensions and gluing of functions. These auxiliary estimates are given in Lemma \ref{lemma:ext1} and Lemma  \ref{lemma: smoothen}.

We emphasize that the main novelties of our approach are the estimates in Proposition \ref{lemma: optimal profile} and Corollary \ref{cor: Besov} which build upon the rigidity estimates of Section \ref{sec: rigidity estimate}. For the construction of the transitions we follow closely the argumentation in \cite[Section 5]{conti.schweizer2}. However, we will work out the main points of the arguments in order to (a) detail the adaptions necessary with respect to \cite{conti.schweizer, conti.schweizer2} due to anisotropic surface energies and to (b) provide a self-contained presentation.

  
 We begin by collecting the main properties of $(d-1)$-dimensional slices.  Recall $p_d$ in \eqref{eq:def-pd}, $\kappa = |A-B|$, and $c_1$ in H4.

\begin{proposition}[Properties  of $(d-1)$-dimensional slices]\label{lemma: optimal profile}
Let $d\in \mathbb{N}$, $d\geq 2$. Let $h>0$,  $0 < \tau \le h/4$,   and let  $\omega, \hat{\omega} \subset \mathbb{R}^{d-1}$ be  Lipschitz domains such that $\omega \subset \subset \hat{\omega}$. Then there  exist $\ep_0=\eps_0(\omega,\hat{\omega},h,\kappa,c_1,\tau) \in (0,1)$ and $C=C(\omega,\hat{\omega},h,\kappa,c_1)>0$  with the following properties: \\  
For all $0<\eps \le \eps_0$, for every  $\eta$ with $\eta_{\ep,d} \le \eta \le \frac{1}{\eps}$,   and for each  $y\in H^2(D_{\hat{\omega},h};\mathbb{R}^d)$ with $\delta_{\ep,\eta}(y;\hat{\omega},h,\tau) \le   (\kappa/64)^2  $ we can find two rotations $R^+,R^- \in SO(d)$  and two constants   $s^+ \in (\tau,2\tau)$, $s^- \in (-2\tau, - \tau)$ such that  
\begin{align*}
 {\rm (i)}&\ \    \int_{\Gamma^+} | \nabla y -  R^+ A|^{p} \, d\mathcal{H}^{d-1} +  \int_{\Gamma^-} | \nabla y -  R^- B|^{p} \, d\mathcal{H}^{d-1} \leq  \frac{C}{\tau}(\delta_{\ep,\eta}(y;\hat{\omega},h, \tau))^{ {p}/  2} \,\eps^{p} \  \, \text{for all }\, 1\le p \le p_d, \notag\\
 {\rm (ii)} & \ \ \Vert \nabla y - A \Vert^2_{L^2(  s^+ {\rm e}_d  + D_{\omega,\eps^2} )}   +   \Vert \nabla y -B \Vert^2_{L^2(  s^- {\rm e}_d  + D_{\omega,\eps^2} )} \le C\eps^2\delta_{\ep,\eta}(y;\hat{\omega},h, \tau),\notag\\
{\rm (iii)}& \ \   \eps^2 \int_{\Gamma^+  \cup \Gamma^-} | \nabla^2 y|^2 \, d\mathcal{H}^{d-1} + \eta^2  \int_{\Gamma^+  \cup \Gamma^-} (|\nabla^2 y|^2-|\partial^2_{dd}y|^2) \, d\mathcal{H}^{d-1} \le  \frac{C}{\tau}\delta_{\ep,\eta}(y;\hat{\omega},h, \tau)  \notag,\\
{\rm (iv)}&  \ \   E_{\ep,\eta}\big(y,  s^+ {\rm e}_d  + D_{\omega,\eps^2} \big) + E_{\ep,\eta}\big(y,  s^- {\rm e}_d  + D_{\omega,\eps^2}  \big) \le  \frac{C\eps^2}{\tau}   \delta_{\ep,\eta}(y;\hat{\omega},h, \tau), \\
{\rm (v)} & \ \  |R^+ -  {\rm Id}|^2  + |R^- - {\rm Id}|^2 \le C\delta_{\ep,\eta}(y;\hat{\omega},h, \tau), 
\end{align*}
where we set  $\Gamma^\pm=\omega \times \lbrace s^\pm \rbrace$  for brevity.  
\end{proposition}

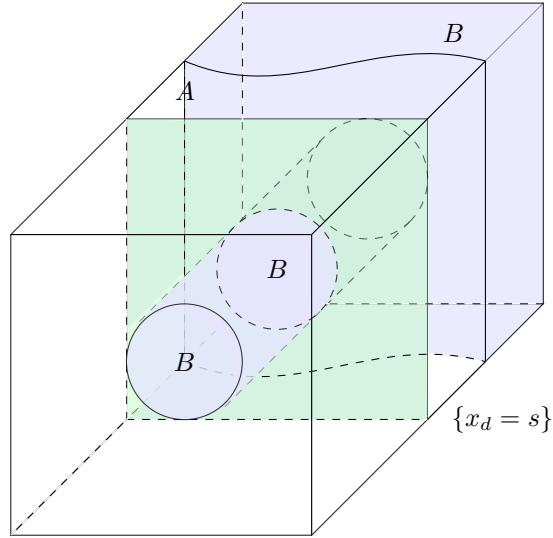
\begin{figure}[h]
\centering
\begin{tikzpicture}
\coordinate (H) at (-2,-2,-2);
\coordinate (A) at (-2,-2,2);
\coordinate (B) at (2,-2,2);
\coordinate (C) at (2,-2,-2);
\coordinate (D) at (-2, 2,-2);
\coordinate (E) at (-2, 2,2);
\coordinate (F) at (2, 2, 2);
\coordinate (G) at (2, 2,-2);

\coordinate (P) at (0,0,0); 
\coordinate (I) at (-2, 2,-6);
\coordinate (L) at (2, 2,-6);
\coordinate (M) at (2, -2,-6);
\coordinate (N) at (-2, -2,-6);
\coordinate (O) at (-2, 2,-4);
\coordinate (Q) at (2, 2,-4);
\coordinate (R) at (-2, -2,-4);
\coordinate (S) at (2, -2,-4);
\coordinate (T) at (1.1,-2,-4.3);

\draw[fill=blue!10, opacity=0.8] (N) -- (M) -- (C) -- (H) -- cycle;
\draw[white, fill=white, opacity=1] (S)--(T)--(C)--cycle; 
\draw[fill=blue!10, opacity=0.8] (I) -- (L) -- (M) -- (N) -- cycle;
\draw[fill=blue!10, opacity=0.8] (O) -- (I) -- (N) -- (R) -- cycle;
\draw[opacity=0.2] (D) -- (I) -- (L) -- (G) -- cycle;
\draw[opacity=0.2] (G) -- (L) -- (M) -- (C) -- cycle;
\begin{scope}[scale=0.8, shift={(-1,-1,-1)}]
\path (1,0,0);
\pgfgetlastxy{\cylxx}{\cylxy}
\path (0,1,0);
\pgfgetlastxy{\cylyx}{\cylyy}
\path (0,0,1);
\pgfgetlastxy{\cylzx}{\cylzy}
\pgfmathsetmacro{\cylt}{(\cylzy * \cylyx - \cylzx * \cylyy)/ (\cylzy * \cylxx - \cylzx * \cylxy)}
\pgfmathsetmacro{\ang}{atan(\cylt)}
\pgfmathsetmacro{\ct}{1/sqrt(1 + (\cylt)^2)}
\pgfmathsetmacro{\st}{\cylt * \ct}
\fill[blue!10, opacity=0.8] (\ct,\st,0) -- ++(0,0,-8) arc[start angle=\ang,delta angle=180,radius=1, dashed] -- ++(0,0,8) arc[start angle=\ang+180,delta angle=-180,radius=1];

\draw[dashed] (\ct,\st,0) -- ++(0,0,-8);
\draw[dashed] (-\ct,-\st,0) -- ++(0,0,-8);
\draw[dashed] (\ct,\st,-8) arc[start angle=\ang,delta angle=180,radius=1];
\draw[dashed] (\ct,\st,-8) arc[start angle=\ang,delta angle=-180,radius=1];
\end{scope}

 \fill[blue!10, opacity=0.8] (D)--(-2,1.2,-4)--(-2,-0.8,-4)--(-2,-0.8,-2)--cycle;
\draw[fill=green!30, opacity=0.4] (H) -- (C) -- (G) -- (D) -- cycle;
\draw[opacity=0.8] (H) -- (A) -- (E) -- (D) -- cycle;
\draw[opacity=0.8] (H) -- (A) -- (B) -- (C) -- cycle;
\draw[opacity=0.2] (D) -- (E) -- (F) -- (G) -- cycle;
\draw[opacity=0.2] (C) -- (B) -- (F) -- (G) -- cycle;

\draw[white!, dashed, line width=0.2mm] (C)--(H);
\draw[white!, dashed, line width=0.2mm] (D)--(H);
\draw[white!, dashed, line width=0.2mm] (A)--(H);
\draw[white!, dashed, line width=0.2mm] (N)--(M);
\draw[white!, dashed, line width=0.2mm] (I)--(N);
\draw[white!, dashed, line width=0.2mm] (H)--(N);
\draw (-2, 2,-4) to [out=-23,in=165] (2, 2,-4);
\draw[ dashed] (-2, -2,-4) to [out=-23,in=165] (2, -2,-4);
\draw (Q)--(S);
\draw[ dashed] (O)--(R);

\begin{scope}[scale=0.8]
\path (1,0,0);
\pgfgetlastxy{\cylxx}{\cylxy}
\path (0,1,0);
\pgfgetlastxy{\cylyx}{\cylyy}
\path (0,0,1);
\pgfgetlastxy{\cylzx}{\cylzy}
\pgfmathsetmacro{\cylt}{(\cylzy * \cylyx - \cylzx * \cylyy)/ (\cylzy * \cylxx - \cylzx * \cylxy)}
\pgfmathsetmacro{\ang}{atan(\cylt)}
\pgfmathsetmacro{\ct}{1/sqrt(1 + (\cylt)^2)}
\pgfmathsetmacro{\st}{\cylt * \ct}
\fill[blue!10, opacity=0.8] (\ct,\st,0) -- ++(0,0,-2) arc[start angle=\ang,delta angle=180,radius=1, dashed] -- ++(0,0,4) arc[start angle=\ang+180,delta angle=-180,radius=1];
\end{scope}
\draw[fill=blue!10, opacity=0.8, dashed] (0,0,-2) circle[radius=0.8];
\draw[fill=blue!10, opacity=0.8] (0,0,1.2) circle[radius=0.77];
\draw[opacity=0.2] (A) -- (B) -- (F) -- (E) -- cycle;
\draw (D)--(O);
\draw (F)--(Q);
\draw (F)--(E);
\draw (F)--(B);
\draw (C)--(S);
\node at (1,1.8,-5.5) {$B$};
\node at (0,0,-2) {$B$};
\node at (0,0,1.2) {$B$};
\node at (-1.8,1.8,-3.5) {$A$};
\node at (3,-2,-2) {\{$x_d=s$\}};
\end{tikzpicture}
\caption{The slice $\{x_d=s\}$ (in green) is contained in the  $A$-phase  region (in white) except for a small set lying in the $B$-phase region (in blue).}
\label{fig:small-set}
\end{figure}
 
\begin{proof}
 For notational convenience, we write $\delta(y)$ instead of $\delta_{\ep,\eta}(y;\hat{\omega},h,\tau)$.   Without restriction, we only select the rotation $R^+ \in SO(d)$ and  the constant $s^+ \in (\tau,2\tau)$, and establish the corresponding properties {\rm (i)}--{\rm (v)}. The selection of $R^-, s^-$ is analogous.  For convenience of the reader we subdivide the proof into three steps. We first discuss some consequences of the two-well rigidity estimate  (Step I) and identify a proportion of  $(d-1)$-dimensional slices which are contained in the $A$-phase region except for a small set (Step II), see Figure \ref{fig:small-set}. Finally, in Step III we select $s^+$ by means of a De Giorgi argument and show properties {\rm (i)}--{\rm (v)}.

\noindent\emph{Step I: Consequences of the  two-well  rigidity estimate.}  Recall the definition of $r(p,d)$ in \eqref{eq: rigidity-new} and note that $r(p,d)\geq 1/2$ for $d \ge 2$ and  $p \le 2$.   We first observe that  for all $1 \le p \le p_d$   there holds 
\begin{align}
\label{eq:prel-comp}
\Big(\frac{\ep}{\eta} \delta(y) + \frac{\sqrt{\ep}}{\eta^{3/2}} \delta(y) \Big)^{r(p,d)} & \le  C  (\delta(y))^{1/ 2} \Big(\frac{\ep}{\eta}  + \frac{\sqrt{\ep}}{\eta^{3/2}}\Big)^{r(p,d)} \leq  2   C  (\delta(y))^{1/ 2}   \Big(\frac{\sqrt{\ep}}{\eta^{3/2}_{\eps,d}}\Big)^{r(p,d)}\leq  2   C (\delta(y))^{1/ 2 } \eps^{\frac{r(p,d)}{r(p_d,d)}}\notag\\
&  \le  C(\delta(y))^{1/2} \ep. 
\end{align}
 for $C=C(\kappa)$. Here,   in the first inequality we used   $\delta(y) \le    (\kappa/64)^2 $ and  $r(p,d)\geq 1/2$. In the second one, we used $\eta_{\ep,d} \le \eta \le \frac{1}{\eps}$. In the third, we exploited that the definition of $\eta_{\ep,d}$ in \eqref{eq:eta-ep}--\eqref{eq:qdrd} implies   $\eps/\eta^3_{\ep,d} = \eps^{2/r(p_d,d)}$.  Finally, the fact that $r(p,d)$ is decreasing in $p$ implies the fourth inequality.


 For notational convenience, we define $F_{\omega,\tau} = \omega \times (\tau,h)$ and  $F_{\hat{\omega},\tau} = \hat{\omega} \times (\tau,h)$. We now apply  Theorem \ref{thm:rigiditythm} for $p=  p_d$ on $F_{\hat{\omega},\tau}$. (Note that for $d=2$ we can apply version (a) since $F_{\hat{\omega},\tau}$ is a rectangle and thus simply connected.)   In view of \eqref{eq:eta}--\eqref{eq:prel-comp}, Proposition \ref{lemma: phases}{\rm (iv)}, and Remark \ref{rem: setT},  we  find a rotation  $R^+ \in SO(d)$ and a set of finite perimeter  $T \subset F_{\hat{\omega},\tau}$  such that 
\begin{align}\label{eq: rigidity for limsup}
{\rm (i)} & \ \ \Vert \nabla y - R^+ A \Vert_{L^{p}( F_{\omega,\tau} \cap  T)} +\Vert \nabla y - R^+ B \Vert_{L^{p}( F_{\omega,\tau} \setminus  T)}   \le  C_0  (\delta(y))^{1/  2} \eps \ \  \text{ for all}\,1 \le p\le p_d,  \notag \\
{\rm (ii)} & \ \ \int_{-h}^h\mathcal{H}^{d-2}\big( (\R^{d-1} \times \lbrace t \rbrace) \cap \partial^* T \cap  F_{\hat{\omega},\tau} \big)  \, dt \le  C  \delta(y)  \eps/\eta  \le  C_0  (\delta(y))^{\frac{d-2}{d-1}} \eps^{1/r_d}  
\end{align}
for $C_0=C_0(\omega,\hat{\omega},h,\kappa,c_1,  p_d  )>0$.  Here, {\rm (i)} follows first for $p=p_d$ and then for $p<p_d$ by H\"older's inequality.  Note that the constant $C_0$ is independent of $\tau \le h/4$ since all sets  $F_{\hat{\omega},\tau}$  are uniformly Lipschitz equivalent to $\omega \times (0,h)$,  see Remark \ref{rem: afterth}(iii).  In the second inequality of {\rm (ii)} we used $\delta(y)\le  (\kappa/64)^2$, $\eta \ge \eta_{\ep,d}$, and  the definitions of $\eta_{\ep,d}$ and $r_d$ in \eqref{eq:eta-ep}. (See \eqref{eq:prel-comp} for a similar computation.)

\noindent\emph{Step II: Slices of the phase region $T$.}   We now show that, for $\eps$ sufficiently small, at least for one-half of the $s \in (\tau,2\tau)$   the set $\omega\times \lbrace s \rbrace$ `mostly lies in $T$', see Figure \ref{fig:small-set}.  More precisely, there exist $\ep_0=\ep_0(\omega,\hat{\omega},h,\kappa,c_1,\tau) \in (0,1)$ and $\bar{C}=\bar{C}(\omega,\hat{\omega},h,\kappa,c_1)>0$ such that   for all $\ep \le \ep_0$ and at least for one-half of the $s \in (\tau,2\tau)$ there holds
 \begin{equation}
 \label{eq:first-case-set}
 \mathcal{H}^{d-1}\big((\omega\times \lbrace s \rbrace) \setminus T\big) \leq \bar{C} \tau^{-1}\delta(y)\ep^{p_d},
 \end{equation} 
 where $p_d$ is defined in \eqref{eq:def-pd}.  To see this, we first   observe by  \eqref{eq: rigidity for limsup}{\rm (ii)}  that there exists $S \subset (\tau,2\tau)$ with $\mathcal{L}^1(S)\ge \frac{3}{4}\tau$ such that for all $s \in S$    there holds
 $$\mathcal{H}^{d-2}\big(\partial^* T\cap (\omega\times \lbrace s \rbrace )\big)\le 4\tau^{-1}  C_0  (\delta(y))^{\frac{d-2}{d-1}} \eps^{1/r_d}.$$
 Using $(d-2)p_d < (d-1)/r_d$, see \eqref{eq:qdrd}--\eqref{eq:def-pd}, we find  some $\eps_0' \in (0,1)$ sufficiently small depending on $\tau$ and $d$ such that  
 $$\mathcal{H}^{d-2}\big(\partial^* T\cap (\omega\times \lbrace s \rbrace )\big)\le 4\tau^{-1} C_0  (\delta(y))^{\frac{d-2}{d-1}} \eps^{1/r_d} \le  C_0 (\tau^{-1}\delta(y)\eps^{p_d})^{\frac{d-2}{d-1}},$$ 
 for all $s \in S$ and $\eps \le \eps_0'$. By applying the relative isoperimetric inequality in dimension $d-1$, cf.\ \cite[Theorem 2, Section 5.6.2]{EvansGariepy92}, we deduce that all $s\in S$ satisfy
 \begin{equation}
 \label{eq:isop-cons}
 \min\big\{\mathcal{H}^{d-1}( (\omega\times \lbrace s \rbrace ) \cap T),\, \mathcal{H}^{d-1}( (\omega\times \lbrace s \rbrace) \setminus T )\big\}  \leq  \bar{C} \tau^{-1}\delta(y)\ep^{p_d} 
 \end{equation}
for some $\bar{C}=\bar{C}(\omega,\hat{\omega},h,\kappa,c_1)>0$. (Note that the theorem in the reference above is stated and proved in a ball, but that the argument only relies on Poincar\'e inequalities, and thus easily extends to bounded Lipschitz domains.) Define $\ep_0=\ep_0(\omega,\hat{\omega},h,\kappa,c_1,\tau)>0$ by
\begin{align}\label{eq: eps0 def}
\eps_0 = \min \Big\{\frac{\tau\mathcal{H}^{d-1}(\omega)}{16 (\bar{C} \kappa^{2} + C_0)}, \eps_0', \frac{h}{2},\tau \Big\},
\end{align}
where $C_0$ is the constant from \eqref{eq: rigidity for limsup}.

We now show that for at least  one-half of the $s \in (\tau,2\tau)$ property \eqref{eq:first-case-set} holds for the  constants $\bar{C}$ and $\eps_0$. Suppose  by contradiction   that the statement was wrong. In view of \eqref{eq:isop-cons}, we get that for at least one-fourth of the $s \in (\tau,2\tau)$ there holds 
  \begin{equation*}
 \mathcal{H}^{d-1}\big(T\cap (\omega\times \lbrace s\rbrace )\big)\leq  \bar{C} \tau^{-1}\delta(y)\ep^{p_d}. 
 \end{equation*}
 Then, setting   $G_{\omega,\tau}:=\omega\times (\tau,2\tau)$,   we obtain by \eqref{eq:eta},   H\"older's inequality, and \eqref{eq: rigidity for limsup}{\rm (i)} for $p=1$
\begin{align*}
 \frac{\tau}{4}   (\mathcal{H}^{d-1}(\omega)-\bar{C} \tau^{-1}\delta(y)\ep^{p_d}) |  A - R^+B| & \le \Vert A - R^+ B \Vert_{L^{1}(G_{\omega,\tau} \setminus T)} \\& \le \Vert \nabla y - A \Vert_{L^{1}(G_{\omega,\tau})}  +  \Vert \nabla y - R^+ B \Vert_{L^{1}(G_{\omega,\tau} \setminus T)} \\
& \le   8\tau \,\mathcal{H}^{d-1}(\omega)  \,  (\delta(y))^{1/2} + C_0 (\delta(y))^{1/ 2} \eps\\& \le \big(\tau\mathcal{H}^{d-1}(\omega)/8 +C_0\eps \big) |A-B|,
\end{align*}
where in the last inequality we used  that $\delta(y) \le   (\kappa/64)^2$, and the fact that $|A-B|=\kappa$. As $|A-R^+B | \ge |A-B|$, this implies
 $$\tau \mathcal{H}^{d-1}(\omega) /8 \le \bar{C}  \delta(y)\eps^{p_d}/4 + C_0\eps \le (\bar{C} \kappa^{2} + C_0)\eps. $$
 This, however, contradicts the choice of $\eps_0$ in \eqref{eq: eps0 def} and $\eps \le \eps_0$. Thus, \eqref{eq:first-case-set} holds.

\noindent\emph{Step III: Selection of $s^+$ and proof of the statement.} In view of \eqref{eq:eta}, \eqref{eq: rigidity for limsup}(i), and \eqref{eq:first-case-set}, we can use  a De Giorgi argument \RRR (see the explanation at the beginning of \RRR the proof of  \cite[Lemma 4.3]{conti.schweizer2} for the details of this technique) \EEE to select  $s^+ \in (\tau,2\tau)$ such that 
\begin{align}\label{eq: DeGiorgi}
{\rm (i)} &  \ \  \mathcal{H}^{d-1}(\Gamma^+ \setminus T ) \leq \bar{C} \tau^{-1}\delta(y)\ep^{p},\notag\\
{\rm (ii)} & \ \ \int_{\Gamma^+ \cap T} | \nabla y - R^+ A|^p \, d\mathcal{H}^{d-1}  + \int_{\Gamma^+ \setminus T} | \nabla y - R^+ B|^p \, d\mathcal{H}^{d-1} \le C \tau^{-1} (\delta(y))^{p/ 2} \eps^{p}, \notag \\
{\rm (iii)} & \ \ (\mathcal{L}^d(D_{\omega,4\tau}))^{-1}\int_{\Gamma^+} | \nabla y -  A|^2 \, d\mathcal{H}^{d-1} +  E_{\ep,\eta}\big(y,\Gamma^+ \big) \le C \tau^{-1}\delta(y), \notag\\
{\rm (iv)} & \ \ (\mathcal{L}^d(D_{\omega,4\tau}))^{-1}\Vert \nabla y - A \Vert^2_{L^2(s^+ {\rm e}_d + D_{\omega,\eps^2} )} +  E_{\ep,\eta}\big(y, s^+ {\rm e}_d + D_{\omega,\eps^2} \big) \le  C \tau^{-1}\eps^2\delta(y)
\end{align}
for all $1 \le p \le p_d$ and $\eps \le \eps_0$, where $\Gamma^+ := \omega \times \lbrace s^+ \rbrace$. Here, we have also used that $2\tau\le h/2$ and  $\eps^2 \le h/2$ (see \eqref{eq: eps0 def}) to guarantee that $s^+ {\rm e}_d + D_{\omega,\eps^2} \subset D_{\omega,h}$. We emphasize  that the constants $C$ and $\bar{C}$ are independent of $\tau$.

Properties {\rm (ii)}--{\rm (iv)} of the statement are immediate from \eqref{eq: DeGiorgi}{\rm (iii)}--{\rm (iv)} and definition \eqref{eq: nonlinear energy}. We now show item {\rm (i)} of the statement. First, in view of  \eqref{eq: DeGiorgi}{\rm (ii)}, the integral on    $\Gamma^+ \cap T$ is controlled and we therefore only need to consider the integral on $\Gamma^+ \setminus T$.  By \eqref{eq: DeGiorgi}{\rm (i)},{\rm (ii)} we get 
 \begin{align*}
  \int_{\Gamma^+ \setminus T}|\nabla y-R^+A|^{p}\,d\mathcal{H}^{d-1}&\leq  2^{p-1 }  \int_{\Gamma^+ \setminus T}|\nabla y-R^+B|^{p}\,d\mathcal{H}^{d-1}+ 2^{p-1}  |A-B|^{p}\,\mathcal{H}^{d-1}(\Gamma^+ \setminus T) \notag \\
 & \le  2^{p-1}  C(\delta(y))^{ p/ 2}\eps^{p}\tau^{-1}+  2^{p-1} \bar{C}\kappa^{p} \delta(y) \eps^{p} \tau^{-1}
 \end{align*}
for all $1 \le p \le p_d$. This along with $\delta(y) \le  C (\delta(y))^{p/2}$ and \eqref{eq: DeGiorgi}{\rm (ii)}   shows {\rm (i)}. We finally observe that {\rm (v)} holds. \RRR Indeed, by combining  property {\rm (i)} of the statement with \eqref{eq: DeGiorgi}{\rm (iii)}, \RRR and by H\"older's inequality, we deduce the estimate
\begin{align*}
|R^+-{\rm Id}|^p&=\frac{1}{\mathcal{H}^{d-1}(\omega)}\left(\int_{\Gamma^+}|R^+-{\rm Id}|^p\,d\mathcal{H}^{d-1}\right)\\
&\leq C \left(\int_{\Gamma^+}|R^+-\nabla y|^p\,d\mathcal{H}^{d-1}+\int_{\Gamma^+}|\nabla y-{\rm Id}|^p\,d\mathcal{H}^{d-1}\right)\leq \frac{C\ep^p(\delta(y))^{p/2}}{\tau}+C(\delta(y))^{p/2}
\end{align*} 
for every $1\leq p\leq p_d$. \EEE Property {\rm (v)} follows noting that $\eps_0 \le \tau$, see \eqref{eq: eps0 def}.
\end{proof}

 \begin{remark}[Amount of ``good" slices]
By the proof of Proposition \ref{lemma: optimal profile} it follows that the statement of the proposition holds for slices  in sets $S^+ \subset (\tau,2\tau)$ and $S^- \subset (-2\tau,-\tau)$, respectively, with $\mathcal{L}^1(S^\pm) \ge c\tau$,  where $0<c<1$ is a suitable ratio.
 \end{remark}

Based on Proposition \ref{lemma: optimal profile}{\rm (i)}, one can also derive an $H^{1/2}$-estimate on the $(d-1)$-dimensional slices.

 \begin{corollary}[$H^{1/2}$-estimate]\label{cor: Besov}
Consider the   setting  of Proposition \ref{lemma: optimal profile}. Then,  there exist  $t^+, t^- \in \R^d$ such that  
\begin{align}\label{eq: besov}
\|y(\cdot,s^+)-R^+A(\cdot,s^+)^T-t^+\|^2_{H^{1/2}(\omega)} + \|y(\cdot,s^-)-R^-B(\cdot,  s^- )^T-t^-\|^2_{H^{1/2}(\omega)}\leq C\ep^2      \delta_{\ep,\eta}(y;\hat{\omega},h, \tau)   
\end{align}
 for a constant  $C=C(\omega,\hat{\omega},h,\kappa,c_1,\tau)>0$.
 \end{corollary}

 \begin{proof}
We only provide the estimate on  $\omega \times \lbrace s^+ \rbrace$.  By   Proposition \ref{lemma: optimal profile}{\rm (i)} and by a $(d-1)$-dimensional Poincar\'e inequality  we find $t^+ \in \R^d$ such that  there holds
$$
\|y(\cdot,s^+)-R^+A(\cdot,s^+)^T-t^+\|_{W^{1,p_d}(\omega)}\leq \frac{C \ep  (\delta_{\ep,\eta}(y;\hat{\omega},h, \tau) )^{ 1/ 2}  }{\tau^{1/p_d}}.
$$
By the definition of $p_d$ (see \eqref{eq:def-pd}) and  by classical Besov embeddings (see, e.g., \cite[Theorem 14.32   and Remark 14.35]{leoni}   and  \cite[Theorem 7.1, Proposition 2.3]{lions.magenes}), we observe that the $H^{1/2}$-norm can be controlled in terms of the $W^{1,p_d}$-norm. This concludes the proof.  
 \end{proof}

\begin{remark}[The role of the quantitative rigidity estimate]\label{rem: comparison}
{\normalfont

The quantitative estimate in terms of $\eps$ provided by \eqref{eq: besov} is the fundamental ingredient to construct transitions to rigid movements in Lemma \ref{lemma: transition1} below. Its derivation relies on the rigidity result of Section \ref{sec: rigidity estimate}, a careful choice of $(d-1)$-dimensional slices, and an embedding into $H^{1/2}$. Let us emphasize that other quantitative two-well rigidity estimates (see Subsection \ref{sec: literature}) cannot be used in the proof of Proposition \ref{lemma: optimal profile}: applying  \eqref{eq: rig-mot2} would  lead to $\eps$ instead of $\eps^2$ on the right hand side of \eqref{eq: besov}. Although \eqref{eq: rig-mot3} would give a correct scaling in terms of $\eps$, no embedding into $H^{1/2}$  would be  possible since only the weak $L^1$-norm of the derivative is controlled. }
\end{remark}
\begin{remark}[The assumption $\eta \le \frac{1}{\eps}$]\label{rem: comparison-2}
{\normalfont

Let us mention that Proposition \ref{lemma: optimal profile} holds true also without the assumption $\eta \le \frac{1}{\eps}$. It will only be crucial in the construction of transitions, see Lemma \ref{lemma: transition1} below. However, we prefer to formulate the proposition with this slightly stronger assumption since $\eta_{\ep,d} \le \eta \le \frac{1}{\eps}$ is the interesting regime. In fact, if $\eta \ge \frac{1}{\eps}$, the proof is much simpler and no rigidity estimates are needed: property {\rm (i)} in Proposition \ref{lemma: optimal profile} can simply be derived by using   property {\rm (iii)} and a Poincar\'e inequality.

} 
\end{remark}

\begin{remark}[Sharpness of the argument]\label{rem: comparison-3}
{\normalfont

 Alternatively, an $H^{1/2}$-estimate on the traces could have been obtained without Besov embeddings by working directly with $p=2$ in Proposition \ref{lemma: optimal profile}. In this case, however, $\eta_{\eps,d}$ in \eqref{eq:eta-ep} has to be chosen larger.    We have preferred to perform the estimates for $p\leq p_d$ in order to obtain a sharpest definition of $\eta_{\eps,d}$ which leads to a sufficient $H^{1/2}$-control of the traces of the deformations. 
}
\end{remark}

The next  lemmas address suitable    $H^2$-extensions   of functions.  We point out that the proof arguments follow  closely \cite[Lemma 5.3, Lemma 5.4]{conti.schweizer2}.  We work out the main points of the proof for convenience of the reader.  In the following, we will frequently write $x' =(x_1,\ldots,x_{d-1})$ for brevity.


\begin{lemma}[Extension of functions defined on $(d-1)$-dimensional slices]
\label{lemma:ext1}
Let $\omega\subset \mathbb{R}^{d-1}$  open, bounded with Lipschitz boundary.  Let $ \eps,  \eta, \theta >0$. Let $u\in H^2(\omega;\mathbb{R}^d)$ be such that
\begin{equation}
\label{eq:need-later1}
\frac{1}{\ep^2}\|u\|_{H^{1/2}(\omega)}^2+\eta^2\|u\|_{H^2(\omega)}^2\leq \theta.
\end{equation}
Then, for any  $\tau>0$  there exists  $z\in H^2(\omega \times (0,\infty);\mathbb{R}^d)$   such that $z(x',0)=u(x')$ for all $x'\in \omega$, $z$ is constant  on $\omega \times (\tau,\infty)$ and
\begin{equation}
\label{eq:need-later2}
\frac{1}{\ep^2}\int_{\omega \times (0,\infty)}|\nabla z|^2\,dx+\eta^2\int_{\omega \times (0,\infty)}|\nabla^2 z|^2\,dx\leq C\theta
\end{equation}
for some constant $C=C(\omega,\tau)>0$.  In a similar fashion, an extension to $\omega \times (-\infty,0)$ can be constructed.  
\end{lemma}

\begin{proof}
 We extend $u$ from $\omega$ to a cube in $\R^{d-1}$ such that \eqref{eq:need-later1} still holds up to multiplying $\theta$ with a constant depending on $\omega$. (For an extension operator in $H^{1/2}$ we refer to \cite[Theorem 5.4]{dinezza.palatucci.valdinoci}.) Without loss of generality, after scaling we can assume that the cube is the unit cube in $\R^{d-1}$ and  $\tau=1$.  Periodically extending $u$ to $\mathbb{R}^{d-1}$ and using a Fourier  representation of $u$,  we have
$$u(x')=\sum_{\alpha\in 2\pi\mathbb{Z}^{d-1}}u_{\alpha}e^{ix'\cdot\alpha} ,$$
where  the Fourier coefficients $\{u_{\alpha}\}_{\alpha}$ satisfy 
\begin{equation}
\label{eq:need-not}
\sum_{\alpha\in 2\pi\mathbb{Z}^{d-1}}\Big(\frac{|\alpha|}{\ep^2}+\eta^2|\alpha|^4\Big)|u_{\alpha}|^2\leq  C\theta 
\end{equation}
 for a constant   $C$   only depending on $\omega$.  Let $\psi:[0,+\infty)\to \mathbb{R}$  be a smooth cut-off function  satisfying $0 \le \psi \le 1$,  $\psi(0)=1$, and $\psi(t)=0$ for $t\geq 1$. Setting
$$z(x',x_d):=u_0+\sum_{\alpha\in 2\pi\mathbb{Z}^{d-1},\,\alpha\neq 0}u_{\alpha}e^{ix'\cdot\alpha}\psi(|\alpha|x_d),$$
it is immediate to see that $z(x',0)=u(x')$ for all $x'\in \omega$ and that $z(x',x_d)=u_0$ for   $x_d>1$.  Using \eqref{eq:need-not} we calculate 
\begin{align*}
& \|\nabla z\|^2_{L^2(\omega \times (0,\infty))}  \le  \hspace{-0.15cm} \sum_{\substack{\alpha\in 2\pi\mathbb{Z}^{d-1}\\ \alpha\neq 0}}\int_0^{1/|\alpha|}|\alpha|^2|u_{\alpha}|^2\big(  \psi(|\alpha|x_d)  +\psi'(|\alpha|x_d)\big)^{ 2} \,dx_d\leq C\hspace{-0.15cm}\sum_{\alpha\in 2\pi\mathbb{Z}^{d-1}}|\alpha||u_{\alpha}|^2\leq C \eps^2  \theta,
\end{align*}
 where $C$ depends on $\Vert\psi  \Vert_\infty$ and $\Vert\psi'  \Vert_\infty$, and similarly 
 
$$\eta^2\|\nabla^2 z\|^2_{L^2(\omega \times (0,\infty))}  \leq C  \eta^2\sum_{\alpha\in 2\pi\mathbb{Z}^{d-1}}|\alpha|^3|u_{\alpha}|^2  \leq  C  \eta^2\sum_{\alpha\in 2\pi\mathbb{Z}^{d-1}}|\alpha|^4|u_{\alpha}|^2\leq  C  \theta,$$
 where $C$  depends additionally on $\Vert\psi''  \Vert_\infty$. This shows property \eqref{eq:need-later2}. 
\end{proof}

 For convenience,  in the next lemmas  we use the following notation: for  $D \subset \R^d$, $\eps,\eta>0$, and $u \in H^2(D;\R^d)$ we define    
   \begin{align}\label{eq: E*}
E^*_{\ep,\eta}(u,D):=\frac{1}{\ep^2}\int_D|\nabla u|^2\,dx+\ep^2\int_{D}|\nabla^2 u|^2\,dx+ \eta^2 \int_{D}\Big(  |\nabla^2 u|^2 - |\partial^2_{dd}u|^2 \Big)  \,dx.
\end{align} 

\begin{lemma}[$H^2$-extension]\label{lemma: smoothen}
Let $h,\tau >0$ with $\tau \le h/4$ and let $\omega\subset \mathbb{R}^{d-1}$   open, bounded with   Lipschitz boundary. Let $\eta, \eps,\theta >0$ with $\eps^2 \le \tau$ and $\eps \le \eta$.  Let $u \in   H^2  (D_{\omega,h};\R^d)$ and $0 < s < 2\tau$ be such that
\begin{equation}
\label{eq:need-later-new}
\frac{1}{\ep^2}\|u(\cdot,s)\|_{H^{1/2}(\omega)}^2+\eta^2\|u(\cdot,s)\|_{H^2(\omega)}^2 + E^*_{\ep,\eta}(u, \omega \times (s,s+\eps^2))  \leq \theta.
\end{equation}
Then there exists a map $v \in H^2(\omega \times (0,\infty);\R^d)$ such that  $v = u$ on $\omega \times (0,s)$, $v$ is constant on $\omega \times (s + \tau,\infty)$, and 
\begin{align*}
E^*_{\ep,\eta}(v, \omega \times (s,\infty)  ) \le C\theta
\end{align*}
for a constant   $C=C(\omega,\tau)>0$.  If \eqref{eq:need-later-new} holds for some $  -2\tau < s < 0$,  one can construct a map $v \in H^2(\omega \times (-\infty,0);\R^d)$ in a similar fashion.  
\end{lemma}

  \begin{proof}
   Let $\hat{z} \in H^2(\omega \times (0,\infty))$ be the function obtained by Lemma \ref{lemma:ext1} applied on $u(\cdot,s) \in H^2(\omega;\R^d)$ and define $z = \hat{z}(\cdot - s{\rm e}_d) \in H^2(\omega \times (s,\infty))$. We note that  $z$ is constant on  $\omega \times (s +  \tau,\infty)$, that $z(\cdot,s) = u(\cdot,s) $ on $\omega$, and that
\begin{align}\label{eq: crazy energy}
E_{\eps,\eta}^*(z,\omega \times (s,\infty)) = E_{\eps,\eta}^*(\hat{z},\omega \times (0,\infty)) \le C(1 + \eps^2\eta^{-2})\theta \le C\theta,
\end{align}
where in the last step we have used that $\eps \le \eta$.

 Let $\psi: \R \to \R$ be a smooth cut-off function with  $0 \le \psi \le 1$,  $\psi(t) = 0$ for $t \le  s $, and $\psi(t) = 1$ for $t \ge  s +  \eps^2$, and satisfying   $\Vert \psi \Vert_{L^\infty(\mathbb{R})}+ \eps^2\Vert \psi' \Vert_{L^\infty(\mathbb{R})} + \eps^4\Vert \psi'' \Vert_{L^{\infty}(\mathbb{R})} \le C$.   We define the map
  $$v(x',x_d) :=  z(x',x_d) \psi(x_d )   + u(x',x_d)(1- \psi(x_d))$$
 on $\omega \times (0,\infty)$.   Clearly, $v$ coincides with $z$ on   $\omega \times (s +  \eps^2,\infty)$ and with $u$  on $\omega \times (0,s)$. Since $\eps^2 \le \tau$ and $z$ is constant on  $\omega \times (s +  \tau,\infty)$, we get that also $v$ is constant on  $\omega \times (s +  \tau,\infty)$.  Additionally, there holds
   \begin{align*}
\nabla v(x',x_d) &= \nabla z(x',x_d) +   (\nabla u(x',x_d) - \nabla z(x',x_d))(1- \psi(x_d))  + \big( 0,  (z(x',x_d) - u(x',x_d))\,\psi'(x_d) \big),
\end{align*}
and
\begin{align*}
\partial^2_{ij} v(x',x_d) &= \partial^2_{ij} z(x',x_d)\psi(x_d) + \partial^2_{ij} u(x',x_d)(1- \psi(x_d)),\\
\partial^2_{id} v(x',x_d) &= \partial^2_{id}z(x',x_d) \psi(x_d)  + \partial^2_{id}u(x',x_d)(1- \psi(x_d))+  (\partial_i z(x',x_d) - \partial_i u(x',x_d))\,\psi'(x_d),\\
\partial^2_{dd} v(x',x_d) &= \partial^2_{dd}z(x',x_d) \psi(x_d )  + \partial^2_{dd} u(x',x_d)(1- \psi(x_d)) + 2(\partial_d z(x',x_d) - \partial_d u(x',x_d))\,\psi'(x_d) \\& \ \ \ + (z(x',x_d) -   u(x',x_d))\,\psi''(x_d),
\end{align*}
for $i,j\in \{1,\dots,d-1\}$.  We set $F_\omega^\eps:= \omega \times (s,s+\eps^2)$ for brevity.  Using the one-dimensional Poincar\'e inequality in the ${\rm e}_d$-direction for each $x'$, and exploiting the fact that   $u(\cdot,s) = z(\cdot,s)$ and $\partial_iu(\cdot,s) = \partial_i z(\cdot,s)$  for every $i=1,\dots,d-1$, we obtain
 \begin{align*}
\int_{F_\omega^\eps}|z - u|^2\, dx &\le C\eps^4 \int_{F_\omega^\eps} |\partial_d z - \partial_d u|^2 \, dx   \le   C\eps^6 \big(E^*_{\eps,\eta}(u,F_\omega^\eps) +   E^*_{\eps,\eta}(z,F_\omega^\eps)\big), \\
\int_{F_\omega^\eps}|\partial_i z - \partial_i u|^2\, dx &\le C\eps^4 \int_{F_\omega^\eps}   |\partial_{id} z - \partial_{id} u|^2   \, dx \le C
   \ep^4\eta^{-2}    \big(E^*_{\eps,\eta}(u,F_\omega^\eps) +  E^*_{\eps,\eta}(z,F_\omega^\eps)\big), 
\end{align*} 
for all $i\in \{1,\dots,d-1\}$. After some elementary,  but tedious  computations,  using \eqref{eq:need-later-new}--\eqref{eq: crazy energy} and $\eps \le \eta$, we get 
$$E^*_{\ep,\eta}(v, F_\omega^\eps ) \le CE^*_{\ep,\eta}(z, F_\omega^\eps ) + CE^*_{\ep,\eta}(u, F_\omega^\eps ) \le  C\theta.$$
The statement now follows from \eqref{eq: crazy energy} and the fact that $v = z$ on $\omega \times (s +  \eps^2,\infty)$.   
  \end{proof}

  The following lemma deals with the transition between a  $(d-1)$-dimensional slice  and a rigid movement.  Recall the  definitions  of the constants $c_1$ and $c_2$ in H4.\ and H5., respectively. 

\begin{lemma}[Transition to a rigid movement]\label{lemma: transition1}
 Let $d\in \mathbb{N}$, $d\geq 2$.  Let $h, \tau,\ep,\eta>0$ and $\omega \subset \subset \hat{\omega} \subset \R^{d-1}$ satisfy the assumptions of Proposition \ref{lemma: optimal profile}.  Assume that the elastic energy density $W$ satisfies assumptions H1.--H5.  Let $y\in H^2(D_{\hat{\omega},h};\mathbb{R}^d)$ with $\delta_{\ep,\eta}(y;\hat{\omega},h,\tau) \le  (\kappa/64)^{2} $  and let $R^+,R^- \in SO(d)$, $s^+ \in (\tau,2\tau)$, $s^- \in (-2\tau, - \tau)$ be the associated rotations and  constants  provided by Proposition \ref{lemma: optimal profile}.  
 Then there exist a map $y^A_+ \in H^2(\omega \times (0,\infty);\R^d )$ and a constant $b^A_+ \in \R^d$ such that 
\begin{align}\label{eq: trans-equ}
{\rm (i)} & \ \  y^A_+ =  y \ \ \text{on} \ \ \omega \times (0, s^+),   \ \ \ \ \
y^A_+(x) = R^+ Ax   + b^A_+ \ \  \text{ for all $x \in \omega \times (s^+ + \tau,\infty)$}, \notag\\
{\rm (ii)} & \ \  \Vert \nabla y^A_+ - R^+A \Vert^2_{L^2(\omega \times (s^+,\infty) )}  \le 
C \eps^2  \delta_{\ep,\eta}(y;\hat{\omega},h, \tau),   \notag\\
{\rm (iii)} & \ \   E_{\eps,\eta}(y^A_+, \omega \times (s^+,\infty))  \le 
C \delta_{\ep,\eta}(y;\hat{\omega},h, \tau)
\end{align}
 where $C=C(\omega,\hat{\omega},h,\tau,\kappa,c_1)>0$.  Analogously, there exist a map ${v}^B_- \in H^2(\omega \times (-\infty,0);  \R^d)$ and a constant $b^B_- \in \R^d$ for which  \eqref{eq: trans-equ} holds  with $B$, $s^-$, and $R^-$ in place of $A$, $s^+$, and $R^+$, respectively. 
\end{lemma}

 \begin{proof}
 We only show the construction of the map $y^A_+$, the proof strategy for proving the existence of $y^B_-$ is analogous.  As in the proof of Proposition \ref{lemma: optimal profile}, we write $\delta(y)$ instead of $\delta_{\ep,\eta}(y;\hat{\omega},h, \tau)$ for brevity. All constants in the following may depend on $\omega$, $\hat{\omega}$, $h$, $\tau$, $c_1$, and $\kappa$.   For convenience of the reader, we subdivide the proof into two steps.

 \noindent\emph{Step I: Transition to a constant function.} Using  Proposition \ref{lemma: optimal profile}{\rm (i)} for $p=1$ and Corollary \ref{cor: Besov} we   have 
\begin{align}\label{eq: 2needed1}
\|\nabla y(\cdot,s^+)-R^+A\|^2_{L^{1}(\omega)} + \|y(\cdot,s^+)-R^+A(\cdot,s^+)^T-t^+\|^2_{H^{1/2}(\omega)}  \le C\ep^2    \delta(y). 
\end{align}
 By  a $(d-1)$-dimensional Poincar\'e inequality on $\omega$, we   find $M^+  \in \R^{d\times d}$ such that by Proposition  \ref{lemma: optimal profile}{\rm (iii)}  there holds
\begin{align}\label{eq: 2needed2}
\|\nabla y(\cdot,s^+)-  M^+\|^2_{L^{2}(\omega)} \le C\sum_{i=1}^{d-1} \sum_{j=1}^{d}   \Vert \partial^2_{ij} y(\cdot,s^+)\Vert^2_{L^2(\omega)} \le C\eta^{-2}   \delta(y). 
\end{align}
Moreover, by Proposition  \ref{lemma: optimal profile}{\rm (ii)},{\rm (v)} we also find 
\begin{align}\label{eq: 2needed2-new}
\|\nabla y-R^+ A\|^2_{L^{2}(\omega \times (s^+,s^+ +\eps^2))}  & \le C\Vert \nabla  y  - A \Vert^2_{L^{2}(s^+ {\rm e}_d + D_{\omega,\eps^2})} + C\mathcal{L}^d(D_{\omega,\eps^2}) |R^+ A-A|^2 \notag\\ &  \le C\eps^2\delta(y).
\end{align}
Using \eqref{eq: 2needed1}--\eqref{eq: 2needed2} and  the triangle inequality, we derive $|R^+ A -  M^+|^2 \le C(\ep^2 + \eta^{-2})    \delta(y)$. Thus, defining $u(x):=y(x)-R^+Ax-t^+$ for $x \in D_{\omega,h}$ we obtain by \eqref{eq: 2needed1}--\eqref{eq: 2needed2}, Proposition  \ref{lemma: optimal profile}{\rm (iii)}, and the assumption $\eps \le 1/\eta$ 
$$\|u(\cdot,s^+)\|^2_{H^{1/2}(\omega)}  \le C\ep^2     \delta(y),  \ \ \ \ \ \ \ \  \|u(\cdot,s^+)\|^2_{H^{2}(\omega)}  \le C\eta^{-2}     \delta(y).  $$
Recalling \eqref{eq: E*}, by Proposition \ref{lemma: optimal profile}{\rm (iv)} and \eqref{eq: 2needed2-new}  we deduce   that
$$E_{\eps,\eta}^*(u,\omega \times (s^+,s^+ +\eps^2)) \le C\delta(y).$$
Observe that $\eps^2 \le  \eps \le  \eps_0 \le  \tau$, see \eqref{eq: eps0 def}, and  $\eps \le   \eta$,  see \eqref{eq:eta-ep}.  Applying Lemma \ref{lemma: smoothen} to the function $u\in H^2(D_{\omega,h};\R^d)$ and $s^+ \in (0,  2\tau)$,    we obtain a map  $v\in H^2(\omega \times (0,\infty);\mathbb{R}^{d})$  such that 
\begin{align}\label{eq: affine}
v = u \ \text{ on } \omega \times (0,  s^+), \ \ \ \ \ \ \ \ \ \  v  \  \text{ is constant on $\omega \times (s^+  + \tau,\infty)$},
\end{align}  
and 
 \begin{equation}
 \label{eq:need-later3}
E^*_{\ep,\eta}(v, F^+_\omega  ) \leq C   \delta(y),  
 \end{equation}
 where for brevity we set $F^+_\omega := \omega \times (s^+,\infty)$.

\noindent\emph{Step II: Transition to a rigid movement.}  We define $y^A_+(x) :=    v(x) + R^+A x  + t^+ $ for $x \in \omega \times (0,\infty)$. Property \eqref{eq: trans-equ}{\rm (i)} follows from \eqref{eq: affine} and the fact that $u(x):=y(x)-R^+Ax-t^+$ for $x \in D_{\omega,h}$. By \eqref{eq:need-later3} we obtain 
\begin{align*}
\|\nabla y^A_+ - R^+A\|_{L^2(F^+_\omega)}^2 =\|\nabla v\|_{L^2(F^+_\omega)}^2 \le C\eps^2E^*_{\ep,\eta}(v, F^+_\omega )\le C\ep^2   \delta(y).  
\end{align*}
This yields \eqref{eq: trans-equ}{\rm (ii)}. By H5.\ and  \eqref{eq:need-later3} we derive the estimate 
\begin{align}\label{eq: z+}
\int_{F^+_\omega}W(\nabla y^A_+)\,dx\leq C\int_{F^+_\omega}{\rm dist}\,^2(\nabla y^A_+, SO(d)\{A,B\})\,dx\leq C\|\nabla y^A_+ - R^+A\|_{L^2(F^+_\omega)}^2 \le C\ep^2   \delta(y) 
\end{align}
 on the nonlinear elastic energy. Similarly, as $\nabla^2 y^A_+ = \nabla^2 v$,  by \eqref{eq:need-later3}  we get 
\begin{align}\label{eq: z+2}
\ep^2\int_{F^+_\omega}|\nabla^2 y^A_+|^2\,dx+ \eta^2 \int_{F^+_\omega}\Big(  |\nabla^2 y^A_+|^2 - |\partial^2_{dd} y^A_+|^2 \Big)  \,dx \le E^*_{\ep,\eta}(v, F^+_\omega  ) \leq C    \delta(y).   
\end{align}
Combining \eqref{eq: z+}--\eqref{eq: z+2} gives   \eqref{eq: trans-equ}{\rm (iii)} and concludes the proof of the lemma. 
\end{proof}

We are now finally in the position to prove Proposition \ref{lemma: local1}.
\begin{proof}[Proof of Proposition \ref{lemma: local1}] 
We perform the construction for  $y_0^+$.  The strategy for $y_0^-$ is analogous.   Let $h>0$  \RRR sufficiently small \EEE and let $\omega \subset \R^{d-1}$ open, bounded with Lipschitz  boundary. Let $\rho >0$ and choose a Lipschitz domain $\hat{\omega} \supset \supset \omega$ with $\mathcal{H}^{d-1}(\hat{\omega} \setminus \omega)\le \rho$.  We first observe that by  Corollary \ref{cor: layer energy}  there exists a sequence $\lbrace y^\eps \rbrace_\eps \subset  H^2(D_{ \hat{\omega}  ,h}; \R^d)$ such that 
\begin{equation}
\label{eq:was-ok}
 \lim_{\ep\to 0}\mathcal{E}_{\ep}(y^{\ep}, D_{\hat{\omega}, h}) = K\mathcal{H}^{d-1}(\hat{\omega}), \ \ \  \lim_{\eps \to 0} \mathcal{E}_{\eps}(y^{\eps}, D_{\hat{\omega},h} \setminus D_{\hat{\omega},h/4}) =0,   \ \ \ \lim_{\eps \to 0}\Vert   y^\eps  -   y_0^+ \Vert_{H^1(D_{\hat{\omega},h})} =0. 
\end{equation}
In view of  Corollary \ref{cor: layer energy}, the existence of a sequence $\{y^{\ep_i}\}_i$ satisfying \eqref{eq:was-ok} is guaranteed for every $\{\ep_i\}_i$ with $\ep_i\to 0$. Hence, in what follows, for notational simplicity we directly work with the continuous parameter $\ep$.

Fix  $\tau=h/4$. Recalling the $(\ep,\eta)$-closeness in \eqref{eq:eta} and applying \eqref{eq:was-ok}, we find that  
\begin{align}\label{eq: right assumption}
\delta_{\ep,\eta_{\eps,d}}(y^{\eps};\hat{\omega},h, \tau) \to 0
\end{align}
as $\eps \to 0$. Without loss of generality, we can assume that $\ep<\ep_0$, where $\eps_0$ is the constant from Proposition \ref{lemma: optimal profile}. Moreover, by \eqref{eq: right assumption} we may assume that $\delta_{\ep,\eta_{\eps,d}}(y^{\eps};\hat{\omega},h, \tau) \le  (\kappa/64)^{2} $.  We also observe that $\eta_{\eps,d} \le 1/\eps$ by \eqref{eq:eta-ep}. We now apply Proposition \ref{lemma: optimal profile} on $\lbrace y^{\ep}\rbrace_\eps$.    Let  $\{R^+_{\ep}\}_{\ep}$, $\{R^-_{\ep}\}_{\ep}\subset SO(d)$ and let $\{s^+_{\ep}\}_{\ep}\subset (\tau,2\tau)$, $\{s^-_{\ep}\}_{\ep}\subset (-2\tau,-\tau)$ be the associated sequences of rotations and  constants.   Additionally,  let  $y^{A,\ep}_+$ and $y^{B,\ep}_-$  be the functions provided by Lemma \ref{lemma: transition1}, and associated to $y^{\ep}$.

  Let now  $w_{\ep}^+\in H^2(D_{\omega,h};\mathbb{R}^d)$  be defined as 
\begin{align*}
{w}^+_\eps(x)= \begin{cases} y_-^{B,\eps}(x)  & \text{if $x_d \le s^-_\eps$}, \\  y^\eps(x) & \text{if $s^-_\eps \le x_d \le s^+_\eps$}, \\  y_+^{A,\eps}(x)    & \text{if $x_d \ge s^+_\eps$}. \end{cases}
\end{align*} 
 Using $\tau = h/4$,  $|s^+_\eps|,|s^-_\eps| \le 2\tau = h/2$, \eqref{eq: conti-schweizer-k-y0}, and  \eqref{eq: trans-equ}{\rm (i)}, we get that $w^+_\eps= I^+_{1,\eps} \circ y_0^+$ on $\lbrace x_d \ge \frac{3}{4}h\rbrace$ and  $w^+_\eps= I^+_{2,\eps} \circ y_0^+$ on $\lbrace x_d \le -\frac{3}{4}h\rbrace$, where $I^+_{1,\eps}$ and $I^+_{2,\eps}$ are isometries. This shows \eqref{eq:local1-3}.

By  Proposition \ref{lemma: optimal profile}(v),  \eqref{eq: trans-equ}{\rm (ii)}, \eqref{eq:was-ok}, and \eqref{eq: right assumption} we also get $\lim_{\eps \to 0}\Vert \nabla w^+_\eps -  \nabla y_0^+ \Vert^2_{L^2(D_{{\omega},h})} =0$. Using $w^+_\eps \in H^2(D_{\omega,h};\R^d)$ and again \eqref{eq:was-ok}, this yields $w^+_\eps \to y_0^+$ in $H^1(D_{{\omega},h};\R^d)$, i.e., \eqref{eq:local1-1} holds. Combining \eqref{eq:local1-1} and \eqref{eq:local1-3} we also see that the isometries  $I^+_{1,\eps}$ and $I^+_{2,\eps}$ converge to the identity as $\eps \to 0$.

It remains to prove \eqref{eq:local1-3a}. The inequality $\liminf_{\eps \to 0}  \mathcal{E}_{\ep}(w^+_\eps, D_{{\omega},h})\ge K\mathcal{H}^{d-1}({\omega}) $ is clear by Proposition \ref{prop:cell-form}. We prove the reverse inequality. By  \eqref{eq: trans-equ}{\rm (iii)}  we obtain  
\begin{align*}
 \mathcal{E}_{\ep}(w^+_\eps, D_{\omega,h})  & \le \mathcal{E}_{\ep}(y^\eps,D_{\omega,2\tau}) + \mathcal{E}_\eps(  y_-^{B,\eps},\omega \times (-\infty,s^-_\eps)) + \mathcal{E}_\eps(y_+^{A,\eps},\omega \times (s^+_\eps,\infty))\\
 & \le  \mathcal{E}_{\ep}(y^\eps,D_{\hat{\omega},h})  + C  \delta_{\ep,\eta_{\eps,d}}(y^{\eps};\hat{\omega},h, \tau). 
\end{align*} 
Using \eqref{eq:was-ok}--\eqref{eq: right assumption} and $\mathcal{H}^{d-1}(\hat{\omega} \setminus \omega)\le \rho$ we find 
$$\limsup_{\eps \to 0}  \mathcal{E}_{\ep}(w^+_\eps, D_{{\omega},h}) \le K\mathcal{H}^{d-1}(\hat{\omega}) \le K\mathcal{H}^{d-1}(\omega) + K\rho.  $$
Property \eqref{eq:local1-3a}   then follows by letting $\rho \to 0$ and using  a diagonal argument. 
\end{proof}

\begin{remark}[Independence of the two constructions above and below the interface] 
Notice that the constructions of the maps $w^{\pm}_{\ep}$ in the sets $\{x_d\geq 3h/4\}$ and $\{x_d\leq -3h/4\}$, respectively, are independent from each other.
\end{remark}

\section*{Acknowledgements}
 We would like to thank {\sc Ben Schweizer} \RRR and {\sc Roberto Alicandro} \EEE for  interesting discussions\RRR, as well as the anonymous Referee for the insightful remarks. \EEE  The support by the Alexander von Humboldt Foundation is gratefully acknowledged. This work has been funded by \RRR the DFG through project  FR 4083/1-1, \EEE the Vienna Science and Technology Fund (WWTF) through Project MA14-009, as well as by the Austrian Science Fund (FWF) project\RRR s F65, V 662 N32, and I 4052 N32, and by BMBWF through the OeAD-WTZ project CZ04/2019. \EEE We are thankful to the Erwin Schr\"odinger Institute in Vienna, where part of this work has been developed during the workshop ``New trends in the variational modeling of failure phenomena".


\begin{thebibliography}{50}

\bibitem{alicandro.dalmaso.lazzaroni.palombaro}
{\sc R.~Alicandro, G.~Dal Maso, G.~Lazzaroni, M.~Palombaro}.
\emph{Derivation of a linearised elasticity model from singularly perturbed multiwell energy functionals}.
Arch.\ Ration.\ Mech.\ Anal.\  \textbf{230} (2018), 1--45.  

 


 \bibitem{ambrosio}
 {\sc L.~Ambrosio}.
{\em Metric space valued functions of bounded variation}.
Ann.\ Scuola Norm.\ Sup.\ Pisa Cl.\ Sci.\ {\bf 17} (1990), 439--478. 


 
 


\bibitem{Ambrosio-Fusco-Pallara:2000} 
{\sc L.~Ambrosio, N.~Fusco, D.~Pallara}.
\newblock {\em Functions of bounded variation and free discontinuity problems}. 
\newblock Oxford University Press, Oxford 2000. 


 
 \bibitem{baldo}
{\sc S.~Baldo}. 
{\em Minimal interface criterion for phase transitions in mixtures of Cahn-Hilliard fluids}. 
Ann.\ Inst.\ H.\ Poicar\'e Anal.\ Non Lin\'eare {\bf 7} (1990), 67--90.

 
\bibitem{ball-currie}
{\sc J.~Ball, J.~C.~Currie, P.~L.~Olver}.
{\em Null Lagrangians, weak continuity, and variational
problems of arbitrary order}.
J.\ Funct.\ Anal.\ {\bf 41} (1981), 135--174.


\bibitem{ball.james}
 {\sc J.~Ball, R.D.~James}.
 {\em Fine phase mixtures as minimizers of the energy}. 
 Arch.\ Ration.\ Mech.\ Anal.\ \textbf{100} (1987), 13--52.

 
  \bibitem{barroso.fonseca}
{\sc A.~C.~Barroso, I.~Fonseca}. 
{\em Anisotropic singular perturbations--the vectorial case}. 
Proc.\ Roy.\ Soc.\ Edinburgh Sect.\ A {\bf 124} (1994), 527--571.
 
 

 
 \bibitem{bhattacharya}
 {\sc K.~Bhattacharya}. 
 {\em Microstructure of martensite: Why it forms and how it gives rise to the shape-memory effect}.
 \newblock Oxford University Press, Oxford 2003.
 
  \bibitem{bhattacharya.kohn}
{\sc K.~Bhattacharya, R.~V.~Kohn}.
{\em Elastic energy minimization and the recoverable
strains of polycrystalline shape-memory materials}.
Arch.\ Ration.\ Mech.\ Anal.\ {\bf 139} (1997), 99--180.

\bibitem{bouchitte}
{\sc G.~Bouchitt\'e}.
{\em Singular perturbations of variational problems arising from a two-phase transition model}.
 Appl.\ Math.\ Optim.\ {\bf 21} (1990), 289--314.
 
    \bibitem{Braides:02}
{\sc A.~Braides}.
\newblock {\em $\Gamma$-convergence for Beginners}.
\newblock Oxford University Press, Oxford 2002.


  
 
 \bibitem{capella.otto2}
 {\sc A.~Capella-Kort, F.~Otto}.
{\em A quantitative rigidity result for the cubic-to-tetragonal phase transition in the geometrically linear theory with
 interfacial energy}.
{Proc.\ Roy.\ Soc.\ Edinburgh Sect.\ A}, \textbf{142} (2012), 273--327.
    
   \bibitem{Chambolle-Conti-Francfort:2014}
{\sc A.~Chambolle, S.~Conti, G.~Francfort}.
\newblock {\em Korn-Poincar\'e inequalities for functions
with a small jump set.} 
\newblock Indiana Univ.\ Math.\ J.\
\newblock {\bf 65} (2016), 1373--1399.  
 
 \bibitem{Chambolle-Giacomini-Ponsiglione:2007}
{\sc A.~Chambolle, A.~Giacomini, M.~Ponsiglione}. 
\newblock {\em Piecewise rigidity}.
\newblock {J.\ Funct.\ Anal.}
\newblock {\bf 244} (2007), 134--153.

 

  \bibitem{Chaudhuri} 
{\sc N.~Chaudhuri, S.~M\"uller}.
\newblock {\em Rigidity Estimate for Two Incompatible Wells}. 
\newblock   Calc.\ Var.\ Partial Differential
Equations
 \newblock {\bf 19} (2004), 379–-390.

\bibitem{cheng}
{\sc S.~Z.~D.~Cheng}.
{\em Chapter 2 - Thermodynamics and Kinetics of Phase Transitions}.
Phase Transitions in Polymers, Elsevier (2008), 17--59.


  




\bibitem{Chermisi-Conti}
{\sc M.~Chermisi, S.~Conti}.
\newblock {\em Multiwell rigidity in nonlinear elasticity}. 
\newblock {SIAM J.\ Math.\ Anal.\ }
\newblock {\bf 42} (2010), 1986--2012. 
 
\bibitem{conti.fonseca.leoni}
{\sc S.~Conti, I.~Fonseca, G.~Leoni}.
{\em A {$\Gamma$}-convergence result for the two-gradient theory of phase transitions}.
Comm.\ Pure Appl.\ Math.\ \textbf{55} (2002), 857--936.
 

\bibitem{conti.schweizer}
{\sc S.~Conti, B.~Schweizer}.
{\em Rigidity and gamma convergence for solid-solid phase transitions with $SO(2)$ invariance}.
Comm.\ Pure Appl.\ Math.\ \textbf{59} (2006), 830--868.
 
\bibitem{conti.schweizer2}
{\sc S.~Conti, B.~Schweizer}.
{\em A sharp-interface limit for a two-well problem in geometrically linear elasticity}.
Arch.\ Ration.\ Mech.\ Anal.\ \textbf{179} (2006), 413--452.


\bibitem{conti.schweizer3}
{\sc S.~Conti, B.~Schweizer}.
{\em Gamma convergence for phase transitions in impenetrable elastic materials. Multi scale problems and asymptotic analysis}, 105-118, GAKUTO Internat. Ser. Math. Sci. Appl., 24, Gakkotosho, Tokyo, 2006.







\bibitem{DalMaso:93}
{\sc G.~Dal Maso}.
\newblock {\em An introduction to $\Gamma$-convergence}.
\newblock Birkh{\"a}user, Boston $\cdot$ Basel $\cdot$ Berlin 1993. 



 

\bibitem{davoli.friedrich}
{\sc E.~Davoli, M.~Friedrich}.
{\em Linearization for solid-solid phase transitions}.
In preparation.

    \bibitem{De Lellis} 
{\sc C.~De Lellis, L.~J.~Szekelyhidi}.
\newblock {\em Simple proof of two well rigidity}. 
\newblock  C.\ R.\ Math.\ Acad.\ Sci.\ Paris
 \newblock {\bf 343} (2006), 367–-370.
 


  \bibitem{dinezza.palatucci.valdinoci}
 {\sc E.~Di Nezza, G.~Palatucci, E.~Valdinoci}.
  {\em Hitchhiker's guide to the fractional {S}obolev spaces}.
  Bull.\ Sci.\ Math.\ {\bf 136} (2012), 521--573.          




\bibitem{dolzmann.muller}
{\sc G.~Dolzmann, S.~M\"uller}.
{\em Microstructures with finite surface energy: the two-well problem}.
Arch.\ Ration.\ Mech.\ Anal.\ \textbf{132} (1995), 101--141.

\bibitem{EvansGariepy92}
{\sc L.~C~Evans, R.~F.~Gariepy}. 
\newblock {\em Measure theory and fine properties of
functions}.
\newblock CRC Press, Boca Raton $\cdot$ London $\cdot$ New York $\cdot$ Washington,
D.C. 1992.


\bibitem{fonseca.leoni}
{\sc I.~Fonseca, G.~Leoni}.
{\em Modern Methods in the Calculus of Variations: $L^p$ Spaces.}
Springer Monographs in Mathematics. Springer, New York, 2007.

\bibitem{fonseca.tartar}
{\sc I.~Fonseca, L.~Tartar}. 
{\em The gradient theory of phase transitions for systems with two potential wells}. 
Proc. Roy. Soc. Edinburgh Sect. A {\bf 111} (1989), 89--102.

 


 

  \bibitem{Friedrich-ARMA}
{\sc M.~Friedrich}.
\newblock {\em A derivation of linearized Griffith energies from nonlinear models}. 
\newblock Arch.\ Ration.\ Mech.\ Anal.\  
\newblock {\bf 225} (2017), 425--467.

 
 

 \bibitem{Friedrich:15-3} 
{\sc M.~Friedrich}.
\newblock {\em A Korn-type inequality in  SBD for functions with small jump sets}. 
\newblock   Math.\ Models Methods Appl.\ Sci.\
 \newblock {\bf 27} (2017), 2461--2484.
 
 \bibitem{Friedrich:15-4}
{\sc M.~Friedrich}.
\newblock {\em A piecewise Korn inequality in SBD and applications to embedding and density results}. 
\newblock  SIAM J.\ Math.\ Anal.\
\newblock {\bf 50} (2018), 3842–-3918.

 

\bibitem{friedrich-kruzik}
{\sc M.~Friedrich, M. Kru\v{z}\'ik}.
{\em On the passage from nonlinear to linearized viscoelasticity}. 
SIAM J.\ Math.\ Anal.\ {\bf 50} (2018), 4426--4456.
 

\bibitem{FrieseckeJamesMueller:02}
{\sc G.~Friesecke, R.~D.~James, S.~M{\"u}ller}.
\newblock {\em A theorem on geometric rigidity and the derivation of nonlinear plate theory from three-dimensional elasticity}. 
\newblock {Comm.\ Pure Appl.\ Math.}
\newblock {\bf 55} (2002), 1461--1506. 
 
 
 


 \bibitem{gurtin}
{\sc M.~E.~Gurtin}. 
\newblock {\em Some Results and Conjectures in the Gradient Theory of Phase Transitions}.
\newblock In: Antman, Ericksen, Kinderlehrer, M\"uller (eds) Metastability and Incompletely Posed Problems.
\newblock  The IMA Volumes in Mathematics and Its Applications, vol 3., Springer, 1987, 135--146. 



 
 
\bibitem{Jerrard-Lorent}
{\sc R.~L.~Jerrard,  A.~Lorent}.
\newblock {\em On multiwell Liouville theorems
in higher dimension}. 
\newblock {Adv.\ Calc.\ Var.}
\newblock {\bf 6} (2013), 247--298. 


\bibitem{kytavsev-lauteri-ruland-luckhaus}
{\sc G.~Kitavtsev, G.~Lauteri, S.~Luckhaus, A.~R\"uland}.
{\em A compactness and structure result for a discrete multi-well problem with $SO(n)$ symmetry in arbitrary dimension}.
\RRR {Arch.\ Ration.\ Mech.\ Anal.} {\bf 232} (2019), 531--555.\EEE

\bibitem{kytavsev-ruland-luckhaus}
{\sc G.~Kitavtsev, S.~Luckhaus, A.~R\"uland}.
{\em Surface energies arising in microscopic modeling of martensitic transformations}.
Math.\ Models Methods Appl.\ Sci.\ \textbf{25} (2015), 647--683. 

\bibitem{kytavsev-ruland-luckhaus2}
{\sc G.~Kitavtsev, S.~Luckhaus, A.~R\"uland}.
{\em Surface energies emerging in a microscopic, two-dimensional two-well problem}. 
 Proc. Roy. Soc. Edinburgh Sect. A \textbf{147} (2017), 1041--1089.
 
\bibitem{kohn.muller}
{\sc R.~V.~Kohn, S.~M\"uller}. 
{\em Surface energy and microstructure in coherent phase transitions}. 
{Comm.\ Pure Appl.\ Math.} \textbf{47} (1994), 405--435. 




\bibitem{kohn.muller2}
{\sc R.~V.~Kohn, S.~M\"uller}. 
{\em Branching of twins near an austenite-twinned-martensite interface}. 
{Philosophical Magazine A} \textbf{66} (1992), 697--715. 







\bibitem{kohn.sternberg}
{\sc R.~V.~Kohn, P.~Sternberg}.
{\em Local minimisers and singular perturbations}. 
Proc. Roy. Soc. Edinburgh Sect. A {\bf 111} (1989), 69--84.

\bibitem{lauteri.luckhaus}
{\sc G.~Lauteri, S.~Luckhaus}.
{\em Geometric rigidity estimates for incompatible fields in dimension $\geq 3$}.
Preprint {arXiv:1703.03288v1}.

\bibitem{leoni}
{\sc G.~Leoni}.
{\em A first course in Sobolev spaces.}
 Graduate Studies in Mathematics, 105. American Mathematical Society, Providence, RI, 2009.

  \bibitem{lions.magenes}
 {\sc J.-L.~ Lions, E.~Magenes}. 
 {\em Non-homogeneous boundary value problems and applications. Vol. I.} 
 Die Grundlehren der mathematischen Wissenschaften, Band 181. Springer-Verlag, New York-Heidelberg, 1972.
 

\bibitem{Lorent}
{\sc A.~Lorent}.
\newblock {\em A two well Liouville Theorem}. 
\newblock  {ESAIM Control Optim.\ Calc.\ Var.}
\newblock {\bf 11} (2005),  310--356. 

  \bibitem{Matos} 
{\sc J.~Matos}.
\newblock {\em Young measures and the absence of fine microstructures in a class of phase transitions}. 
\newblock   European J.\ Appl.\ Math.\
 \newblock {\bf 3} (1992),  31--54.
 
 
 \bibitem{mielke.roubicek}
{\sc A.~Mielke, T.~Roub\'i\v{c}ek}.
{\em Rate-independent elastoplasticity at finite strains and its numerical
approximation}
 Math.\ Models Methods Appl.\ Sci.\  {\bf 26} (2016), 2203--2236.



 
  
\bibitem{modica}
{\sc L.~Modica}.
{\em The gradient theory of phase transitions and the minimal interface criterion}.
 Arch.\ Ration.\ Mech.\ Anal.\ {\bf 98} (1987), 123--142.

\bibitem{modica.mortola}
{\sc L.~Modica, S.~Mortola}.
{\em Un esempio di $\Gamma$-convergenza}. 
Boll.\ Un.\ Mat.\ Ital.\ B {\bf 14} (1977), 285--299.

\bibitem{muller}
{\sc S.~M\"uller}. 
{\em Variational models for microstructure and phase transitions}. In: Calculus of variations and geometric evolution problems (F. Bethuel et al., eds.), Springer Lecture Notes in Math. 1713. Springer, Berlin, 85--210, 1999.


\bibitem{Mueller-Scardia-Zeppieri}
{\sc S.~M\"uller, L.~Scardia, C.~I.~Zeppieri}. 
\newblock {\em Geometric rigidity for incompatible fields and an application to strain-gradient plasticity}.
\newblock {Indiana Univ.\ Math.\ J.}
\newblock {\bf 63} (2014), 1365--1396. 

\bibitem{owen.sternberg} 
{\sc N.~C.~Owen, P.~Sternberg}. 
{\em Nonconvex variational problems with anisotropic perturbations}.
Nonlinear Anal.\ {\bf 16} (1991), 705--719.

 

\bibitem{podioGuidugli}
{\sc P.~Podio-Guidugli}.
{\em Contact interactions, stress, and material symmetry for nonsimple elastic materials}.
Theor.\ Appl.\ Mech.\ {\bf 28-29} (2002), 261--276.

 
\bibitem{ruland-ARMA}
{\sc A.~R\"uland}.
{\em The Cubic-to-Orthorhombic Phase Transition:
Rigidity and Non-Rigidity Properties in the
Linear Theory of Elasticity}.
Arch.\ Ration.\ Mech.\ Anal.\ {\bf 221} (2016), 23-106.


%
%

\bibitem{Schmidt:08}
{\sc B.~Schmidt}. 
\newblock {\em Linear $\Gamma$-limits of multiwell energies in nonlinear elasticity theory}. 
\newblock {Continuum Mech.\ Thermodyn.}
\newblock {\bf 20} (2008), 375--396. 

 
\bibitem{sternberg} 
{\sc P.~Sternberg}.
{\em The effect of a singular perturbation on nonconvex variational problems}. 
Arch.\ Ration.\ Mech.\ Anal.\ {\bf 101} (1988), 209--260.

\bibitem{sternberg2} 
{\sc P.~Sternberg}.
{\em Vector-valued local minimizers of nonconvex variational problems}. 
Rocky Mountain J.\ Math.\ {\bf 21} (1991), 799--807.


\bibitem{toupin1}
{\sc R.~A.~Toupin}.
{\em Elastic materials with couple stresses}.
Arch.\ Ration.\ Mech.\ Anal.\ {\bf 11} (1962), 385--414.


\bibitem{toupin2}
{\sc R.~A.~Toupin}.
{\em Theory of elasticity with couple stress}.
Arch.\ Ration.\ Mech.\ Anal.\ {\bf 17} (1964), 85--112.

 
 
 \bibitem{Zwicknagl}
{\sc B.~Zwicknagl}. 
\newblock {\em Microstructures in low-hysteresis shape memory alloys: scaling regimes and optimal
needle shapes}.
\newblock Arch.\ Ration.\ Mech.\ Anal.\
\newblock {\bf 213} (2014), 355–-421. 



\end{thebibliography}
\end{document}